\documentclass[a4paper, 10pt]{amsart}

\usepackage{amscd}
\usepackage{amsfonts}
\usepackage{amsmath}
\usepackage{mathrsfs}
\usepackage{amssymb}

\usepackage{amsthm}
\usepackage{mathtools}
\usepackage{MnSymbol}
\usepackage{stmaryrd}
\usepackage{multirow,bigdelim}
\usepackage{eucal}
\usepackage{url}
\usepackage{bm}
\usepackage[dvipdfmx]{hyperref}
\usepackage[dvipdfmx]{graphicx}
\usepackage[all]{xy}
\usepackage{color}
\usepackage{xcolor}
\usepackage{tikz}
\usetikzlibrary{arrows}
\usetikzlibrary{cd}
\usetikzlibrary{decorations.pathmorphing}
\usetikzlibrary{patterns}
\DeclareMathOperator{\Ad}{Ad}

\DeclareMathOperator{\ev}{ev}

\DeclareMathOperator{\Gr}{Gr}
\DeclareMathOperator{\Hom}{Hom}
\DeclareMathOperator{\id}{id}
\DeclareMathOperator{\Image}{Im}

\DeclareMathOperator{\Ker}{Ker}

\DeclareMathOperator{\Proj}{Proj}

\DeclareMathOperator{\std}{std}
\DeclareMathOperator{\Spec}{Spec}
\DeclareMathOperator{\Sym}{Sym}

\DeclareMathOperator{\type}{type}

\DeclareMathOperator{\Der}{Der}

\DeclareMathOperator{\PGL}{PGL}

\DeclareMathOperator{\SL}{SL}
\DeclareMathOperator{\SO}{SO}
\DeclareMathOperator{\SU}{SU}

\newcommand{\bA}{\mathbb{A}}
\newcommand{\bC}{\mathbb{C}}

\newcommand{\bP}{\mathbb{P}}

\newcommand{\bR}{\mathbb{R}}
\newcommand{\bZ}{\mathbb{Z}}

\newcommand{\cB}{\mathcal{B}}

\newcommand{\cI}{\mathcal{I}}
\newcommand{\cJ}{\mathcal{J}}

\newcommand{\cP}{\mathcal{P}}
\newcommand{\cO}{\mathcal{O}}

\newcommand{\fg}{\mathfrak{g}}

\newcommand{\fh}{\mathfrak{h}}
\newcommand{\fk}{\mathfrak{k}}
\newcommand{\fl}{\mathfrak{l}}
\newcommand{\fm}{\mathfrak{m}}

\newcommand{\fp}{\mathfrak{p}}
\newcommand{\fq}{\mathfrak{q}}
\newcommand{\fsl}{\mathfrak{sl}}

\newcommand{\ft}{\mathfrak{t}}

\newcommand{\fz}{\mathfrak{z}}

\def\bfg{\boldsymbol{\mathfrak{g}}}

\def\bfl{\boldsymbol{\mathfrak{l}}}
\def\bfq{\boldsymbol{\mathfrak{q}}}

\def\bff{\boldsymbol{f}}
\def\bfi{\boldsymbol{i}}
\def\bfl{\boldsymbol{l}}
\def\bfp{\boldsymbol{p}}
\def\bfr{\boldsymbol{r}}
\def\bfx{\boldsymbol{x}}
\def\bfA{\boldsymbol{A}}
\def\bfB{\boldsymbol{B}}
\def\bfC{\boldsymbol{C}}
\def\bfG{\boldsymbol{G}}
\def\bfH{\boldsymbol{H}}
\def\bfK{\boldsymbol{K}}

\def\bfQ{\boldsymbol{Q}}

\def\bfX{\boldsymbol{X}}
\def\bfY{\boldsymbol{Y}}
\def\bfZ{\boldsymbol{Z}}

\def\bfcB{\boldsymbol{\mathcal{B}}}

\def\bfepsilon{\boldsymbol{\epsilon}}
\def\bfDelta{\boldsymbol{\Delta}}

\theoremstyle{plain}
\newtheorem{thm}{Theorem}[section]
\newtheorem{cor}[thm]{Corollary}
\newtheorem{lem}[thm]{Lemma}
\newtheorem{prop}[thm]{Proposition}
\newtheorem{var}[thm]{Variant}
\theoremstyle{definition}

\newtheorem{cond}[thm]{Condition}
\newtheorem{cons}[thm]{Construction}
\newtheorem{defn}[thm]{Definition}
\newtheorem{ex}[thm]{Example}

\newtheorem{rem}[thm]{Remark}
\begin{document}
	\title[Algebraic approach]{Algebraic approach to contraction families}
	\author{Takuma Hayashi}
	\address{Osaka Central Advanced Mathematical Institute, 3-3-138 Sugimoto, Sumiyoshi-ku Osaka 558-8585, Japan}
	\email{takuma.hayashi.forwork@gmail.com}
	\date{}
	\begin{abstract}
		In this paper, we give a purely algebraic approach to the contraction group scheme predicted by Bernstein--Higson--Subag and constructed by Barbasch--Higson--Subag. We also compare quotient schemes of contraction group schemes with other related schemes, equipped with actions of contraction group schemes in the cases of symmetric and $\theta$-stable parabolic subgroups.
	\end{abstract}
	
	\maketitle
	
	\section{Introduction}\label{sec:intro}
	
	The basic idea of contraction after In\"on\"u--Wigner in \cite{MR55352} is to obtain a new Lie algebra (group) by replacing its structure constant with a parameter $t$ and sending it to zero. On this course, we obtain a one-parameter family of Lie algebras (groups), which is called a contraction family.
	
	Around 2016, Bernstein, Higson, and Subag started a project on an algebraic formalism for In\"on\"u--Wigner's contraction families (\cite{MR4123111, MR4130851}). In fact, they predicted the existence of group schemes over the complex and real projective lines for an algebraic model of the contraction families attached to groups with involutions in \cite{MR4123111}. They also gave examples for some classical symmetric pairs. After that, Barbasch, Higson, and Subag constructed the group schemes in \cite{MR3797197}. Their group schemes generalize the example of a non-reductive group scheme in \cite[Section 5]{MR0228502}. They also proved that the proposed group schemes are smooth (\cite[Proposition 4.3]{MR3797197}), and that the fibers at $t=0,\infty$ are the Cartan motion groups (\cite[Proposition 4.2]{MR3797197}). In the real setting, they also proved in \cite[Theorem 5.1]{MR3797197} the fiber at positive (resp.~negative) real points are isomorphic to the given real group (resp.~the real form of the complexfication attached to the composition of the given complex conjugate action and the involution).
	
	Their strategy is as follows:
	\begin{itemize}
		\item Choose a faithful representation $G\to \SL_n$ of a complex group $G$ to the special linear group $\SL_n$ of degree $n$.
		\item Define a new faithful representation 
		\[G\times_{\Spec\bC} \Spec\bC\left[\sqrt{t}^{\pm 1}\right]
		\to \SL_{2n}\times_{\Spec\bC} \Spec\bC\left[\sqrt{t}^{\pm 1}\right].\]
		\item See the conjugate action defined by $\sqrt{t}\mapsto -\sqrt{t}$ in $\SL_{2n}\times_{\Spec\bC} \Spec\bC\left[\sqrt{t}^{\pm 1}\right]$ to descend $G\times_{\Spec\bC} \Spec \bC\left[\sqrt{t}^{\pm 1}\right]$ to a group scheme $G_t$ over $\bC\left[t^{\pm 1}\right]$.
		\item Take the Zariski closure in $\SL_{2n}\times_{\Spec\bC} \bP^1_\bC$ to obtain a group scheme $\bfG$ over $\bP^1_\bC$, where $\bP^1_\bC$ is the complex projective line.
		\item Use the classical topology to compute the fibers at $t=0,\infty$ and to prove that $\bfG$ is smooth over $\bP^1_\bC$.
		\item If $G$ is defined over the real numbers, induce the anti-holomorphic involution on $\bfG$ from that on $G$ to obtain a group scheme over the real projective line. Its fibers are computed by unwinding their definitions.
	\end{itemize}
	
	It is still an interesting problem to treat contraction families in an algebraic way to let us work over general ground rings with $1/2$. This allows us to iterate the one-parameter contraction described above to obtain contraction of multi-variables (multi-contraction). Once we prove the smoothness over a general ground ring, we repeat this result to deduce smoothness of multi-contraction schemes. Another point is to obtain arithmetic structures of the families. In fact, the lack of $\sqrt{q}$ for positive rational numbers $q$ in general should lead to a family of algebraic groups which are not isomorphic to a given rational algebraic group with an involution.
	
	The purpose of this paper is to give a purely algebraic and direct approach to their theory: We work over general ground rings $k$ with $1/2$. We define the contraction algebra $\bfA$ over the polynomial ring $k\left[t\right]$ explicitly for a commutative $k$-algebra $A$ with an involution $\theta$ (see Definition \ref{defn:contraction}). In particular, the construction of $\bfA$ is functorial in $(A,\theta)$. As we will occasionally prove, $\bfA\otimes_{k\left[t\right]} k\left[\sqrt{t}\right]$ is isomorphic to the extended Rees algebra (see \eqref{eq:rees}). In other words, $\bfA$ can be obtained by a certain faithfully flat descent of the extended Rees algebra (see also \cite[Sections 2.1.2 and 2.1.3]{MR4130851} for appearance of the extended Rees algebras). The fiber $\bfA_{t_0}$ of $\bfA$ at $t=t_0$ for a unit $t_0$ of $k$ is a twisted $k$-form of $A\left[x\right]/(x^2-t_0)$ with respect to the quadratic Galois extension $k\left[x\right]/(x^2-t_0)\supset k$ by manifest:
	
	\begin{prop}[Proposition \ref{prop:fiber_descent} (2) and (3), Lemma \ref{lem:relation} (4)]
		Define involutions on $k\left[x\right]/(x^2-t_0)$ and $A\left[x\right]/(x^2-t_0)$ by
		\[\begin{array}{cc}
			\sigma_k:a+ bx\mapsto a-bx,&\sigma_A:a+ bx\mapsto \theta(a)-\theta(b)x
		\end{array}\]
		respectively. 
		\begin{enumerate}
			\renewcommand{\labelenumi}{(\arabic{enumi})}
			\item The canonical homomorphism $\varphi:k\left[x\right]/(x^2-t_0)\to A\left[x\right]/(x^2-t_0)$ commutes with the involution, that is, $\varphi\circ \sigma_k=\sigma_A\circ \varphi$. Equivalently, $\sigma_A$ is semi-linear over $k\left[x\right]/(x^2-t_0)$ for $\sigma_k$, i.e., we have \[\sigma_A(\varphi(\alpha)\beta)=\varphi(\sigma_k(\alpha))\sigma_A(\beta)\]
			for all $\alpha\in k\left[x\right]/(x^2-t_0)$ and $\beta\in A\left[x\right]/(x^2-t_0)$.
			\item There is an isomorphism of $k$-algebras from $A_{t_0}$ onto the fixed point subalgebra of $A\left[x\right]/(x^2-t_0)$ by $\sigma_A$. Moreover, its scalar extension gives rise to an isomorphism $\bfA_{t_0}\left[x\right]/(x^2-t_0)\cong A\left[x\right]/(x^2-t_0)$.
			\item If the equation $x^2=t_0$ admits a solution in $k$, then we have a $k$-algebra isomorphism $\bfA_{t_0}\cong A$.
		\end{enumerate}
	\end{prop}
	
	This generalizes the description of fibers of $\bfG$ at nonzero real points in \cite[Theorem 5.1]{MR3797197}.
	
	Our main results are:
	
	\begin{thm}[Theorems \ref{thm:smooth}, \ref{thm:Hopf}]\label{thm:summary}
		Suppose that $A$ is smooth over $k$. Let $I\subset A$ be the ideal generated by
		$\{a\in A:~\theta(a)=-a\}$.
		\begin{enumerate}
			\renewcommand{\labelenumi}{(\arabic{enumi})}
			\item There is a $k$-algebra isomorphism $\bfA/(t)\cong \Sym_{A/I} I/I^2$, where $\Sym_{A/I} I/I^2$ is the symmetric algebra over $A/I$ generated by $I/I^2$.
			\item The structure homomorphism $k\left[t\right]\to \bfA$ is smooth.
			\item If $A$ is a Hopf algebra over $k$, $\bfA$ is naturally equipped with the structure of a Hopf algebra over $k\left[t\right]$.
		\end{enumerate}
	\end{thm}
	Part (1) is proved by a direct computation. We prove (2) by showing
	\begin{enumerate}
		\renewcommand{\labelenumi}{(\roman{enumi})}
		\item $\bfA$ is flat over $k\left[t\right]$;
		\item $\bfA\otimes_{k\left[t\right]} k\left[t^{\pm 1}\right]$ is smooth over $k\left[t^{\pm 1}\right]$;
		\item $\bfA\otimes_{k\left[t\right]} k\left[t\right]/(t)$ is smooth over $ k\left[t\right]/(t)$;
		\item $\bfA$ is finitely presented over $k\left[t\right]$.
	\end{enumerate}
	The crucial step is (i). This is verified by studying a certain filtration on \[\bfA\otimes_{k\left[t\right]} k\left[\sqrt{t}\right].\]
	It is technical to check (iv). After the faithfully flat descent, we may prove that extended Rees algebras are finitely presented under a corresponding assumption. Then we work locally in the \'etale topology of $\Spec A$ and then may and do replace $k$ with a Noetherian ring. Then the assertion follows from Hilbert's basis theorem. The Hopf algebra structure stated in (3) is derived from the base change of that on $A$ to $k\left[\sqrt{t}^{\pm 1}\right]$.
	
	Let us also note:
	
	\begin{prop}[Corollary \ref{cor:compatibility}, see also Corollary \ref{cor:comparisonwithbhs} and Propositions \ref{prop:comparisonwithbhs_groups}, \ref{prop:bc_image}]\label{prop:comparison}
		Our definition agrees with that in \cite{MR3797197}.
	\end{prop}
	
	On the course of its proof, we interpret the Zariski closure in \cite{MR3797197} as the scheme-theoretic closure, which leads us to a direct construction of contraction families of real groups in the fashion of \cite{MR3797197} for complex groups. 
	
	\begin{rem}
		The group scheme $G_t$ in \cite{MR3797197} is recovered from $\bfG$ as the localization of $\bfG$ by $t$ in the sense of \cite[Chapitre premier, (2.2.4.1)]{MR217083} by Proposition \ref{prop:bc_image} and Example \ref{ex:closed_immersion} (see also Corollary \ref{cor:comparisonwithbhs}). Since $\bfG$ is defined without $G_t$ in our approach, we would like to write $\bfG_t$ for $G_t$. We apply similar notations to our generalized settings in this paper.
	\end{rem}
	
	In our definition of contraction families, the compatibility of the complex and real settings is verified in a more general context:
	
	\begin{prop}[Propositions \ref{prop:flatbc}, \ref{prop:flatbc_hopf}, Theorem \ref{thm:smooth} (1)]
		The contraction of arbitrary (resp.~smooth) commutative $k$-algebras with involutions and of their Hopf algebra structures in Theorem \ref{thm:summary} (3) commutes with flat (resp.~arbitrary) base changes.
	\end{prop}
	
	For simplicity, we only work over the polynomial ring $k\left[t\right]$. One can work over the projective line like \cite{MR3797197} if one wishes by working over $k\left[t^{-1}\right]$ and gluing the algebras over $k\left[t^{\pm 1}\right]$ as pointed out in \cite[Remark 3.2.1]{MR4130851} (see Remark \ref{rem:gluing} for details). Similar results still hold since the statements are local with respect to the two principal open subsets of the projective line.
	
	One may also replace $k\to A$ with any affine morphism over $\bZ\left[1/2\right]$ since the statements so far are local in the base. The affine hypothesis for the structure morphism can not be removed since we do not have $\theta$-stable affine open covering in general when the base is affine. We still try to discuss the local behavior of the contraction. Namely, we study the contraction algebra attached to a localization of an algebra with an involution:
	
	\begin{prop}[Variant \ref{var:localization}]
		Let $(A,\theta)$ be a commutative algebra with an involution over a commutative ring $k$ with $1/2$, and $f\in A$. If $\theta(f)=f$ (resp.~$\theta(f)=-f$) then $\theta$ extends to an involution of the localization $A_f$. Moreover, the contraction algebra of $A_f$ is isomorphic to the localization of the contraction algebra $\bfA$ of $A$ by $f$ (resp.~$\sqrt{t}f$) if $\theta(f)=f$ (resp.~$\theta(f)=-f$).
	\end{prop}
	
	We discuss this because we can glue up the affine contraction schemes locally to give global contraction schemes if a nice $\theta$-stable affine open covering exists (cf.~Example \ref{ex:sl_2}).
	
	Our algebraic approach allows us to apply the operation of contraction twice if we are given mutually commutative two involutions (double contraction). We prove an analogous result to Proposition \ref{prop:comparison} (see Corollary \ref{cor:schematic_closure_double}).
	
	As an application of removal of the group structure in our construction and the functoriality of $\bfA$, we should be able to define the contraction of group actions on schemes. This is guaranteed by the preservation of tensor products which holds under mild assumptions:
	
	\begin{prop}[Proposition \ref{prop:monoidal}, Theorem \ref{thm:summary} (2), Example \ref{ex:flat_field}]
		The assignment $(A,\theta)\rightsquigarrow \bfA$ respects the tensor product if either $k$ is a field or $A$ is smooth over $k$.
	\end{prop}
	For a digression, we compute the differential structure of the contraction of affine group schemes and their actions (Proposition \ref{prop:liealg}, Corollary \ref{cor:diffaction}). In particular, we guarantee that the spectrum $\bfG=\Spec\bfA$ of the contraction of a commutative Hopf algebra $A$ is a group lift of the contraction $\bfg$ of the Lie algebra of $G=\Spec A$ in \cite[Section 2.1.3]{MR4130851}.
	
	As a typical example of group actions on affine schemes in representation theory, one can think of the contraction families attached to symmetric varieties. That is, let $G$ be a reductive algebraic group over a field $F$ of characteristic not two with an involution $\theta$, and $K$ be the fixed point subgroup by $\theta$. Then $X=G/K$ is an affine variety over $F$, equipped with a natural involution (the Matsushima criterion). It gives rise to a contraction family $\bfX$ over $F\left[t\right]$ by taking the spectrum of the contraction algebra of the coordinate ring of $X=G/K$. For example, this connects \textit{algebraic} models of compact and noncompact symmetric spaces by thinking of the Cartan involution over the field $\bR$ of real numbers; The manifolds of real points of the fibers at nonzero points are not Riemannian symmetric since they are disconnected in general. Indeed, let $F^s$ be the separable closure of $F$. Write $H^1(\Gamma,K(F^s))$ and $H^1(\Gamma,G(F^s))$ for the first Galois cohomology of $K$ and $G$ respectively (see \cite[Chapter I, section 5.1 and Chapter II, section 1.1]{MR1867431}). According to \cite[Chapter II, Caution of Section 6.8]{MR1102012} and \cite[Chapter I, section 5.4, Corollary 1]{MR1867431}, there is a one-to-one correspondence between the set of $G(F)$-orbits in the set $(G/K)(F)$ of $F$-points of the algebraic variety $G/K$ and the kernel of the canonical map
	\[H^1(\Gamma,K(F^s))\to H^1(\Gamma,G(F^s))\]
	of pointed sets. For explicit computation of the Galois cohomology in the case of $F=\bR$, see \cite{MR3545963} and the referred papers therein.
	
	\begin{ex}\label{ex:galois}
		Put $F=\bR$ and $G=\SL_n$ with $n\geq 1$. Set $\theta=((-)^T)^{-1}$, where $(-)^T$ denotes the transpose of matrices. Then the fiber of $\bfX$ at $t=1$ (resp.~$t=-1$) is identified with $\SL_n/\SO(n)$ (resp.~$\SU(n)/\SO(n)$). One can see that \[(\SU(n)/\SO(n))(\bR)=\SU(n,\bR)/\SO(n,\bR).\]
		On the other hand, $(\SL_n/\SO(n))(\bR)$ has $\left\lfloor\frac{n}{2}\right\rfloor+1$ connected components. To understand them more concretely, put $n=2$. Then $(\SU(2)/\SO(2))(\bR)$ is the complex projective line, and $(\SL_2/\SO(2))(\bR)$ is the disjoint union of the upper and lower half planes.
	\end{ex}
	
	As pointed out in \cite{MR3545963}, none of known approaches to the computation of the Galois cohomology give us a clear description of the kernel except some special cases. To see its difficulties, let us note an easy fact: For a given possibly disconnected real reductive algebraic group $G$ with the trivial Cartan involution in the sense of \cite{MR3786301} and an algebraic subgroup $H$, we have a canonical bijection
	\[G(\bR)/H(\bR)\cong(G/H)(\bR).\]
	In particular, the kernel is trivial even if $G$ and $H$ have nontrivial Galois cohomology. For example, the cardinality of the Galois cohomology of $\SO(4)$ is smaller than that of the subgroup $\SO(2)\times \SO(2)$.

	One can connect more symmetric varieties by thinking of other involutions. For instance, we can connect the real symmetric varieties attached to commuting Galois actions on a given complex reductive algebraic group with respect the Galois group of the extension $\bC/\bR$ (cf.~\cite[Theorem 5.2, Example 5.1]{MR3797197}). We note that if one is interested in symmetric spaces of real Lie groups, a simple solution to the component issue is to take the unit component. For example, this allows us to only pick up the upper half plane in Example \ref{ex:galois}.
	
	As for a formalism to connect symmetric varieties, it is a natural question to compare $\bfX$ with the fppf quotient $\bfG/\bfK$ in the sense of \cite{MR0237513}:
	
	\begin{thm}[Theorem \ref{thm:G/K}]
		The base point $x_0=K$ of $X$ naturally extends to an $F\left[t\right]$-point $\bfx_0$ of $\bfX$. Moreover, the $\bfG$-orbit containing $\bfx_0$ is affine open in $\bfX$, and it is isomorphic to $\bfG/\bfK$.
	\end{thm}
	
	If $F=\bR$, we can again take the unit components of the manifolds of real points fiberwisely to obtain the same family whose fibers at nonzero points are symmetric spaces (Corollary \ref{cor:symmetric_space}).
	
	More generally, we introduce the quotient scheme of contraction families attached to a pair of algebraic groups with compatible involutions (the beginning of Section \ref{sec:quotient}). As another extremal example of this quotient, we discuss the case of $\theta$-stable parbaolic subgroups $Q$ when $G$ is connected reductive. We can come up with two other possible families related to $G/Q$. To explain the first candidate, let $\fg$ and $\fq$ be the Lie algebras of $G$ and $Q$ respectively. Notice that $G/Q$ can be identified with the $G$-orbit in the Grassmannian of $\fg$ containing $\fq$ if $G$ is of type (RA) in the sense of \cite[D\'efinition 5.1.6]{MR0218363}, for example, if the characteristic of $F$ is zero or $G$ is connected semisimple of adjoint type (see \cite[Remarques 5.1.7 and Proposition 5.1.3]{MR0218363}). As its contraction analog, we can define the $\bfG$-orbit of the Grassmannian $\Gr(\bfg)$ of $\bfg$ containing $\bfq$ (if it is represented by a scheme).
	
	The second candidate is to follow the idea of \cite{MR3797197}. I.e., we embed $\bfG_t/\bfQ_t$ to $\Gr(\bfg)$ and take the scheme-theoretic closure $\overline{\bfG_t/\bfQ_t}$ in the sense of \cite[Chapter 2, Exercise 3.17]{MR1917232} (see Definitions \ref{defn:img}, \ref{defn:closure}). It is naturally equipped with an action of $\bfG$ (Proposition \ref{prop:action}). They are related as follows:
	
	\begin{thm}[Propositions \ref{prop:action}, \ref{prop:point}, Corollary \ref{cor:orbit}]
		Assume the following conditions:
		\begin{enumerate}
			\renewcommand{\labelenumi}{(\roman{enumi})}
			\item $G$ is simply connected.
			\item $G$ and $K$ are of type $\mathrm{(RA)}$.
			\item Every geometric fiber of $Q\cap K$ admits a maximal torus $T$ such that the differential of each nonzero weight of the geometric fiber of $\fp$ is nonzero.
		\end{enumerate}
		Then:
		\begin{enumerate}
			\renewcommand{\labelenumi}{(\arabic{enumi})}
			\item The closed subscheme $\overline{\bfG_t/\bfQ_t}\subset\Gr(\bfg)$ is $\bfG$-invariant.
			\item The fppf quotient $\bfG/\bfQ$ coincides with the $\bfG$-orbit in $\Gr(\bfg)$ attached to $\bfq$. Moreover, $\bfG/\bfQ$ is represented by a smooth quasi-compact separated scheme over $F\left[t\right]$.
			\item The orbit map $\bfG/\bfQ\hookrightarrow\Gr(\bfg)$ factors through $\overline{\bfG_t/\bfQ_t}$.
		\end{enumerate}
	\end{thm}
	
	We remark that we can still work nicely without the simply connected assumption, but then the centralizer is slightly different from $\bfQ$ in general. We resolve this problem by taking the unit component $\bfG^\circ$ of $\bfG$ in the sense of \cite[D\'efinition 3.1]{MR0234961}. We need the third condition for good behavior of the centralizer subgroup at $t=0$. We can also work over an arbitrary ground ring with $1/2$, but we need to assume the nonvanishing condition of (iii) geometric fiberwisely. The general statement is given in Theorem \ref{thm:stabilizer}. In particular, the quotient $\bfG^\circ/\bfQ^\circ$ is representable under reasonable conditions on the characteristics of residue fields. As an application, we will be able to obtain a contraction analog of arithmetic models of $A_{\fq}(\lambda)$-modules through the theory of twisted D-modules over schemes in \cite{hayashijanuszewski}. The author is studying its structure in progress.

	\section*{Acknowledgments}
	
	I am grateful to Eyal Subag for the answer to my question in the online seminar ``Suita Representation Theory Seminar'' on adding parameters of contraction. Corollary \ref{cor:schematic_closure_double} is due to him. I also thank him for pointing out some typos of a draft of this paper.
	
	I am grateful to Professor Hisayosi Matumoto for stimulating comments. 
	
	Thanks Yoshiki Oshima for discussions in an early stage of this work.
	
	This work was supported by JSPS KAKENHI Grant Number 21J00023.
	
	\section*{Organization of this paper}
	
	In Section \ref{sec:alg}, we study the general formalism of contraction algebras $\bfA$ and their spectra $\bfX=\Spec\bfA$. In Section \ref{sec:hopf}, we lift the structures of Hopf algebras and cogroup coactions on given commutative algebras with involutions to those on contraction algebras. Their differential structures are also studied. In Section \ref{sec:quotient}, we compare quotient schemes of contraction group schemes and other related objects. In Appendix \ref{sec:sch_img}, a short discussion on scheme-theoretic images is given.
	
	\section*{Notation}
	
	We denote the ring of integers by $\bZ$. Let $\bR$ (resp.~$\bC$) denote the field of real (resp.~complex) numbers.
	
	Let $k$ be a commutative ring. We refer to the group of units of $k$ as $k^\times$. Write $\id_k$ for the identity map of $k$. For $k$-modules $M$ and $N$, write $\Hom_k(M,N)$ for the $k$-module of $k$-homomorphisms from $M$ to $N$.
	
	Let $A$ be a commutative ring, and $M$ be an $A$-module. For a nonnegative integer $n$, write $\Sym^n_A M$ for the $n$th symmetric product of $M$ as an $A$-module. Let $\Sym_A M$ be the symmetric algebra of $M$ over $A$.
	
	Let $k$ be a commutative ring. For a (small) set $\Lambda$, write $\bA^\Lambda_k$ for the affine $\Lambda$-space over $k$, i.e., $\bA^\Lambda_k=\Spec k\left[x_\lambda:~\lambda\in\Lambda\right]$. For a finitely generated and projective $k$-module $V$, let $\Gr(V)$ denote the Grassmanian scheme of $V$, i.e., the disjoint union of the Grassmanian schemes of all ranks (see \cite[Section (8.6)]{MR4225278}). This is a projective $k$-scheme by \cite[Remarks 8.24, Example 13.69, and the closed immersion (8.8.8) of Remark 8.21]{MR4225278}. For a $k$-module $M$, we denote the copresheaf $\underline{M}$ of abelian groups on the category of commutative $k$-algebras by
	$\underline{M}(R)=M\otimes_k R$.
	We remark that $\underline{M}$ is represented by the affine $k$-scheme $\Spec \Sym_k \Hom_k(M,k)$ if $M$ is finitely generated and projective as a $k$-module.
	
	For a scheme $S$, we denote its structure sheaf by $\cO_S$. For a morphism
	\[f:X\to Y\]
	of schemes, write $f^\sharp:\cO_Y\to f_\ast\cO_X$ for its structure homomorphism. For a (separated) scheme $X$ over a commutative ring with an involution $\theta$, let $X^\theta$ denote the (closed) subscheme of $\theta$-fixed points as in \cite[Lemma 3.1.1]{MR4627704}. For a smooth affine group scheme $G$ over a commutative ring, let $G^\circ$ denote the unit component of $G$ in the sense of \cite[D\'efinition 3.1]{MR0234961}. It agrees with the open subscheme of $G$ attached to the unit section defined in \cite[Corollaire (15.6.5)]{MR0217086} by \cite[Cas particulier 3.4 and the proof of Th\'eor\`eme 3.10]{MR0234961}. 
	
	For a pair of a commutative ring $A$ and an ideal $I\subset A$, define its extended Rees algebra $A\left[u,Iu^{-1}\right]$ as the subring
	$A\left[u\right]\oplus \oplus_{n\geq 1} I^n u^{-n}
	\subset A\left[u^{\pm 1}\right]$.
	Similarly, for a pair of scheme $X$ and an ideal sheaf $\cI\subset\cO_X$, define its extended Rees algebra as
	$\cO_X\left[u\right]\oplus \oplus_{n\geq 1} \cI^n u^{-n}\subset \cO_X\left[u^{\pm 1}\right]$.
	
	Let $\Proj$ denote the functor of the projective spectrum . For a commutative nonnegatively graded ring $A$ and an integer $n$, write $\cO_{\Proj A}(n)$ for the $n$th Serre twist of $\cO_X$.
	
	For a group $G$, equipped with an involution $\theta$, we denote its first group cohomology by $H^1(\theta,G)$ (\cite[Chapter I, section 5.1]{MR1867431}). For a homomorphism $H\to G$ of groups with compatible involutions $\theta$, we denote the kernel of the map between their first group cohomology by $\Ker(H^1(\theta,H)\to H^1(\theta,G))$.
	
	For a set $E$, equipped with an action of a group $\Gamma$, we denote its fixed point part by $E^\Gamma$.
	
	For a manifold $X$, let $\pi_0(X)$ denote the set of its path components.
	
	\section{Contraction algebras}\label{sec:alg}
	Let $k$ be a commutative $\bZ\left[1/2\right]$-algebra. Let $A$ be a commutative $k$-algebra with an involution $\theta$. Henceforth write
	\[\begin{array}{cc}
		A^{\theta}=\{a\in A:~\theta(a)=a\},&A^{-\theta}=\{a\in A:~\theta(a)=-a\}.
	\end{array}\]
	Since $2\in A^\times$, we have a decomposition $A=A^{\theta}\oplus A^{-\theta}$. Write
	$X=\Spec A$. For a positive integer $n\geq 1$, write $(A^{-\theta})^n$ for the $A^{\theta}$-submodule of $A$ spanned by products of $n$ elements of $A^{-\theta}$. For convention, we set $(A^{-\theta})^0\coloneqq A^{\theta}$.
	
	The symbol $t$ will be the variable for the parameter of contractions.
	
	\begin{defn}\label{defn:contraction}
		\begin{enumerate}
			\renewcommand{\labelenumi}{(\arabic{enumi})}
			\item Write $\bfA$ for the $k\left[t\right]$-subalgebra of $A\left[\sqrt{t}^{\pm 1}\right]$ generated by $A^{\theta}$ and $\frac{1}{\sqrt{t}}A^{-\theta}$. Explicitly,
			\begin{equation}
				\bfA=A^{\theta}\left[t\right]
				\oplus \left(\oplus_{n\geq 0}A^{-\theta}\frac{t^n}{\sqrt{t}}\right)
				\oplus \oplus_{n\geq 2} \frac{1}{\sqrt{t}^n} (A^{-\theta})^n
				\subset A\left[\sqrt{t}^{\pm 1}\right],\label{eq:bfA}
			\end{equation}
			We call $\bfA$ the contraction algebra.
			\item Let $\bfA_t$ be the localization of $\bfA$ by $t$, i.e., $\bfA_t=\bfA\otimes_{k\left[t\right]} k\left[t^{\pm 1}\right]$.
			\item Write $\bfX=\Spec\bfA$ and $\bfX_t=\Spec \bfA_t$.
			We call $\bfX$ the contraction scheme.
			\item For $t_0\in k$, set $\bfA_{t_0}\coloneqq \bfA_t/(t-t_0)$ and $\bfX_{t_0}=\Spec \bfA_{t_0}$.
		\end{enumerate}
		We will apply similar notations to morphisms.
	\end{defn}
	
	Though $\bfX$ is called a family in \cite{MR4123111, MR4130851, MR3797197}, we prefer to use the terminology ``scheme'' in order to emphasize scheme-theoretic perspectives.

	\begin{ex}\label{ex:trivial}
		Put $\theta=\id_A$. Then we have $\bfA=A\left[t\right]$.
	\end{ex}
	
	\begin{ex}
		Put $A=k\left[x\right]$. Define an involution $\theta$ on $A$ by $\theta(x)=-x$. Then one has $\bfA=k\left[x^2,t,\frac{x}{\sqrt{t}}\right]$ (see Proposition \ref{prop:generatorofbfA} if necessary). More simply, $y\mapsto \frac{x}{\sqrt{t}}$ determines an isomorphism $\bfA\cong k\left[t,y\right]$ of $k\left[t\right]$-algebras. Likewise, put $B=k\left[x^{\pm 1}\right]$ with $\eta(x)=-x$. Then we have \[\bfB=k\left[x^{\pm 2},t,\frac{x}{\sqrt{t}}\right]
		\cong k\left[x^2,t,\frac{x}{\sqrt{t}}\right]_{\sqrt{t}x}
		\cong k\left[t,y\right]_{ty}\]
		(use Variant \ref{var:localization} if necessary for the middle isomorphism). In particular, $\bfB_0$ is zero. Its geometric counterpart is the fact that $\Spec\eta$ acts freely on $Y=\Spec k\left[x^{\pm 1}\right]$ (see Theorem \ref{thm:smooth} (2))
	\end{ex}

	\begin{ex}[double contraction]
		Let $\eta$ be another involution of $A$ over $k$ commuting with $\theta$. Let $\bfA(\theta)$ be the contraction algebra attached to $(A,\theta)$. Then $\eta$ naturally extends to an involution of $\bfA(\theta)$, which we denote by the same symbol $\eta$. Write $\bfA(\theta,\eta)$ for the contraction algebra attached to $(\bfA,\eta)$. We call $\bfA(\theta,\eta)$ the double contraction of $A$ attached to $(\theta,\eta)$. One can define multi-contractions in a similar way.
	\end{ex}
	
	Let us note fundamental relations of the algebras appearing above:
	
	\begin{lem}\label{lem:relation}
		\begin{enumerate}
			\renewcommand{\labelenumi}{(\arabic{enumi})}
			\item In $\bfX$, the open subscheme $\bfX_t$ is scheme-theoretically dense.
			\item The containment $\bfA\subset A\left[\sqrt{t}^{\pm 1}\right]$ induces an isomorphism of $k\left[t^{\pm 1}\right]$-algebras from $\bfA_t$ onto
			\begin{equation}
				A^{\theta}\left[t^{\pm 1}\right]\oplus \frac{1}{\sqrt{t}}A^{-\theta}\otimes_k k\left[t^{\pm 1}\right]
				\subset A\left[\sqrt{t}^{\pm 1}\right].\label{eq:A_t}
			\end{equation}
			\item The map $\bfA_t\hookrightarrow A\left[\sqrt{t}^{\pm 1}\right]$ of (1) extends to an isomorphism
			\[\bfA_t\otimes_{k\left[t^{\pm 1}\right]} k\left[\sqrt{t}^{\pm 1}\right]\cong A\left[\sqrt{t}^{\pm 1}\right]\]
			of $k\left[\sqrt{t}^{\pm 1}\right]$-algebras.
			\item For $t_0\in k^\times$, we have a natural isomorphism of $k\left[x\right]/(x^2-t_0)$-algebras
			\[\bfA_{t_0}\otimes_k k\left[x\right]/(x^2-t_0)
			\cong A\otimes_k k\left[x\right]/(x^2-t_0).\]
		\end{enumerate}
	\end{lem}
	
	\begin{proof}
		For (1) and (2), observe that we obtain a map $\bfA_t\to A\left[\sqrt{t}^{\pm 1}\right]$ from the containment $\bfA\subset A\left[\sqrt{t}^{\pm 1}\right]$ in virtue of the universal property of the localization. In particular, the localization map $\bfA\to\bfA_t$ is injective. Part (1) then follows by definition of the scheme-theoretic closure. In view of generalities on the localization of commutative rings, we also find that the map $\bfA_t\to A\left[\sqrt{t}^{\pm 1}\right]$ is injective. Its image is the $k\left[t^{\pm 1}\right]$-subalgebra of $A\left[\sqrt{t}^{\pm 1}\right]$ generated by $A^{\theta}$ and $\frac{1}{\sqrt{t}}A^{-\theta}$ by definitions. It is straightforward that it coincides with \eqref{eq:A_t}. This shows (2).
		
		Part (3) is verified by the following identification:
		\begin{flalign*}
			&A\left[\sqrt{t}^{\pm 1}\right]\\
			&=A^{\theta}\left[t^{\pm 1}\right]\oplus \frac{1}{\sqrt{t}}A^{-\theta}\otimes_k k\left[t^{\pm 1}\right]
			\oplus \frac{1}{\sqrt{t}}A^{\theta}\left[t^{\pm 1}\right]
			\oplus A^{-\theta}\otimes_k k\left[t^{\pm 1}\right]\\
			&\cong \bfA_t\oplus \frac{1}{\sqrt{t}} \bfA_t\\
			&\cong \bfA_t\otimes_{k\left[t^{\pm 1}\right]} k\left[\sqrt{t}^{\pm 1}\right].
		\end{flalign*}
		
		Take its base change by the evaluation map $\ev_{x}:k\left[\sqrt{t}^{\pm 1}\right]\to k\left[x\right]/(x^2-t_0)$ at $\sqrt{t}=x$ to obtain (4). In fact, $\ev_{x}$ is well-defined since $x$ is a unit of $k\left[x\right]/(x^2-t_0)$ by $x\cdot t^{-1}_0x=1$ in $k\left[x\right]/(x^2-t_0)$. Since the restriction of $\ev_x$ to $k\left[t^{\pm 1}\right]$ factors through the evaluation map $\ev_{t_0}$ to $k\subset k\left[x\right]/(x^2-t_0)$ at $t=\ev_x(t)=x^2=t_0$, the base change of the left hand side in (3) is isomorphic to
		\[\bfA_t\otimes_{k\left[t^{\pm 1}\right]} k \otimes_k k\left[x\right]/(x^2-t_0)
		\cong \bfA_{t_0}\otimes_k k\left[x\right]/(x^2-t_0);\]
		\[\begin{tikzcd}
			k\left[t^{\pm 1}\right]\ar[r, hook]\ar[d, "\ev_{t_0}"']
			&k\left[\sqrt{t}^{\pm 1}\right]\ar[d, "\ev_{x}"]\\
			k\ar[r, hook]& k\left[x\right]/(x^2-t_0).
		\end{tikzcd}\]
		This completes the proof.
	\end{proof}
	
	There are conceptual proofs of (3) and (4) from the perspectives of Galois theory. Recall that $k\left[t^{\pm 1}\right]\subset k\left[\sqrt{t}^{\pm 1}\right]$ is a quadratic Galois extension in the sense of \cite[Definition 1.3.1]{MR4627704} for the involution $\sqrt{t}\mapsto -\sqrt{t}$ of $k\left[\sqrt{t}^{\pm 1}\right]$. Similarly, for a unit $t_0\in k^\times$, $k\left[x\right]/(x^2-t_0)$ is a (possibly split) quadratic Galois extension of $k$ for
	$a+bx\mapsto a-bx$.
	
	\begin{prop}\label{prop:fiber_descent}
		Let $t_0\in k^\times$.
		\begin{enumerate}
			\renewcommand{\labelenumi}{(\arabic{enumi})}
			\item Define an involution on $A\left[\sqrt{t}^{\pm 1}\right]$ by
			$\sum_i a_i \sqrt{t}^i\mapsto \sum_i \theta(a_i)(-\sqrt{t})^i$.
			Then the canonical homomorphism $k\left[\sqrt{t}^{\pm 1}\right]\to A\left[\sqrt{t}^{\pm 1}\right]$ commutes with the involution. Moreover, its fixed point subalgebra agrees with \eqref{eq:A_t}.
			\item Define an involution on $A\left[x\right]/(x^2-t_0)$ by
			$a+bx\mapsto \theta(a)-\theta(b)x$.
			Then the canonical homomorphism $k\left[x\right]/(x^2-t_0)\to A\left[x\right]/(x^2-t_0)$ commutes with the involution. Moreover, its fixed point subalgebra $A^\theta\oplus A^{-\theta}x$ is naturally isomorphic to $\bfA_{t_0}$ as a commutative $k$-algebra.
			\item If we are given an element $\alpha\in k$ such that $\alpha^2=t_0$ then $\alpha$ determines a $k$-algebra isomorphism $\bfA_{t_0}\cong A$.
		\end{enumerate}
	\end{prop}
	
	We remark that (2) and (3) generalize the latter half of \cite[Theorem 5.1]{MR3797197}. 
	
	\begin{proof}
		Part (1) is straightforward. Henceforth let us identify the subalgebra in (1) with $\bfA_t$ through Lemma \ref{lem:relation} (1). Take the base change of (1) by the quotient map $k\left[\sqrt{t}^{\pm 1}\right]\to k\left[\sqrt{t}^{\pm 1}\right]/(t-t_0)$, and replace the symbol $\sqrt{t}$ with $x$ to deduce (2) since base changes respect Galois extensions. We check directly as in (1) that the fixed point subalgebra of $A\left[x\right]/(x^2-t_0)$ coincides with $A^\theta\oplus A^{-\theta}x$. Part (3) is obtained by sending $x$ to $\alpha$.
	\end{proof}
	
	Then the isomorphisms of (3) and (4) of Lemma \ref{lem:relation} follow by generalities on the Galois descent (see \cite[Theorem A.3]{hayashikgb}).

	\begin{prop}\label{prop:flatbc}
		Let $k\to k'$ be a flat homomorphism of commutative $\bZ\left[1/2\right]$-algebras. Put $B=k'\otimes_k A$ and $\eta=k'\otimes_k\theta$. Then there is a canonical isomorphism $k'\otimes_k\bfA\cong \bfB$ of $k'\left[t\right]$-algebras.
	\end{prop}
	
	\begin{proof}
		Since $k'$ is flat over $k$, $\bfA\otimes_k k'$ maps injectively into $(A\otimes_k k')\left[\sqrt{t}^{\pm 1}\right]$. Therefore the proof will be completed by comparing the generators. Since we have a canonical splitting
		$A=A^{\theta}\oplus A^{-\theta}$,
		one has canonical isomorphisms
		\[\begin{array}{cc}
			k'\otimes_k A^{\theta}\cong B^{\eta},
			&k'\otimes_k A^{-\theta}\cong B^{-\eta}.
		\end{array}\]

	\end{proof}

	We next see the relation with the definition in \cite{MR3797197}. We begin with an elementary observation from a corresponding result in ring theory. To state it, let us introduce the following notations: For $a\in A$, write
	\[\begin{array}{cc}
		a_+=\frac{a+\theta (a)}{2},
		&a_-=\frac{a-\theta (a)}{2}.
	\end{array}\]
	For an indexed element $a_\lambda\in A$, we will write $a_{\lambda\pm}=(a_\lambda)_{\pm}$.
	
	\begin{lem}\label{lem:elementaryinvarianttheory}
		Let $\{a_\lambda\}$ be a (possibly infinite) generator of $A$.
		\begin{enumerate}
			\renewcommand{\labelenumi}{(\arabic{enumi})}
			\item As a $k$-algebra, $A^{\theta}$ is generated by elements of the forms $a_{\lambda+}$ and $a_{\lambda-}a_{\mu-}$ ($\lambda,\mu\in\Lambda$).
			\item As an $A^{\theta}$-module, $A^{-\theta}$ is generated by elements of the form $a_{\lambda-}$ ($\lambda\in\Lambda$).
		\end{enumerate}
	\end{lem}
	
	For a tuple $I=(\lambda_1,\lambda_2,\ldots,\lambda_n)\in\Lambda^n$ ($n\geq 0$), write
	\[\begin{array}{ccc}
		a_I=\prod_{i=1}^n a_{\lambda_i},
		&a_{(I,+)}=\prod_{i=1}^n a_{\lambda_i+},
		&a_{(I,-)}=\prod_{i=1}^n a_{\lambda_i-}
	\end{array}\]
	We also put $|I|=n$.
	
	\begin{proof}
		Let $a\in A$. Choose a presentation
		\[a=\sum_{I\in\Lambda^n} c_I a_I =\sum_{I=(\lambda_1,\lambda_2,\ldots,\lambda_n)\in\Lambda^n} c_I\prod_{i=1}^n (a_{\lambda+}+a_{\lambda-}),\]
		where $c_I\in k$ for each tuple $I$. Expand the right hand side to obtain an expression
		\[a=\sum_{I\in\Lambda^n,J\in\Lambda^m} c_{IJ} a_{(I,+)} a_{(J,-)},\]
		where $c_{IJ}\in k$.
		
		For (1), assume $\theta(a)=a$. Then we have
		\[a=\frac{a+\theta(a)}{2}
		=\sum_{\overset{I,J}{|J|\mathrm{~is~even}}} c_{IJ} a_{(I,+)}a_{(J,-)}.\]
		For each $J=(\lambda_1,\lambda_2,\ldots,\lambda_{2n})$, we have
		\[a_{(J,-)}=\prod_{i=1}^n (a_{\lambda_{2i-1}-}a_{\lambda_{2i}-}).\]
		Therefore $a$ is expressed by a polynomial of $a_{\lambda+}$ and $a_{\lambda-}a_{\mu-}$.
		
		For (2), assume $\theta(a)=-a$. Then we have
		\[a=\frac{a-\theta(a)}{2}
		=\sum_{\overset{I,J}{|J|\mathrm{~is~odd}}} c_{IJ} a_{(I,+)}a_{(J,-)}.\]
		For each $J=(\lambda_1,\lambda_2,\ldots,\lambda_{2n+1})$, write $\lambda_J=\lambda_{2n+1}$ and $J^\circ=(\lambda_1,\lambda_2,\ldots,\lambda_{2n})$. Then we have
		\[a=\frac{a-\theta(a)}{2}=\sum_{\overset{I,J}{|J|\mathrm{~is~odd}}} 
		c_{IJ} a_{(I,+)}a_{(J^\circ,-)}
		a_{\lambda_J-}.\]
		Since $c_{IJ} a_{(I,+)}a_{(J^\circ,-)}$ belongs to $A^{\theta}$ for each pair $(I,J)$ of tuples with $|J|$ odd, $a$ is expressed as an $A^{\theta}$-linear combination of elements of the form $a_{\lambda-}$.
	\end{proof}
	
	As an immediate consequence, we obtain:
	
	\begin{prop}\label{prop:generatorofbfA}
		Let $\{a_\lambda\}_{\lambda\in\Lambda}$ be as in Lemma \ref{lem:elementaryinvarianttheory}. Then $\bfA$ is generated by the elements of the forms $a_{\lambda+}$ and $\frac{1}{\sqrt{t}}a_{\lambda-}$ as a $k\left[t\right]$-algebra ($\lambda\in\Lambda$).
	\end{prop}
	
	\begin{cor}\label{cor:comparisonwithbhs}
		Choose a generator $(a_\lambda)_{\lambda\in\Lambda}$ of $A$.
		Write
		$\iota:X\hookrightarrow \bA^\Lambda_k$
		for the corresponding closed immersion. 
		\begin{enumerate}
			\renewcommand{\labelenumi}{(\arabic{enumi})}
			\item There is a closed immersion
			$\tilde{\iota}:\bfX\hookrightarrow\bA^{4\Lambda}_{k\left[t\right]}$
			such that the map
			\[X\otimes_k k\left[\sqrt{t}^{\pm 1}\right]
			\cong \bfX\otimes_{k\left[t\right]} k\left[\sqrt{t}^{\pm 1}\right]
			\xrightarrow{\tilde{\iota}\otimes_{k\left[t\right]} k\left[\sqrt{t}^{\pm 1}\right]}
			\bA^{4\Lambda}_{k\left[\sqrt{t}^{\pm 1}\right]}\]
			(see Lemma \ref{lem:relation} (3) for the first isomorphism)
			is expressed as
			\[x\mapsto \frac{1}{2}\left(\begin{array}{cc}
				\iota(x)+\iota(\theta(x)) & \sqrt{t}(\iota(x)-\iota(\theta(x))) \\
				\sqrt{t}^{-1}(\iota(x)-\iota(\theta(x))) & \iota(x)+\iota(\theta(x))
			\end{array}\right),\]
			where $x$ runs through $R$-points of $X$ for $k\left[\sqrt{t}^{\pm 1}\right]$-algebras $R$.
			\item The scheme $\bfX$ exhibits the scheme-theoretic closure of $\bfX_t$ along
			\begin{equation}
				\bfX_t
				\xrightarrow{\tilde{\iota} \otimes_{k\left[t\right]} k\left[t^{\pm 1}\right]}
				\bA^{4\Lambda}_{k\left[t^{\pm 1}\right]}
				\hookrightarrow \bA^{4\Lambda}_{k\left[t\right]}.
				\label{eq:restriction}
			\end{equation}
		\end{enumerate}
	\end{cor}
	
	\begin{proof}
		Define a $k\left[t\right]$-algebra homomorphism
		\[\tilde{f}:k\left[t,x_{ij\lambda};~1\leq i,j\leq 2,~\lambda\in\Lambda\right]\to\bfA\]
		by
		\[\begin{array}{cc}
			\tilde{f}(x_{ii\lambda})=a_{\lambda+}&(i\in\{1,2\})\\
			\tilde{f}(x_{12\lambda})=\frac{t}{\sqrt{t}}a_{\lambda-}\\
			\tilde{f}(x_{21\lambda})=\frac{1}{\sqrt{t}}a_{\lambda-}
		\end{array}\]
		for $\lambda\in\Lambda$. We prove that $\tilde{\iota}\coloneqq \Spec \tilde{f}$ satisfies the conditions of (1). Observe that $\tilde{f}$ is surjective by Proposition \ref{prop:generatorofbfA}. Hence $\tilde{\iota}$ is a closed immersion. The base change of $\tilde{f}$ to $k\left[\sqrt{t}^{\pm 1}\right]$ has the following identification:
		\[k\left[\sqrt{t}^{\pm 1},x_{ij\lambda};~1\leq i,j\leq 2,~\lambda\in\Lambda\right]
		\to \bfA\otimes_{k\left[t\right]} k\left[\sqrt{t}^{\pm 1}\right]
		\cong A\left[\sqrt{t}^{\pm 1}\right];\]
		\[\begin{array}{cc}
			x_{ii\lambda}\mapsto a_{\lambda+}&(i\in\{1,2\})\\
			x_{12\lambda}\mapsto \frac{t}{\sqrt{t}}a_{\lambda-}\\
			x_{21\lambda}\mapsto\frac{1}{\sqrt{t}}a_{\lambda-}.
		\end{array}\]
		It is easy to show that the corresponding map of the sets of $R$-points coincides with that in (1).
		
		For (2), notice that the morphism \eqref{eq:restriction} factors through $\bfX$ by definitions. The assertion now follows from Lemma \ref{lem:relation} (1) and Proposition \ref{prop:detect_image}.
	\end{proof}
	
	\begin{rem}\label{rem:simplified_BHS}
		More easily, similar assertions hold if we replace the closed immersion of (1) with the map
		$\bfX\hookrightarrow
		\bA^{2\Lambda}_{k\left[\sqrt{t}^{\pm 1}\right]}$ defined by
		\[k\left[t,x_{\lambda},y_\lambda;~\lambda\in\Lambda\right]\to\bfA;\]
		\[\begin{array}{cc}
			x_{\lambda}\mapsto a_{\lambda+}&
			y_{\lambda}\mapsto\frac{1}{\sqrt{t}}a_{\lambda-}.
		\end{array}\]
		Its base change to $k\left[\sqrt{t}^{\pm 1}\right]$ is identified with the map
		\[X\otimes_k k\left[\sqrt{t}^{\pm 1}\right]
		\hookrightarrow
		\bA^{2\Lambda}_{k\left[\sqrt{t}^{\pm 1}\right]};~x\mapsto\frac{1}{2}\left(\begin{array}{c}
			\iota(x)+\iota(\theta(x))\\
			\sqrt{t}^{-1}(\iota(x)-\iota(\theta(x)))
		\end{array}\right).\]
		We followed \cite{MR3797197} to see the complicated version in Corollary \ref{cor:comparisonwithbhs} for later discussions in the group setting.
	\end{rem}
	
	\begin{ex}\label{ex:sl_2/so(2)}
		Put $A=k\left[x,y,z\right]/(x^2-y^2+z^2-1)$. Define an involution $\theta$ on $A$ by $\theta(x)=x$, $\theta(y)=-y$, $\theta(z)=-z$. Then we follow the construction of Remark \ref{rem:simplified_BHS} to obtain $\bfA\cong k[t,x,y,z]/(x^2-ty^2-tz^2-1)$ by choosing $(x,y,z)$ as a generator of $A$.
	\end{ex}
	
	\begin{cor}
		Let $\eta$ be an involution on $A$ over $k$ commuting with $\theta$. Write
		\[\begin{array}{cccc}
			A^{\theta,\eta}=A^{\theta}\cap A^\eta,&A^{\theta,-\eta}=A^{\theta}\cap A^{-\eta},
			&A^{-\theta,\eta}=A^{-\theta}\cap A^\eta,&A^{-\theta,-\eta}=A^{-\theta}\cap A^{-\eta}.
		\end{array}\]
		We denote the contraction parameters of $\bfA(\theta)$ and $\bfA(\theta,\eta)$ by $t_1$ and $t_2$ respectively to distinguish the variables of the first and second contractions. Then $\bfA(\theta,\eta)$ is identified with the $k\left[t_1,t_2\right]$-subalgebra of $A\left[\sqrt{t}_1^{\pm 1},\sqrt{t}_2^{\pm 1}\right]$ generated by
		\[A^{\theta,\eta},~A^{\theta,-\eta}\frac{1}{\sqrt{t}_2},~A^{-\theta,\eta}\frac{1}{\sqrt{t}_1}
		,~A^{-\theta,-\eta}\frac{1}{\sqrt{t_1}\sqrt{t_2}}.\]
		In particular, one has a canonical isomorphism $\bfA(\theta,\eta)\cong \bfA(\eta,\theta)$ because of the symmetry of the description of $\bfA(\theta,\eta)$ in $\theta$ and $\eta$.
	\end{cor}
	
	\begin{proof}
		One may regard $\bfA(\theta)\left[\sqrt{t}_2^{\pm 1}\right]$ as a subalgebra of $A\left[\sqrt{t}_1^{\pm 1},\sqrt{t}_2^{\pm 1}\right]$. Therefore we may work within $\bfA(\theta)\left[\sqrt{t}_2^{\pm 1}\right]$.
		
		Notice that $\bfA(\theta)$ is generated by
		\[A^{\theta,\eta},~A^{-\theta,\eta}\frac{1}{\sqrt{t}_1}\subset\bfA(\theta)^{\eta}\]
		and
		\[A^{\theta,-\eta},~A^{-\theta,-\eta}\frac{1}{\sqrt{t_1}}\subset\bfA(\theta)^{-\eta}.\]
		The assertion now follows from Proposition \ref{prop:generatorofbfA}.
	\end{proof}
	
	\begin{cor}[Subag]\label{cor:schematic_closure_double}
		Suppose that we are given an involution $\eta$ on $A$ over $k$ commuting with $\theta$.
		Choose a closed immersion $i:X\hookrightarrow\bA^{\Lambda}_k$ as in Corollary \ref{cor:comparisonwithbhs}. Define $\bfX(\theta)=\Spec\bfA(\theta)$ and $\bfX(\theta,\eta)=\Spec\bfA(\theta,\eta)$. Let
		$\tilde{i}_\theta:\bfX(\theta)\hookrightarrow\bA^{4\Lambda}_{k\left[t_1\right]}$
		be the map in Corollary \ref{cor:comparisonwithbhs}. Let
		$\tilde{i}_{\theta,\eta}:\bfX(\theta,\eta)\hookrightarrow
		\bA^{16\Lambda}_{k\left[t_1,t_2\right]}$
		be the map obtained by applying $\tilde{i}_\theta$ to Corollary \ref{cor:comparisonwithbhs}. Then $\bfX(\theta,\eta)$ exhibits the scheme-theoretic closure of $\bfX(\theta,\eta)\otimes_{k\left[t_1,t_2\right]}
		k\left[t_1^{\pm 1},t^{\pm 1}_2\right]$
		in $\bA^{16\Lambda}_{k\left[t^{\pm 1}_1,t^{\pm 1}_2\right]}$. The base change of $\tilde{i}_{\theta,\eta}$ to $k\left[\sqrt{t}^{\pm 1}_1,\sqrt{t}^{\pm 1}_2\right]$ is identified with the map
		\[X\otimes_k k\left[\sqrt{t}^{\pm 1}_1,\sqrt{t}^{\pm 1}_2\right]\hookrightarrow
		\bA^{16\Lambda}_{k\left[\sqrt{t}^{\pm 1}_1,\sqrt{t}^{\pm 1}_2\right]};~x\mapsto \frac{1}{4}(x_{ij}),
		\]
		where
		\[\begin{array}{cc}
			x_{jj}=i(x)+ i(\theta x)+i(\eta x)+i(\theta\eta x)&(1\leq j\leq 4)\\
			x_{12}=x_{34}=i(x)+ i(\theta x)+i(\eta x)+i(\theta\eta x)\\
			x_{13}=x_{24}=\sqrt{t}_2(i(x)+ i(\theta x)-i(\eta x)-i(\theta\eta x))\\
			x_{14}=\sqrt{t}_1 \sqrt{t}_2(i(x)-i(\theta x)-i(\eta x)+i(\theta\eta x))\\
			x_{21}=x_{43}=\sqrt{t}^{-1}_1(i(x)-i(\theta x)+i(\eta x)-i(\theta\eta x))\\
			x_{23}=\sqrt{t}^{-1}_1\sqrt{t}_2(i(x)-i(\theta x)-i(\eta x)+i(\theta\eta x))\\
			x_{31}=x_{42}=\sqrt{t}_2^{-1}(i(x)+ i(\theta x)-i(\eta x)-i(\theta\eta x))\\
			x_{32}=\sqrt{t}_1\sqrt{t}_2^{-1}(i(x)-i(\theta x)-i(\eta x)+i(\theta\eta x))\\
			x_{41}=\sqrt{t}_1^{-1}\sqrt{t}_2^{-1} (i(x)-i(\theta x)-i(\eta x)+i(\theta\eta x)).
		\end{array}\]
	\end{cor}
	
	\begin{cor}\label{cor:offinitetype}
		If $A$ is of finite type over $k$, then so is $\bfA$ over $k\left[t\right]$.
	\end{cor}
	
	\begin{cor}\label{cor:surj}
		The assignment $(A,\theta)\rightsquigarrow \bfA$ respects surjective maps.
	\end{cor}
	
	It turns out in scheme theory that contraction of affine $k$-schemes with involutions respects closed immersions. For a digression, let us also prove an open analog of Corollary \ref{cor:surj}:

	\begin{var}\label{var:localization}
		Let $f\in A^{\theta}\cup A^{-\theta}$. Define $\bff\in\bfA$ by
		\[\bff=\begin{cases}
			f&(f\in A^{\theta})\\
			\sqrt{t}f&(f\in A^{-\theta}).
		\end{cases}\]
		\begin{enumerate}
			\renewcommand{\labelenumi}{(\arabic{enumi})}
			\item Write $B=A_f$ for the localization of $A$ by $f$. Then $\theta$ induces an involution on $B$, which we will denote by the same symbol $\theta$.
			\item We have a canonical isomorphism $\bfA_{\bff}\cong \bfB$. 
		\end{enumerate}
	\end{var}
	
	\begin{proof}
		Part (1) is clear. We prove (2). Since the canonical homomorphism
		\[p:A\to B\]
		respects the involution, it induces a homomorphism $\bfp:\bfA\to\bfB$. We wish to prove that $\bfp(\bff)$ is a unit of $\bfB$ in order to extend $\bfp$ to a map $\alpha:\bfA_{\bff}\to \bfB$. To see this, recall that $p(f)$ is a unit of $B$ by definition of $B=A_f$. The assertion for $f\in A^{\theta}$ then follows since $\bfp$ restricts to $p:A^{\theta}\to B^\theta$. Assume $f\in A^{-\theta}$. In this case, $\frac{1}{\sqrt{t}}p(f)^{-1}$ belongs to $\bfB$ and
		$\bfp(\bff)\frac{1}{\sqrt{t}}p(f)^{-1}=1$.
		This shows that $\bfp(\bff)$ is a unit as desired.
		
		We next construct its inverse. We may and do regard $\bfA_{\bff}$ as a subset of $A\left[\sqrt{t}^{\pm 1}\right]_{\bff}$ since localization is flat in general. The canonical homomorphism $A\to A\left[\sqrt{t}^{\pm 1}\right]_{\bff}$ extends to $q:B\left[\sqrt{t}^{\pm 1}\right]\to A\left[\sqrt{t}^{\pm 1}\right]_{\bff}$ by definition of $\bff$. In fact, $f$ is a unit of $A\left[\sqrt{t}^{\pm 1}\right]_{\bff}$ by definition of $\bff$.
		
		We next prove that $q|_{\bfB}$ factors through $\bfA_{\bff}$. In particular, we will obtain
		\[\beta:\bfB\to \bfA_{\bff}.\]
		We may restrict ourselves to generators. Each element of $B$ is expressed as $\frac{a}{f^n}$ with $a\in A$ and $n\geq 0$ (put $f^0=1$). We may multiply $f^n$ to $a$ and $f^n$ if necessary to assume that $n$ is even. In particular, the denominator belongs to $A^{\theta}$. Then we can easily see that $\frac{a}{f^n}$ belongs to $B^\theta$ (resp.~$B^{-\theta}$) if and only if there exists a positive integer $m$ such that $\theta(a)f^m=af^m$ (resp.~$\theta(a)f^m=-af^m$). We may multiply $f^m$ if necessary to assume that $m$ is even. In particular, $\frac{a}{f^n}$ belongs to $B^\theta$ (resp.~$B^{-\theta}$) if and only if there exists a positive even integer $m$ such that $af^m\in A^{\theta}$ (resp.~$af^m\in A^{-\theta}$). Let us take $m$ in each case. We prove that $q\left(\frac{a}{f^n}\right)$ lies in $\bfA_{\bff}$ by case-by-case study:
		\[q\left(\frac{a}{f^n}\right)=q\left(\frac{af^m}{f^{n+m}}\right)
		=\frac{af^m}{\bff^{n+m}}\in\bfA_{\bff}\]
		if $f\in A^{\theta}$ and $\frac{a}{f^n}\in B^{\theta}$;
		\[q\left(\frac{a}{f^n}\right)
		=q\left(\frac{af^m}{f^{n+m}}\right)
		=\frac{af^m\sqrt{t}^{n+m}}{\bff^{n+m}}\in\bfA_{\bff}\]
		if $f\in A^{-\theta}$ and $\frac{a}{f^n}\in B^{\theta}$;
		\[q\left(\frac{a}{f^n}\frac{1}{\sqrt{t}}\right)
		=q\left(\frac{af^m}{f^{n+m}}\frac{1}{\sqrt{t}}\right)
		=\frac{af^m}{\sqrt{t}}\frac{1}{\bff^{n+m}}\in\bfA_{\bff}\]
		if $f\in A^{\theta}$ and $\frac{a}{f^n}\in B^{-\theta}$;
		\[q\left(\frac{a}{f^n}\frac{1}{\sqrt{t}}\right)=
		q\left(\frac{af^m}{f^{n+m}}\frac{1}{\sqrt{t}}\right)
		=\frac{af^m\sqrt{t}^{n+m-1}}{\bff^{n+m}}\in\bfA_{\bff}
		\]
		if $f\in A^{-\theta}$ and $\frac{a}{f^n}\in B^{-\theta}$ (remember that $n$ and $m$ are even).
		
		Finally, we prove that $\alpha$ and $\beta$ are mutually inverse. To see $\beta\circ\alpha=\id_{\bfA_{\bff}}$, we may restrict to $\bfA$ since the localization map $\bfA\to\bfA_{\bff}$ is an epimorphism of $k\left[t\right]$-algebras. Then we may restrict to the generator of $\bfA$ in Definition \ref{defn:contraction}. In this case, the coincidence is evident by construction of $\alpha$ and $\beta$. Conversely, we compute $\alpha\circ\beta$. We may again restrict to $B^\theta\cup\frac{1}{\sqrt{t}}B^{-\theta}$. Let $\frac{a}{f^n}\in B^\theta\cup B^{-\theta}$. We may and do assume $n$ even and take $m$ as in the former paragraph. Then we have
		\[(\alpha\circ\beta)\left(\frac{a}{f^n}\right)=\alpha\left(\frac{a}{\bff^{n}}\right)
		=\frac{a}{f^{n}}\]
		if $f\in A^{\theta}$ and $\frac{a}{f^n}\in B^{\theta}$;
		\[\begin{split}
			(\alpha\circ\beta)\left(\frac{a}{f^n}\right)
			&=(\alpha\circ\beta)\left(\frac{af^m}{f^{n+m}}\right)\\
			&=\alpha\left(\frac{af^m\sqrt{t}^{n+m}}{\bff^{n+m}}\right)\\
			&=af^m\sqrt{t}^{n+m}\left(\frac{1}{f}\frac{1}{\sqrt{t}}\right)^{n+m}\\
			&=\frac{a}{f^n}
		\end{split}
		\]
		if $f\in A^{-\theta}$ and $\frac{a}{f^n}\in B^{\theta}$ since $\frac{1}{f}\in B^{-\theta}$;
		\[\begin{split}
			(\alpha\circ\beta)\left(\frac{a}{f^n}\frac{1}{\sqrt{t}}\right)
			&=(\alpha\circ\beta)\left(\frac{af^m}{f^{n+m}}\frac{1}{\sqrt{t}}\right)\\
			&=\alpha\left(\frac{af^m}{\sqrt{t}}\frac{1}{\bff^{n+m}}\right)\\
			&=\frac{af^m}{\sqrt{t}}\frac{1}{f^{n+m}}\\
			&=\frac{a}{f^n}\frac{1}{\sqrt{t}}
		\end{split}
		\]
		if $f\in A^{\theta}$ and $\frac{a}{f^n}\in B^{-\theta}$;
		\[\begin{split}
			(\alpha\circ\beta)\left(\frac{a}{f^n}\frac{1}{\sqrt{t}}\right)
			&=\alpha\left(\frac{af^m\sqrt{t}^{n+m-1}}{\bff^{n+m}}\right)\\
			&=af^m\sqrt{t}^{n+m-1}\left(\frac{1}{f}\frac{1}{\sqrt{t}}\right)^{n+m}\\
			&=\frac{a}{f^n}\frac{1}{\sqrt{t}}
		\end{split}
		\]
		if $f\in A^{-\theta}$ and $\frac{a}{f^n}\in B^{-\theta}$. Therefore $\beta\circ\alpha$ is the identity on the generator. This completes the proof.
	\end{proof}
	
	\begin{ex}\label{ex:sl_2}
		Assume that $k$ is a $\bZ\left[1/2,\sqrt{-1}\right]$-algebra. Following \cite[Corollary 13.33]{MR4225278}, identify $\bP^1_k$ with the moduli scheme of ordered pairs of generators of $k$. Namely, for a commutative $k$-algebra $R$, the $R$-point set $\bP^1_k(R)$ of $\bP^1_k$ is naturally identified with the set of isomorphism classes of locally free $R$-modules $L$ of rank $1$ and pairs $(a_1,a_2)$ of elements of $L$ such that $a_1$ and $a_2$ generate $L$ as $R$-modules. We define an involution $\theta$ on $\bP^1_k$ by
		$(L,a_1,a_2)\mapsto (L,a_2,-a_1)$.
		Define affine open immersions $\Spec k\left[w_i\right]\hookrightarrow \bP^1_k$ ($i\in \{1,2\}$) by
		\[\begin{array}{cc}
			(k\left[w_1\right],1+\sqrt{-1}w_1,\sqrt{-1}+w_1),&
			(k\left[w_2\right],w_2+\sqrt{-1},w_2\sqrt{-1}+1).
		\end{array}\]
		They are $\theta$-stable. Moreover, the induced involutions on the two affine lines are given by $\theta:w_i\mapsto -w_i$. The intersection of these affine open subschemes is \[\Spec k\left[w^{\pm 1}_i\right].\]
		The transition map
		$\Spec k\left[w^{\pm 1}_1\right]\cong\Spec k\left[w^{\pm 1}_2\right]$
		of the two affine open subschemes on their intersection is given by $w_1\mapsto w^{-1}_2$. The elements $w_i$ and $w^{-1}_i$ generate the coordinate rings $k\left[w^{\pm 1}_i\right]$, and they belong to $k\left[w^{\pm 1}_i\right]^{-\theta}$. Therefore we can glue up contraction of these affine lines. In fact, one can identify the contraction algebras of $(A_i,\theta)$ with $k\left[t,v_i\right]$ by $\frac{w_i}{\sqrt{t}}\mapsto v_i$. In view of Variant \ref{var:localization}, the gluing isomorphism is given by
		\[k\left[t,v_1\right]_{tv_1}\cong k\left[t,v_2\right]_{tv_2};~v_1\mapsto \frac{1}{tv_2}\]
		(use the equality $\frac{1}{\sqrt{t}w_2}=\frac{\sqrt{t}}{w_2}\frac{1}{t}$).
		
		The resulting scheme is the open subscheme of 
		\[\bfcB\coloneqq\Proj k\left[t,x,y,z\right]/(x^2+y^2-tz^2)\]
		obtained by removing the closed subscheme defined by
		$(k,0,0,1)$ at $t=0$. In fact, for a commutative $k$-algebra $R$, the $R$-point set of $\cB$ is identified with the set of isomorphism classes of locally free $R$-modules $L$ of rank $1$ and triples $(a_1,a_2,a_3)$ of generators of $L$ such that
		\[a_1\otimes a_1+a_2\otimes a_2-ta_3\otimes a_3=0\]
		in $L\otimes_R L$ (see \cite[Lemma 01NA]{stacks}). Define two open immersions
		$\Spec k\left[t,v_i\right]\to \bfcB$
		by
		\begin{flalign}
			(k\left[t,v_1\right],1-tv^2_1,-\sqrt{-1}(tv^2_1+1),2\sqrt{-1}v_1),
			\label{eq:local_immersion_1}\\
			(k\left[t,v_2\right],tv^2_2-1,-\sqrt{-1}(1+tv^2_2),2\sqrt{-1}v_2).
			\label{eq:local_immersion_2}
		\end{flalign}
		In fact, they are identified with the principal affine open subschemes attached to the homogeneous polynomials $x+\sqrt{-1}y$ and $x-\sqrt{-1}y$ respectively. These maps satisfy the gluing condition. Moreover, the resulting map is an isomorphism on $t\neq 0$. The maps at $t=0$ are the open immersion from the affine line $\Spec k\left[v_i\right]$ into $\Proj k\left[x,y,z\right]/(x^2+y^2)$ defined by 
		\[\begin{array}{cc}
			(k\left[v_1\right],1,-\sqrt{-1},2\sqrt{-1}v_1),
			&(k\left[v_2\right],-1,-\sqrt{-1},2\sqrt{-1}v_2)
		\end{array}\]
		in terms of a similar moduli description. We remark that $\Proj k\left[x,y,z\right]/(x^2+y^2)$ is identified with the union of two projective lines intersecting at a single $k$-point by
		\[x^2+y^2=(x+\sqrt{-1}y)(x-\sqrt{-1}y).\]
		Each of the open immersions constructed above is onto the affine line obtained by removing the intersection from one of the two projective lines. One can switch the two lines by replacing $\sqrt{-1}$ with $-\sqrt{-1}$.
		
		We remark that the nasty coefficients $-1$ and $\sqrt{-1}$ in \eqref{eq:local_immersion_1} and \eqref{eq:local_immersion_2} are for the compatibility with Example \ref{ex:SO(2,1)} (see Remark \ref{rem:relation_with_gluing} for the details).
	\end{ex}
	
	As an application of our description of $\bfA$, we prove corresponding results to \cite[Propositions 4.2 and 4.3]{MR3797197} in a purely algebraic way. Let $I$ be the ideal of $A$ generated by $A^{-\theta}$. We note that $\oplus_{n\geq 0} I^n/I^{n+1}$ is naturally equipped with the structure of a graded $A/I$-algebra.

	\begin{prop}\label{prop:fiber_at_t=0}
		\begin{enumerate}
			\renewcommand{\labelenumi}{(\arabic{enumi})}
			\item We have $I^n=(A^{-\theta})^{n}\oplus (A^{-\theta})^{n+1}$ for every nonnegative integer $n$.
			\item There exists a natural isomorphism $\bfA_0\cong \oplus_{n\geq 0} I^n/I^{n+1}$ of $k$-algebras.
		\end{enumerate}
	\end{prop}
	
	\begin{proof}
		For (1), observe that $I^n$ is the ideal of $A$ generated by $(A^{-\theta})^{n}$ for each nonnegative integer $n$. Part (1) thus follows from $A=A^{\theta}\oplus A^{-\theta}$.
		
		To prove (2), notice that $\bfA_0\cong\oplus_{n\geq 0} (A^{-\theta})^n/(A^{-\theta})^{n+2}$
		by the expression \eqref{eq:bfA}. One can identify $\oplus_{n\geq 0} (A^{-\theta})^n/(A^{-\theta})^{n+2}$ with the graded algebra $\oplus_{n\geq 0} I^n/I^{n+1}$ by (1). This completes the proof.
	\end{proof}

	\begin{thm}\label{thm:smooth}
		Suppose that $A$ is smooth over $k$.
		\begin{enumerate}
			\renewcommand{\labelenumi}{(\arabic{enumi})}
			\item Let $k'$ be an arbitrary commutative $k$-algebra. Let $B$ and $\eta$ be the base change of $A$ and $\theta$ to $k'$ respectively. Then we have $k'\otimes_k\bfA\cong \bfB$.
			\item There exists an isomorphism $\bfA_0\cong \Sym_{A/I} I/I^2$ of $k$-algebras.
			\item The $k\left[t\right]$-algebra $\bfA$ is smooth.
		\end{enumerate}
	\end{thm}
	
	Towards the proof, we give preliminary observations on extended Rees algebras. 
	
	\begin{lem}\label{lem:smoothI}
		Let $p:C\to D$ be a surjective homomorphism of smooth commutative algebras over a commutative ring $R$. Write $J=\Ker p$. 
		\begin{enumerate}
			\renewcommand{\labelenumi}{(\arabic{enumi})}
			\item Let $R'$ be an arbitrary commutative $R$-algebra. We denote the base change $R'\otimes_R p$ by $p':C'\to D'$. Put $J'=\Ker p'$. Then we have a canonical isomorphism $(J')^n\cong R'\otimes_R J^n$ for every nonnegative integer $n$.
			\item We have a canonical isomorphism $\oplus_{n\geq 0} J^n/J^{n+1}\cong \Sym_D J/J^2$ of graded $D$-algebras.
			\item For every nonnegative integer $n$, $J^n$ is flat as an $R$-module.
			\item The extended Rees algebra $C\left[u,Ju^{-1}\right]$ is smooth over $R\left[u\right]$ if $R$ is Noetherian. 
		\end{enumerate}
	\end{lem}
	
	\begin{proof}
		The morphism $\Spec p$ is a regular immersion by \cite[Th\'eor\`eme (17.12.1)]{MR0238860}. Part (2) then follows from \cite[Proposition (16.9.8)]{MR0238860}.
		
		We next prove (1) and (3) simultaneously by induction on $n$. The assertions are clear if $n=0$. Assume $n\geq 1$. In view of (2), we have a canonical short exact sequence
		\begin{equation}
			0\to J^n\to J^{n-1}\to \Sym^{n-1}_D J/J^2\to 0.\label{eq:short_exact}
		\end{equation}
		Since $J/J^2$ is finitely generated and projective as a $D$-module from \cite[Proposition (16.9.8)]{MR0238860}, so is $\Sym^{n-1}_D J/J^2$ as a $D$-module. Since $D$ is flat over $R$, $\Sym^{n-1}_D J/J^2$ is a flat $R$-module. Since $J^{n-1}$ is a flat $R$-module by the induction hypothesis, so is $J^n$. The base change
		\[0\to R'\otimes_RJ^n\hookrightarrow R'\otimes_RJ^{n-1}\to R'\otimes_R \Sym^{n-1}_D J/J^2
		\to 0\]
		of the sequence \eqref{eq:short_exact} is exact since $\Sym^{n-1}_D J/J^2$ is flat over $R$. The middle term is identified with $(J')^{n-1}$ by the induction hypothesis. One can proceeds the induction by showing $R'\otimes_R \Sym^{n-1}_D J/J^2\cong (J')^{n-1}/(J')^n$. It is evident that $R'\otimes_R \Sym^{n-1}_D J/J^2\cong \Sym^{n-1}_{D'} R'\otimes_R (J/J^2)$. To compute $R'\otimes_R (J/J^2)$, recall that we have a short exact sequence
		\begin{equation}
			0\to J/J^2\to D\otimes_C\Omega^1_{C/R}\to \Omega^1_{D/R}\to 0.\label{eq:short_exact_2}
		\end{equation}
		Since $D$ is smooth over $R$, $\Omega^1_{D/R}$ is flat over $R$. Therefore the base change
		\[0\to R'\otimes_R (J/J^2)\to D'\otimes_{C'}\Omega^1_{C'/R'}\to \Omega^1_{D'/R'}\to 0\]
		of \eqref{eq:short_exact_2} is exact. Therefore we obtain $R'\otimes_R (J/J^2)'\cong J'/(J')^2$ and
		\[R'\otimes_R(J^{n-1}/J^n)
		\cong \Sym^{n-1}_{D'} (J'/(J')^2).
		\]
		Apply (2) to $p'$ to get $R'\otimes_R(J^{n-1}/J^n)\cong (J')^{n-1}/(J')^n$ as desired.
		
		Finally, we see (4). The proof will be completed by checking the following conditions:
		\begin{enumerate}
			\renewcommand{\labelenumi}{(\roman{enumi})}
			\item $C\left[u,Ju^{-1}\right]\otimes_{R\left[u\right]} R\left[u^{\pm 1}\right]$ is smooth over $R\left[u^{\pm 1}\right]$.
			\item $C\left[u,Ju^{-1}\right]\otimes_{R\left[u\right]} R\left[u\right]/(u)$ is smooth over $ R\left[u\right]/(u)$.
			\item $C\left[u,Ju^{-1}\right]$ is flat over $R\left[u\right]$.
			\item $C\left[u,Ju^{-1}\right]$ is finitely presented over $R\left[u\right]$ if $R$ is Noetherian.
		\end{enumerate}
		
		Condition (i) holds since $C\left[u,Ju^{-1}\right]\otimes_{R\left[u\right]} R\left[u^{\pm 1}\right]$ is identified with $C\left[u^{\pm 1}\right]$.
		
		For (ii), identify $C\left[u,Ju^{-1}\right]\otimes_{R\left[u\right]} R\left[u\right]/(u)$ and $R\left[u\right]/(u)$ with $\Sym_D J/J^2$ and $R$ respectively (recall (2)). Notice that $\Sym_D J/J^2$ is a smooth $D$-algebra by \cite[Proposition (16.9.8)]{MR0238860}. Since $D$ is smooth over $R$, so is $\Sym_{D} J/J^2$.
		
		To prove (iii), set
		$F_nC\left[u,Ju^{-1}\right]=C\left[u\right]\oplus \oplus_{1\leq i\leq n} J^i u^{-i}
		\subset C\left[u,Ju^{-1}\right]$
		for each nonnegative integer $n$. The passage to the direct limit reduces the proof to showing that $F_nC\left[u,Ju^{-1}\right]$ is flat as an $R\left[u\right]$-module for every $n$. For $0\leq m\leq n$, set
		$F^m_nC\left[u,Ju^{-1}\right]=\oplus_{i\geq -m} J^m u^{i}\oplus 
		\oplus_{m+1\leq i\leq n} J^i u^{-i}$.
		Then we have
		\[F^n_nC\left[u,Ju^{-1}\right]=\oplus_{i\geq -n} J^n u^{i},\]
		\[\begin{split}
			F^m_nC\left[u,Ju^{-1}\right]/F^{m+1}_nC\left[u,Ju^{-1}\right]&\cong \oplus_{i\geq -m} J^m/J^{m+1} u^i\\
			&\cong J^m/J^{m+1}\otimes_{D} D\left[u\right]\\
			&\cong \Sym^m_D J/J^2\otimes_D D\left[u\right]
		\end{split}\]
		for $0\leq m\leq n-1$ by (2). Since $\Sym^m_D J/J^2$ is a flat $D$-module, \[F^m_nC\left[u,Ju^{-1}\right]/F^{m+1}_nC\left[u,Ju^{-1}\right]\]
		is flat as a $D\left[u\right]$-module. Since $D$ is flat over $R$, $F^m_nC\left[u,Ju^{-1}\right]/F^{m+1}_nC\left[u,Ju^{-1}\right]$ is flat as an $R\left[u\right]$-module. We have also proved in (3) that $J^m$ is a flat $R$-module for $0\leq m\leq n$ by (3). We thus conclude that
		\[\begin{array}{cc}
			F^n_nC\left[u,Ju^{-1}\right],& F^m_nC\left[u,Ju^{-1}\right]/F^{m+1}_nC\left[u,Ju^{-1}\right]
		\end{array}\]
		are flat as $R\left[u\right]$-modules. A descending induction implies \[F^0_nC\left[u,Ju^{-1}\right]=F_nC\left[u,Ju^{-1}\right]\]
		is flat as an $R\left[u\right]$-module.
		
		Finally, we prove (iv). Assume that $R$ is Noetherian. In view of Hilbert's basis theorem, it suffices to show that $C\left[u,Ju^{-1}\right]$ is finitely generated over $R\left[u\right]$. It is evident by definition that $C\left[u,Ju^{-1}\right]$ is generated by $C$ and $Jt^{-1}$ as an $R\left[u\right]$-algebra. Since $C$ is of finite type over $R$, one can find a finite set $\{c_i\}$ of generators of $C$ over $R$. Since $R$ is Noetherian, so is $C$ by Hilbert's basis theorem. Therefore one can find a finite set $\{c'_j\}$ of generators of $J$ as a $C$-module. It is easy to show that $\{c_i\} \cup\{c'_j u^{-1}\}$ generates $C\left[u,Ju^{-1}\right]$ as an $R\left[u\right]$-algebra. In particular, $C\left[u,Ju^{-1}\right]$ is of finite type over $R\left[u\right]$.
	\end{proof}
	
	We next make an attempt to remove the Noether hypothesis in Lemma \ref{lem:smoothI} (4). To clarify the scheme-theoretic considerations below, we interpret and generalize the former lemma to a scheme-theoretic statement.
	
	\begin{lem}\label{lem:smoothII}
		Let $S$ be a scheme, and $j:Z\hookrightarrow Y$ be a closed immersion of smooth $S$-schemes with the ideal sheaf $\cJ\subset\cO_Y$. Then $\Spec \cO_Y\left[u,\cJ u^{-1}\right]$ is smooth over $\Spec\cO_S\left[u\right]$.
	\end{lem}
	
	Since $\cJ$ is a quasi-coherent $\cO_Y$-module, and the collection of quasi-coherent $\cO_Y$-modules is closed under formation of small colimits and finite limits,
	the extended Rees algebra $\cO_Y\left[u,\cJ u^{-1}\right]$ of $(\cO_Y,\cJ)$
	is a quasi-coherent $\cO_Y$-algebra. Therefore $\Spec \cO_Y\left[u,\cJ u^{-1}\right]$ makes a sense.
	
	\begin{proof}
		It is easy to show that $\cJ^n$ respects the flat base changes of $Y$ for $n\geq 0$. Namely, suppose that we are given a Cartesian diagram
		\[\begin{tikzcd}
			Z'\ar[r, "j'"]\ar[d, "p'"']&Y'\ar[d, "p"]\\
			Z\ar[r, "j"]&Y,
		\end{tikzcd}\] 
		where $p$ is flat. Let $\cJ'\subset\cO_{Y'}$ be the ideal attached to the closed immersion $j'$. Then we have a canonical isomorphism $p^\ast \cJ^n\cong (\cJ')^n$ for $n\geq 0$. We may therefore work locally in the \'etale (flat) topology of $Y$. In particular, we may work locally in the Zariski topology of $S$ to assume that $S$ is affine.
		
		Let $U$ be the complementary open subscheme to $Z$ in $Y$. Then we have
		\[\cO_U\left[u,\cJ|_U u^{-1}\right]\cong \cO_U\left[u^{\pm 1}\right]\]
		since $\cJ|_U=\cO_Y|_U$. This implies that $\Spec\cO_Y\left[u,\cJ u^{-1}\right]\times_Y U$ is smooth over $\Spec\cO_S\left[u\right]$. We next work on an open neighborhood of $Z$ in $Y$. In this case, we may replace $Z\hookrightarrow Y$ with the standard closed immersion
		\[\Spec \cO_S\left[u_1,\ldots,u_m\right]\hookrightarrow
		\Spec \cO_S\left[u_1,u_2,\ldots,u_n\right]
		\]
		of affine spaces over $S$ for some $m\leq n$ by \cite[Corollaire (17.12.2)]{MR0238860}. In this case, we may replace $S$ with $\Spec\bZ$ by Lemma \ref{lem:smoothI} (1) since we are currently assuming $S$ to be affine. The assertion now follows from Lemma \ref{lem:smoothI} (4).
	\end{proof}
	
	\begin{proof}[Proof of Theorem \ref{thm:smooth}]
		Recall that $A/I$ is a smooth $k$-algebra by \cite[Proof of Lemma 3.1.1]{MR4627704}. Then (2) follows from Proposition \ref{prop:fiber_at_t=0} and Lemma \ref{lem:smoothI} (2).
		
		To prove (1) and (3), observe that we have an isomorphism
		\begin{equation}
			\bfA\otimes_{k\left[t\right]} k\left[\sqrt{t}\right]\cong
			A\left[\sqrt{t},I\sqrt{t}^{-1}\right]\label{eq:rees}
		\end{equation}
		by a similar argument to Lemma \ref{lem:relation} (3) (use the equalities of \eqref{eq:bfA} and Proposition \ref{prop:fiber_at_t=0} (1)). For (1), we also note that there is a canonical homomorphism $k'\otimes_k\bfA\to \bfB$ by definitions. Part (1) (resp.~(3)) now follows from Lemma \ref{lem:smoothI} (1) (resp.~Lemma \ref{lem:smoothII}) through the faithfully flat descent with respect to $k\left[t\right]\subset k\left[\sqrt{t}\right]$.
	\end{proof}
	
	As we can easily see in examples, the fiber of $\bfX$ at $t=0$ is disconnected in general even if $X$ is connected. This may sometimes cause difficulties in analysis of geometric structures of $\bfX$ (cf.~Section \ref{sec:quotient}). We can resolve this problem in practice by killing additional components at $t=0$. For generalities on this refined object, let us note that one can literally generalize the argument of \cite[Section 5.12]{MR0228502} to obtain the following assertion:
	
	\begin{lem}\label{lem:affine}
		Let $j:V\hookrightarrow Y$ be an open immersion of $k\left[u\right]$-schemes satisfying the following properties:
		\begin{enumerate}
			\renewcommand{\labelenumi}{(\roman{enumi})}
			\item $j\otimes_{k\left[u\right]} k\left[u^{\pm 1}\right]$ is an isomorphism.
			\item The map $j\otimes_{k\left[u\right]} k\left[u\right]/(u)$ is a closed immersion.
		\end{enumerate}
		Then $j$ is an affine open immersion.
	\end{lem}
	
	For its applications to $\bfX$, let $x_0$ be a $\theta$-invariant $k$-point of $X$. Regard $x_0$ as a homomorphism $A\to k$ and take its contraction to obtain a $k\left[t\right]$-point $\bfx_0$ of $\bfX$. If $X$ is smooth, we can apply \cite[Corollaire (15.6.5)]{MR0217086} to $(\bfX,\bfx_0)$ to obtain an open subscheme of $\bfX$ since $\bfX$ is smooth and affine over $k\left[t\right]$ in this case. We will denote this open subscheme by $\bfX^\circ$. We define $(X^\theta)^\circ$ in a similar way. If $X=G$ is a group $k$-scheme, we will always assume $x_0$ to be the unit section in this paper. In particular, these notations will not conflict with \cite[D\'efinition 3.1]{MR0234961}.
	
	\begin{cor}\label{cor:affine}
		Assume the following conditions:
		\begin{enumerate}
			\renewcommand{\labelenumi}{(\roman{enumi})}
			\item $X$ is smooth over $k$.
			\item The fibers of $X$ are connected.
			\item The open subscheme $(X^\theta)^\circ\subset X^\theta$ is closed. 
		\end{enumerate}
		Then the open subscheme $\bfX^\circ\subset\bfX$ attached to $\bfx_0$ is affine.
	\end{cor}
	
	We end this section with study of preservation of tensor products.
	
	\begin{cons}
		For commutative $k$-algebras with involutions $(A,\theta)$ and $(B,\eta)$, define a canonical homomorphism
		\[i_{A,B}:\bfA\otimes_{k\left[t\right]} \bfB\to A\left[\sqrt{t}^{\pm 1}\right]\otimes_{k\left[\sqrt{t}^{\pm 1}\right]}
		B \left[\sqrt{t}^{\pm 1}\right]\]
		by the componentwise embedding:
		\[\begin{split}
			\bfA\otimes_{k\left[t\right]} \bfB
			&\to (\bfA\otimes_{k\left[t\right]} \bfB)\otimes_{k\left[t\right]}
			k\left[\sqrt{t}^{\pm 1}\right]\\
			&\cong \left(\bfA\otimes_{k\left[t\right]} k\left[\sqrt{t}^{\pm 1}\right]\right)
			\otimes_{k\left[\sqrt{t}^{\pm 1}\right]} 
			\left(\bfB\otimes_{k\left[t\right]} k\left[\sqrt{t}^{\pm 1}\right]\right)\\
			&\cong A\left[\sqrt{t}^{\pm 1}\right]\otimes_{k\left[\sqrt{t}^{\pm 1}\right]}
			B \left[\sqrt{t}^{\pm 1}\right],
		\end{split}\]
		where the first map is given by the unit. We define 
		\[i_{A,B,C}:
		\bfA\otimes_{k\left[t\right]} \bfB \otimes_{k\left[t\right]}\bfC
		\to A\left[\sqrt{t}^{\pm 1}\right]\otimes_{k\left[\sqrt{t}^{\pm 1}\right]}
		B \left[\sqrt{t}^{\pm 1}\right]\otimes_{k\left[\sqrt{t}^{\pm 1}\right]}
		C \left[\sqrt{t}^{\pm 1}\right]
		\]
		for an additional commutative $k$-algebra with involution $(C,\zeta)$ in a similar way.
	\end{cons}

	\begin{cond}\label{cond:flat}
		The structure homomorphism $k\left[t\right]\to\bfA$ is flat.
	\end{cond}
	
	\begin{ex}\label{ex:flat_smooth}
		Condition \ref{cond:flat} holds if $A$ is smooth over $k$ by Theorem \ref{thm:smooth}
	\end{ex}
	
	\begin{ex}\label{ex:flat_field}
		Condition \ref{cond:flat} holds if $k=F$ is a field. In fact, regard $\bfA$ as an $F\left[t\right]$-submodule of $A\left[\sqrt{t}^{\pm 1}\right]$ to see that $\bfA$ is a torsion-free $F\left[t\right]$-module. Since $F\left[t\right]$ is a PID, $\bfA$ is flat over $F\left[t\right]$.
	\end{ex}
	
	\begin{ex}\label{ex:flatbc}
		Suppose that Condition \ref{cond:flat} is satisfied. Then for a flat homomorphism $k\to k'$ of commutative rings, the contraction algebra for $A\otimes_k k'$ satisfies Condition \ref{cond:flat} by Proposition \ref{prop:flatbc}.
	\end{ex}

	\begin{prop}\label{prop:monoidal}
		For $i\in\{1,2,3\}$, let $(A_i,\theta_i)$ be commutative $k$-algebras with involutions. Set
		$(A_{ij},\theta_{ij})=(A_i\otimes_k A_j,\theta_i\otimes_k \theta_j)$
		for $(i,j)\in\{1,2,3\}^2$, and
		\[(A_{123},\theta_{123})=(A_1\otimes_k A_2\otimes_k A_3,
		\theta_1\otimes_k \theta_2\otimes_k \theta_3).\]
		\begin{enumerate}
			\renewcommand{\labelenumi}{(\arabic{enumi})}
			\item The composition of the canonical isomorphism
			\begin{equation}
				A_1\left[\sqrt{t}^{\pm 1}\right]\otimes_{k\left[\sqrt{t}^{\pm 1}\right]}
				A_2 \left[\sqrt{t}^{\pm 1}\right]
				\cong A_{12}\left[\sqrt{t}^{\pm 1}\right]\label{eq:monoidal}
			\end{equation}
			with $i_{A_1,A_2}$ is surjective onto $\bfA_{12}$. We denote the resulting map \[\bfA_1\otimes_{k\left[t\right]} \bfA_2\twoheadrightarrow\bfA_{12}\]
			by $I_{A_1,A_2}$.
			\item Th map $I_{A_1,A_2}$ is natural in both $A_1$ and $A_2$.
			\item The compositions of $i_{A_1,A_2,A_3}$ with the canonical isomorphisms
			\begin{flalign*}
				&A_1\left[\sqrt{t}^{\pm 1}\right]\otimes_{k\left[\sqrt{t}^{\pm 1}\right]}
				A_2 \left[\sqrt{t}^{\pm 1}\right]\otimes_{k\left[t\right]} A_3\left[\sqrt{t}^{\pm 1}\right]\\
				&\cong A_{12}\left[\sqrt{t}^{\pm 1}\right]\otimes_{k\left[\sqrt{t}^{\pm 1}\right]}
				A_3 \left[\sqrt{t}^{\pm 1}\right]\\
				&\cong A_{123}\left[\sqrt{t}^{\pm 1}\right],
			\end{flalign*}
			\begin{flalign*}
				&A_1\left[\sqrt{t}^{\pm 1}\right]\otimes_{k\left[\sqrt{t}^{\pm 1}\right]}
				A_2 \left[\sqrt{t}^{\pm 1}\right]\otimes_{k\left[t\right]} A_3\left[\sqrt{t}^{\pm 1}\right]\\
				&\cong A_1\left[\sqrt{t}^{\pm 1}\right]\otimes_{k\left[\sqrt{t}^{\pm 1}\right]}
				A_{23} \left[\sqrt{t}^{\pm 1}\right]\\
				&\cong A_{123}\left[\sqrt{t}^{\pm 1}\right]
			\end{flalign*}
			coincide with the compositions of the surjective maps
			\[\bfA_1\otimes_{k\left[t\right]} \bfA_2\otimes_{k\left[t\right]} \bfA_3
			\xrightarrow{I_{A_1,A_2}\otimes_{k\left[t\right]} \bfA_3}
			\bfA_{12}\otimes_{k\left[t\right]} \bfA_3
			\overset{I_{A_{12},A_3}}{\to} \bfA_{123},\]
			\[\bfA_1\otimes_{k\left[t\right]} \bfA_2\otimes_{k\left[t\right]} \bfA_3
			\xrightarrow{\bfA_1\otimes_{k\left[t\right]} I_{A_2,A_3}}
			\bfA_1\otimes_{k\left[t\right]} \bfA_{23}
			\overset{I_{A_{1},A_{23}}}{\to} \bfA_{123}\]
			with the inclusion
			$\bfA_{123}\subset A_{123}\left[\sqrt{t}^{\pm 1}\right]$
			respectively.
			\item Let $C:A_{12}\cong A_{21}$ denote the canonical isomorphism. Then the diagram
			\[\begin{tikzcd}
				\bfA_1\otimes_{k\left[t\right]} \bfA_2\ar[rr, "\sim"]\ar[d, "I_{A_1,A_2}"']
				&&\bfA_2\otimes_{k\left[t\right]} \bfA_1
				\ar[d, "I_{A_2,A_1}"]\\
				\bfA_{12}\ar[rr, "\bfC"]\ar[d, hook]
				&&\bfA_{21}\ar[d, hook]\\
				A_{12}\left[\sqrt{t}^{\pm 1}\right]
				\ar[rr, "{C\otimes_k k\left[\sqrt{t}^{\pm 1}\right]}"]
				&&A_{21}\left[\sqrt{t}^{\pm 1}\right]
			\end{tikzcd}\]
			commutes, where the upper horizontal arrow is the canonical isomorphism.
			\item If one of $(A_1,\theta_1)$ and $(A_2,\theta_2)$ satisfies Condition \ref{cond:flat} then $i_{A_1,A_2}$ is injective. In particular, $I_{A_1,A_2}$ is an isomorphism of $k\left[t\right]$-algebras.
			\item If two of $(A_1,\theta_1)$, $(A_2,\theta_2)$, and $(A_3,\theta_3)$ satisfy Condition \ref{cond:flat} then $i_{A_1,A_2,A_3}$ is injective.
			\item Put $A_2=k$ and $\theta_2=\id_k$. Write
			\[\begin{array}{c}
				r:A_1\otimes_k k\cong A_1\\
				l:k\otimes_k A_3\cong A_3\\
				\mu_{A_1,k}:A_{1} \left[\sqrt{t}^{\pm 1}\right]
				\otimes_{k\left[\sqrt{t}^{\pm 1}\right]} k\left[\sqrt{t}^{\pm 1}\right] \cong A_{12}\left[\sqrt{t}^{\pm 1}\right]\\
				\mu_{k,A_3}:k\left[\sqrt{t}^{\pm 1}\right]\otimes_{k\left[\sqrt{t}^{\pm 1}\right]}
				A_{3} \left[\sqrt{t}^{\pm 1}\right]\cong A_{23}\left[\sqrt{t}^{\pm 1}\right]
			\end{array}\]
			for the canonical isomorphisms. Then the diagrams
			\[\begin{tikzcd}
				\bfA_1\otimes_{k\left[t\right]} k\left[t\right]
				\ar[r, "I_{A_1,k}"]\ar[d, hook, "i_{A_1,k}"']
				&\bfA_{12}\ar[r, "\bfr"]\ar[d,hook]&\bfA_1\ar[d, hook]\\
				A_{1} \left[\sqrt{t}^{\pm 1}\right]
				\otimes_{k\left[\sqrt{t}^{\pm 1}\right]} k\left[\sqrt{t}^{\pm 1}\right] 
				\ar[r, "\mu_{A_1,k}"]
				&A_{12}\left[\sqrt{t}^{\pm 1}\right]
				\ar[r, "{r\left[\sqrt{t}^{\pm 1}\right]}"]&A_1\left[\sqrt{t}^{\pm 1}\right]
			\end{tikzcd}\]
			\[\begin{tikzcd}
				k\left[t\right]\otimes_{k\left[t\right]} \bfA_3
				\ar[r, "I_{k,A_3}"]\ar[d, hook, "i_{k,A_3}"']
				&\bfA_{23}\ar[r, "\bfl"]\ar[d,hook]&\bfA_3\ar[d, hook]\\
				k\left[\sqrt{t}^{\pm 1}\right]\otimes_{k\left[\sqrt{t}^{\pm 1}\right]}
				A_{3} \left[\sqrt{t}^{\pm 1}\right]\ar[r, "\mu_{k,A_3}"]
				&A_{23}\left[\sqrt{t}^{\pm 1}\right]
				\ar[r, "{l\left[\sqrt{t}^{\pm 1}\right]}"]&A_3\left[\sqrt{t}^{\pm 1}\right]
			\end{tikzcd}\]
			commute.
		\end{enumerate}
	\end{prop}
	
	\begin{proof}
		For (1), it will suffice to compare the generators 
		\[A^{\theta_1}_1\otimes 1,~\frac{1}{\sqrt{t}}A^{-\theta_1}_1\otimes 1,~
		1\otimes A^{\theta_2}_2,
		~\frac{1}{\sqrt{t}}A^{-\theta_1}\otimes \frac{1}{\sqrt{t}}A^{-\theta_2}_2\]
		of $\bfA_1\otimes_{k\left[t\right]} \bfA_2$ and
		\[A^{\theta_1}_1\otimes A^{\theta_2}_2,~A^{-\theta_1}_1\otimes A^{-\theta_2}_2,~
		\frac{1}{\sqrt{t}}A^{-\theta_1}_1\otimes A^{\theta_2}_2,~
		A^{\theta_1}_1\otimes \frac{1}{\sqrt{t}}A^{-\theta_2}_2\]
		of $\bfA_{12}$. Since an involution fixes the unit in general, we have
		\[A^{\theta_1}_1\otimes 1\subset A^{\theta_1}_1\otimes A^{\theta_2}_2,~
		1\otimes A^{\theta_2}_2\subset A^{\theta_1}_1\otimes A^{\theta_2}_2\]
		\[\frac{1}{\sqrt{t}}A^{-\theta_1}_1\otimes 1\subset 
		\frac{1}{\sqrt{t}}A^{-\theta_1}_1\otimes A^{\theta_2}_2,~
		1\otimes A^{\theta_2}_2\subset A^{\theta_1}_1\otimes \frac{1}{\sqrt{t}}A^{-\theta_2}_2\]
		\[\frac{1}{\sqrt{t}}A^{-\theta_1}_1\otimes \frac{1}{\sqrt{t}}A^{-\theta_2}_2
		\subset \left(\frac{1}{\sqrt{t}}A^{-\theta_1}_1\otimes A^{\theta_2}_2\right)
		\cdot\left(A^{\theta_1}_1\otimes \frac{1}{\sqrt{t}}A^{-\theta_2}_2\right).
		\]
		These formulas of containment show that the composite map of \eqref{eq:monoidal} with $i_{A_1,A_2}$ factors through $\bfA_{12}$. This map is onto $\bfA_{12}$ from
		\[A^{\theta_1}_1\otimes A^{\theta_2}_2=(A^{\theta_1}_1\otimes 1)\cdot (1\otimes A^{\theta_2}_2)\]
		\[A^{-\theta_1}_1\otimes A^{-\theta_2}_2
		=t\cdot \frac{1}{\sqrt{t}}A^{-\theta}\otimes \frac{1}{\sqrt{t}}A^{-\theta_2}_2
		\]
		\[\frac{1}{\sqrt{t}}A^{-\theta_1}_1\otimes A^{\theta_2}_2
		=\left(\frac{1}{\sqrt{t}}A^{-\theta_1}_1\otimes 1\right)
		\cdot(1\otimes A^{\theta_2}_2)\]
		\[A^{\theta_1}_1\otimes \frac{1}{\sqrt{t}}A^{-\theta_2}_2
		=\left(A^{\theta_1}_1\otimes 1\right)\cdot
		\left(1\otimes \frac{1}{\sqrt{t}}A^{-\theta_2}_2\right).
		\]
		Part (2) follows by seeing values of the map in $A_{12}\left[\sqrt{t}^{\pm 1}\right]$.
		
		Parts (3), (4), and (7) are evident by definitions. Part (6) will follow from (3) and (5). Therefore the proof will be completed by showing (5). The map $i_{A_1,A_2}$ is identified with the composition of the sequence of canonical maps
		\[\begin{split}
			\bfA_1\otimes_{k\left[t\right]} \bfA_2
			&\to \bfA_1\otimes_{k\left[t\right]} A_2\left[\sqrt{t}^{\pm 1}\right]\\
			&\cong \left(\bfA_1\otimes_{k\left[t\right]} k\left[\sqrt{t}^{\pm 1}\right]\right)
			\otimes_{k\left[\sqrt{t}^{\pm 1}\right]}
			A_2 \left[\sqrt{t}^{\pm 1}\right]\\
			&\cong A_1\left[\sqrt{t}^{\pm 1}\right]\otimes_{k\left[\sqrt{t}^{\pm 1}\right]}
			A_2 \left[\sqrt{t}^{\pm 1}\right].
		\end{split}\]
		The first map is the base change of the inclusion $\bfA_2\hookrightarrow A_2\left[\sqrt{t}^{\pm 1}\right]$. The second map is the canonical isomorphism using the $k\left[\sqrt{t}^{\pm 1}\right]$-algebra structure of $A_2\left[\sqrt{t}^{\pm 1}\right]$. The isomorphism in the third row is obtained by Lemma \ref{lem:relation}. If $(A_1,\theta_1)$ satisfies Condition \ref{cond:flat}, then the first map is injective. This proves that $i_{A_1,A_2}$ is injective if $(A_1,\theta_1)$ satisfies Condition \ref{cond:flat}. One can see that $i_{A_1,A_2}$ is injective if $(A_2,\theta_2)$ satisfies Condition \ref{cond:flat} in a similar way. This proves (5).
	\end{proof}
	
	\begin{cor}\label{cor:monoidal}
		The assignment $(A,\theta)\rightsquigarrow \bfA$ determines a lax symmetric monoidal functor from the symmetric monoidal category of commutative $k$-algebras with involutions to that of $k\left[t\right]$-algebras. Moreover, it restricts to a symmetric monoidal functor from the symmetric monoidal category of commutative $k$-algebras with involutions satisfying Condition \ref{cond:flat} to that of flat $k\left[t\right]$-algebras.
	\end{cor}
	
	\begin{proof}
		We can reduce the first assertion to the fact that $-\otimes_k k\left[\sqrt{t}^{\pm 1}\right]$ is symmetric monoidal by seeing the relations in the Laurent polynomial algebras of tensor products of commutative $k$-algebras with variable $\sqrt{t}$. For the unitality, recall Example \ref{ex:trivial} if necessary. For the latter statement, suppose that we are given pairs $(A,\theta)$ and $(B,\eta)$ of commutative $k$-algebras with involutions satisfying Condition \ref{cond:flat}. Put $(C,\zeta)=(A\otimes_k B,\theta\otimes_k\eta)$. Then we have a natural isomorphism $\bfA\otimes_{k\left[t\right]} \bfB\cong \bfC$. Since $\bfA$ and $\bfB$ are flat over $k\left[t\right]$, so is $\bfC$. In other words, $(A\otimes_k B,\theta\otimes_k\eta)$ satisfies Condition \ref{cond:flat}. The latter assertion now follows from the former one.
	\end{proof}
	
	\begin{rem}[{\cite[Remark 3.2.1]{MR4130851}}]\label{rem:gluing}
		If one would like to work over the projective line $\bP^1_k$ over $k$, define the corresponding contraction schemes over the two principal open affine lines in $\bP^1_k$. Let $\bar{A}_t$ be the $k\left[t^{\pm 1}\right]$-algebra obtained from $A_t$ by switching the action of $t$ with $t^{-1}$. Then the map $\bar{A}_t\to A_t$ defined by
		\[\begin{array}{cc}
			at^n\mapsto at^{-n}&(a\in A^{\theta})\\
			a\frac{t^n}{\sqrt{t}}\mapsto a\frac{t^{-n+1}}{\sqrt{t}}&(a\in A^{-\theta})
		\end{array}\]
		is a $k\left[t^{\pm 1}\right]$-algebra isomorphism (cf.~\cite[the last line of Section 3]{MR3797197}). Using this map, one can glue the schemes over the two principal open affine lines of $\bP^1_k$ to obtain a scheme over $\bP^1_k$.
	\end{rem}
	
	\begin{rem}
		If one wishes to work with $n$-contraction over the projective $n$-scheme $\bP^n_k$, work on the principal open affine $n$-spaces in $\bP^n_k$, and glue up the contraction schemes on them in a similar way to Remark \ref{rem:gluing}.
	\end{rem}

	\section{Hopf structure}\label{sec:hopf}
	
	Let $k$ be a commutative $\bZ\left[1/2\right]$-algebra. Let $A$ be a commutative Hopf $k$-algebra with an involution. Let
	\[\begin{array}{ccc}
		\Delta:A\to A\otimes_k A,&\epsilon:A\to k,&i:A\to A\
	\end{array}\]
	be the comultiplication, counit, and antipode of $A$ respectively. We will use the Sweedler notation of \cite[Section 1.2]{MR0252485}.
	
	In this section, we put the structure of a Hopf $k\left[t\right]$-algebra on $\bfA$. We also give a group theoretic analog of Corollary \ref{cor:comparisonwithbhs}, which is rather a straightforward generalization of \cite[Definition 4.2]{MR3797197}. We follow the notations in the former section, but write $G=\Spec A$. 
	
	The $k\left[\sqrt{t}^{\pm 1}\right]$-algebra $A\left[\sqrt{t}^{\pm 1}\right]$ is equipped with the structure of a Hopf algebra by the base change from $A$. One can put the structure of a Hopf algebra over $k\left[t^{\pm 1}\right]$ on $A_t$ by the Galois descent (cf.~Proposition \ref{prop:fiber_descent}, \cite[Theorem A.3]{hayashikgb}). In fact, one can define the structure homomorphisms by restriction from $A\left[\sqrt{t}^{\pm 1}\right]$ to $A_t$. We wish to restrict the structure homomorphisms of Hopf algebras on $A\left[\sqrt{t}^{\pm 1}\right]$ and $A_t$ to $\bfA$. This follows as a formal consequence of Corollary \ref{cor:monoidal} under Condition \ref{cond:flat}. Let us note below how we can construct the structure homomorphisms.
	
	\begin{lem}\label{lem:counit}
		The counit $\epsilon$ is zero on $A^{-\theta}$.
	\end{lem}
	
	\begin{proof}
		This is straightforward: For $a\in A^{-\theta}$, we have
		\[\epsilon(a)=\frac{1}{2}\epsilon(a-\theta(a))=\frac{1}{2}(\epsilon(a)-\epsilon(a))=0.\]
	\end{proof}
	
	\begin{cons}[Counit]\label{cons:counit}
		The map $i\otimes_k k\left[\sqrt{t}^{\pm 1}\right]$ restricts to a $k\left[t\right]$-algebra homomorphism
		$\bfepsilon:\bfA\to k\left[t\right]\subset k\left[\sqrt{t}^{\pm 1}\right]$
		by Lemma \ref{lem:counit}. Explicitly, $\bfepsilon$ is computed by
		\[\bfepsilon|_{A^{\theta}\left[t\right]}=\epsilon|_{A^{\theta}}\otimes_k k\left[t\right]\]
		\[\bfepsilon|_{\left(\oplus_{n\geq 0}A^{-\theta}\frac{t^n}{\sqrt{t}}\right)
			\oplus \oplus_{n\geq 2} \frac{1}{\sqrt{t}^n} (A^{-\theta})^n}=0.\]
	\end{cons}
	
	\begin{cons}[Antipode]
		Since $i$ commutes with $\theta$, $i$ respects $A^{\theta}$ and $A^{-\theta}$. Hence $i\otimes_k k\left[\sqrt{t}^{\pm 1}\right]$ naturally restricts to a $k\left[t\right]$-algebra automorphism $\bfi$ of $\bfA$ (recall the definition of $\bfA$).
	\end{cons}
	
	In the rest of this section, we assume Condition \ref{cond:flat}.
	
	\begin{cons}[Comultiplication]
		Put $(B,\eta)=(A\otimes_k A,\theta\otimes_k \theta)$. Then we have a commutative diagram
		\[\begin{tikzcd}
			\bfA\ar[rr, "\bfDelta"]\ar[d, hook]
			&&\bfB\ar[d, hook]&\bfA\otimes_{k\left[t\right]} \bfA\ar[l, "I_{A,A}"', "\sim"]
			\ar[d, hook, "i_{A,A}"]\\
			A\left[\sqrt{t}^{\pm 1}\right]
			\ar[rr, "{\Delta\otimes_{k}k\left[\sqrt{t}^{\pm 1}\right]}"]
			&&B\left[\sqrt{t}^{\pm 1}\right]
			&A\left[\sqrt{t}^{\pm 1}\right]\otimes_{k\left[\sqrt{t}^{\pm 1}\right]}
			A\left[\sqrt{t}^{\pm 1}\right]\ar[l, "\sim"']
		\end{tikzcd}\]
		by Proposition \ref{prop:monoidal}. We denote the composite horizontal $k\left[t\right]$-algebra homomorphism $\bfA\to \bfA\otimes_{k\left[t\right]} \bfA$ by the same symbol $\bfDelta$. For $a\in\bfA$, we will regard $\bfDelta(a)$ as an element of $A\left[\sqrt{t}^{\pm 1}\right]\otimes_{k\left[\sqrt{t}^{\pm 1}\right]}
		A\left[\sqrt{t}^{\pm 1}\right]$ through the counterclockwise sequence of arrows in computations rather than $\bfA\otimes_{k\left[t\right]} \bfA$.
	\end{cons}
	
	We now record our result as a formal statement:
	
	\begin{thm}\label{thm:Hopf}
		The $k\left[t\right]$-algebra $\bfA$ is a Hopf algebra for the homomorphisms $(\bfDelta,\bfepsilon,\bfi)$.
	\end{thm}
	
	Set $K=\Spec A/I$. According to \cite[Proof of Lemma 3.1.1]{MR4627704}, $K\subset G$ can be identified with the fixed point subgroup scheme by $\Spec \theta$. 
	
	Let $\fg$ be the Lie algebra of $G$ (see \cite[Chapter II, \S 4]{MR563524} for the general formalism). We denote the differential of $\theta$ by the same symbol. Set
	\[\begin{array}{cc}
		\fk=\{x\in\fg:~\theta(x)=x\},&
		\fp=\{x\in\fg:~\theta(x)=-x\}.
	\end{array}\]
	It follows by definition that $\fk$ is naturally identified with the Lie algebra of $K$.
	
	\begin{cor}\label{cor:cartan_motion_group}
		If $A$ is smooth over $k$ then $\Spec \bfA_0$ is isomorphic to $K\ltimes\underline{\fp}$ as an affine group scheme over $k$.
	\end{cor}
	
	For a digression, let us compute the differential structures:
	
	\begin{cons}[{\cite[2.1.3]{MR4130851}}]
		Write $\left[-,-\right]_{\fg}$ for the Lie bracket of $\fg$.
		Define a Lie algebra over $k\left[t\right]$ as follows:	
		\[\bfg= \fg\otimes_k k\left[t\right],\]
		\[\left[x,y\right]_{\bfg}=\begin{cases}
			\left[x,y\right]_{\fg}&(x\in\fk,~y\in\fg)\\
			t\left[x,y\right]_{\fg}&(x,y\in\fp),
		\end{cases}\]
		where $\left[x,y\right]_{\bfg}$ is the Lie bracket of $x$ and $y$ in $\bfg$. We also set $\bfg_t\coloneqq \bfg\otimes_{k\left[t\right]} k\left[t^{\pm 1}\right]$.
	\end{cons}
	
	\begin{prop}\label{prop:liealg}
		Suppose that $A$ is finitely presented over $k$. Then the Lie algebra of $\bfG$ is isomorphic to $\bfg$.
	\end{prop}
	
	Let $\Der_k(A,k)$ denote the $k$-derivations of $A$ on $k$, where $k$ is regarded as an $A$-module for $\epsilon$. We use similar notations for other commutative Hopf algebras. According to \cite[Chapter I, \S 4, 2.2 Proposition and Chapter II, \S 4, 3.6 Corollary]{MR563524}, one can identify $\Der_k(A,k)$ with the Lie algebra of $G$.
	
	\begin{proof}
		We identify $\bfg$ with the Lie subalgebra
		\[\fk\otimes_k k\left[t\right]\oplus \fp\otimes_k k\left[t\right]\sqrt{t}
		\subset\fg\otimes_k  k\left[\sqrt{t}^{\pm 1}\right] \]
		over $k\left[t\right]$.
		
		Observe that $\fg\otimes_k k\left[\sqrt{t}^{\pm 1}\right]\cong \Der_{k\left[\sqrt{t}^{\pm 1}\right]}\left(A\left[\sqrt{t}^{\pm 1}\right],k\left[\sqrt{t}^{\pm 1}\right]\right)$ since $A$ is finitely presented (\cite[Corollaire (16.4.22)]{MR0238860}).	
		Define a Lie algebra homomorphism
		\[\Der_{k\left[t\right]}(\bfA,k\left[t\right])\to \Der_{k\left[\sqrt{t}^{\pm 1}\right]}\left(A\left[\sqrt{t}^{\pm 1}\right],k\left[\sqrt{t}^{\pm 1}\right]\right)
		\cong \fg\otimes_k k\left[\sqrt{t}^{\pm 1}\right]
		\]
		by the base change $-\otimes_{k\left[t\right]} k\left[\sqrt{t}^{\pm 1}\right]$ and the identification of Lemma \ref{lem:relation} (3). The proof will be completed by showing that this is an isomorphism onto $\bfg$.
		
		For this, let $\xi\in \Der_{k\left[t\right]}(\bfA,k\left[t\right])$. Then define a $k\left[t\right]$-module homomorphism $x:A\left[t\right]\to k\left[t\right]$ by
		\[x(a)=\begin{cases}
			\xi(a)&(a\in A^{\theta})\\
			\xi\left(\frac{1}{\sqrt{t}} a\right)&(a\in A^{-\theta}).
		\end{cases}\]
		It is easy to show that $x$ belongs to $\Der_{k\left[t\right]}(A\left[t\right],k\left[t\right])$ (use Lemma \ref{lem:counit}). Conversely, suppose that we are given a $k\left[t\right]$-derivation $x$ of $A\left[t\right]$. Then define a $k\left[t\right]$-module homomorphism $\xi:\bfA\to k\left[t\right]$ by
		\[\begin{array}{cc}
			\xi(a)=x(a)&(a\in A^{\theta})\\
			\xi\left(\frac{1}{\sqrt{t}} a\right)=x(a)&(a\in A^{-\theta})\\
			\xi|_{\oplus_{n\geq 2} \frac{1}{\sqrt{t}^n} (A^{-\theta})^n}=0.
		\end{array}\]
		One can easily prove that $\xi$ belongs to $\Der_{k\left[t\right]}(\bfA,k\left[t\right])$ (use Lemma \ref{lem:counit}). These constructions give a $k\left[t\right]$-module isomorphism
		\[\Der_{k\left[t\right]}(\bfA,k\left[t\right])
		\cong \Der_{k\left[t\right]}(A\left[t\right],k\left[t\right]).\]
		Since $A$ is finitely presented over $k$, the canonical map
		\[\fg\otimes_k k\left[t\right] \to
		\Der_{k\left[t\right]}(A\left[t\right],k\left[t\right]) \]
		is an isomorphism (use \cite[Corollaire (16.4.22)]{MR0238860}). We also note that $x(a)=0$ for $(x,a)\in (\fk\times A^{-\theta})\cup (\fp\times A^\theta)\subset\fg\times A$.
		It is now straightforward by unwinding the definitions that the total map
		\[\begin{split}
			\fg\otimes_k k\left[t\right]
			&\cong
			\Der_{k\left[t\right]}(A\left[t\right],k\left[t\right])\\
			&\cong \Der_{k\left[t\right]}(\bfA,k\left[t\right])\\
			&\to \Der_{k\left[\sqrt{t}^{\pm 1}\right]}\left(A\left[\sqrt{t}^{\pm 1}\right],k\left[\sqrt{t}^{\pm 1}\right]\right)\\
			&\cong \fg\otimes_k k\left[\sqrt{t}^{\pm 1}\right]
		\end{split}
		\]
		is given by $\fk\ni x\mapsto x$ and $\fp\ni x\mapsto \sqrt{t}x$. The assertion is now obvious.
		
	\end{proof}
	
	\begin{cor}\label{cor:diffaction}
		Let $Y=\Spec B$ be an affine $k$-scheme, equipped with an action of $G$, and $\eta$ be an involution of $B$. Assume:
		\begin{enumerate}
			\renewcommand{\labelenumi}{(\roman{enumi})}
			\item The action of $G$ on $Y$ respects the involutions of $G$ and $Y$;
			\item $A$ satisfies Condition \ref{cond:flat}.
		\end{enumerate}
		Then $\bfY$ is naturally equipped with an action of $\bfG$. Moreover, its differential action is given by
		\[\bfx(b)=\begin{cases}
			x(b)&(x\in\fk,~b\in B^\eta)\\
			\sqrt{t} x(b)&(x\in\fp,~b\in B^\eta)\\
		\end{cases}\]
		\[\bfx\left(\frac{1}{\sqrt{t}}b\right)=\begin{cases}
			\frac{1}{\sqrt{t}}x(b)&(x\in\fk,~b\in B^{-\eta})\\
			x(b)&(x\in\fp,~b\in B^{-\eta})
		\end{cases}\]
		under the identification of Proposition \ref{prop:liealg}, where for $x\in\fg\subset\bfg$, $\bfx$ is the image of $x$ in the Lie algebra of $\bfG$.
	\end{cor}
	
	We remark that $\bfx$ is uniquely determined by the formulas above through the Leibniz rule.
	
	\begin{proof}
		The first assertion follows from Proposition \ref{prop:monoidal} in a similar way to Corollary \ref{cor:monoidal}. We compute the differential by embedding $\bfB$ into $B\otimes_k k\left[\sqrt{t}^{\pm 1}\right]$. That is, we have a commutative diagram
		\[\begin{tikzcd}
			\bfg\otimes_{k\left[t\right]} \bfB\ar[r]\ar[d]&\bfB\ar[d,hook]\\
			(\fg\otimes_k k\left[\sqrt{t}^{\pm 1}\right])
			\otimes_{k\left[\sqrt{t}^{\pm 1}\right]} B\otimes_k k\left[\sqrt{t}^{\pm 1}\right]
			\ar[r]&B\otimes_k k\left[\sqrt{t}^{\pm 1}\right].
		\end{tikzcd}\]
		Here the horizontal arrows are given by the differential actions. The vertical arrows the induced maps from $k\left[t\right]\subset k\left[\sqrt{t}^{\pm 1}\right]$ (recall Lemma \ref{lem:relation} (3) and the proof of Proposition \ref{prop:liealg}). The differential is now computed as follows:
		\[\begin{array}{cc}
			\bfx(b)=x(b)&(\xi\in \fk,~b\in B^{\eta});\\
			\bfx(b)=(\sqrt{t}x)(b)=\sqrt{t}x(b)&(\xi\in \fp,~b\in B^{\eta});\\
			\bfx\left(\frac{1}{\sqrt{t}}b\right)=x\left(\frac{1}{\sqrt{t}}b\right)
			=\frac{1}{\sqrt{t}}x(b)&(\xi\in \fk,~b\in B^{-\eta});\\
			\bfx\left(\frac{1}{\sqrt{t}}b\right)
			=(\sqrt{t}x)\left(\frac{1}{\sqrt{t}}b\right)
			=x(b)&(\xi\in \fp,~b\in B^{-\eta}).
		\end{array}\]
		This completes the proof.

	\end{proof}
	
	The following assertion is evident by construction:
	
	\begin{prop}\label{prop:flatbc_hopf}
		Let $k\to k'$ be a flat (resp.~arbitrary) homomorphism of commutative $\bZ\left[1/2\right]$-algebras, and $A$ be a commutative Hopf $k$-algebra with an involution $\theta$. Write $B=k'\otimes_k A$. If $(A,\theta)$ satisfies Condition \ref{cond:flat} (resp.~$A$ is a smooth $k$-algebra) then the canonical map
		$k'\otimes_k\bfA\cong \bfB$ of $k'\left[t\right]$-algebras in Proposition \ref{prop:flatbc} (resp.~Theorem \ref{thm:smooth} (1)) is an isomorphism of Hopf algebras over $k'\left[t\right]$. We remark that this statement makes a sense since $B$ satisfies (resp.~$A$ and $B$ satisfy) Condition \ref{cond:flat} by Example \ref{ex:flatbc} (resp.~\ref{ex:flat_smooth}). 
	\end{prop}
	
	For a positive integer $n$, let $\SL_n$ be the special linear group scheme of degree $n$ over $k$.
	
	\begin{prop}\label{prop:comparisonwithbhs_groups}
		Suppose that we are given a representation $\iota:G\to \SL_n$ ($n\geq 1$) which is a closed immersion as a morphism of schemes. Regard $\SL_n$ as a closed subscheme of $\bA^{n^2}_k$. Then the morphism $\tilde{\iota}:\bfG\hookrightarrow\bA^{4n^2}_{k\left[t\right]}$ in Corollary \ref{cor:comparisonwithbhs} factors through $\SL_{2n}\otimes_k k\left[t\right]$. Moreover, the resulting map $\bfG\to \SL_{2n}\otimes_k k\left[t\right]$ is a homomorphism of group schemes. 
	\end{prop}
	
	\begin{proof}
		The map $\tilde{\iota}\otimes_{k\left[t\right]} k\left[\sqrt{t}^{\pm 1}\right]$ is expressed as
		\[g\mapsto\left(\begin{array}{cc}
			\sqrt{t} I_n & -\sqrt{t} I_n\\
			I_n & I_n
		\end{array}\right)\left(\begin{array}{cc}
			\iota(g) & 0 \\
			0 & \iota(\theta(g))
		\end{array}\right)\left(\begin{array}{cc}
			\sqrt{t} I_n & -\sqrt{t} I_n\\
			I_n & I_n
		\end{array}\right)^{-1}\]
		under the identification
		$G\otimes_k k\left[\sqrt{t}^{\pm 1}\right]\cong\bfG\otimes_{k\left[t\right]} k\left[\sqrt{t}^{\pm 1}\right]$,
		where $I_n$ is the $n$th unit matrix (cf.~\cite[Definition 4.1]{MR3797197}). Therefore $\tilde{\iota}\otimes_{k\left[t\right]} k\left[\sqrt{t}^{\pm 1}\right]$ factors through $\SL_{2n} \otimes_k k\left[\sqrt{t}^{\pm 1}\right]$. It descends to a homomorphism
		\[\bfG\otimes_{k\left[t\right]} k\left[t^{\pm 1}\right]\to \SL_{2n}\otimes_k k\left[t^{\pm 1}\right].\]
		Since $\SL_{2n}\otimes_k k\left[t\right]$ is closed in $\bA^{n^2}_{k\left[t\right]}$, $\tilde{\iota}$ factors through $\SL_{2n}\otimes_k k\left[t\right]$ (recall \cite[Chapter 2, Exercise 3.17 (d)]{MR1917232}). One can verify that the resulting map $\bfG\hookrightarrow \SL_{2n}\otimes_k k\left[t\right]$ is a morphism of affine group schemes over $k\left[t\right]$ by a similar argument to Corollary \ref{cor:monoidal}. In fact, one can check that the corresponding map of the coordinate rings is a homomorphism of Hopf algebras over $k\left[t\right]$ by comparing the comultiplications in
		$A\left[\sqrt{t}^{\pm 1}\right]\otimes_{k\left[\sqrt{t}^{\pm 1}\right]}
		A\left[\sqrt{t}^{\pm 1}\right]$.
		This completes the proof.
	\end{proof}
	
	\begin{cor}\label{cor:compatibility}
		Suppose $k=F\in\{\bR,\bC\}$. Let $G$ be an affine algebraic group over $F$ with an involution. Then the attached group scheme $\bfG$ of ours is isomorphic to the contraction family constructed in \cite[Definition 4.2 and Theorem 5.1]{MR3797197}.
	\end{cor}
	
	As we mentioned in the introduction, we only work over the polynomial ring $F\left[t\right]$ just for simplicity.
	
	\begin{proof}
		For existence of a faithful representation of $G$ over $F$, see \cite[Theorem 4.9]{MR3729270}. We can find a faithful representation of $G$ to a special linear group by taking the determinant.
		
		Henceforth fix a faithful representation $\iota:G\hookrightarrow \SL_n$ for some positive integer $n$. Then the assertion for $F=\bC$ follows from Proposition \ref{prop:comparisonwithbhs_groups} and \cite[Chapter 2, Exercise 3.17 (e)]{MR1917232}. Put $F=\bR$. We may assume the matrix $S$ in \cite[Lemma 5.1]{MR3797197} to be the unit since our faithful representation is defined over the real numbers. The real form of \cite[Theorem 5.1]{MR3797197} now coincides with our $\bfG$ by Propositions \ref{prop:comparisonwithbhs_groups} and \ref{prop:bc_image} (use Proposition \ref{prop:comparisonwithbhs_groups} twice; one each of $\bR,\bC$).
	\end{proof}

	\section{Contraction families of quotient varieties}\label{sec:quotient}
	
	Throughout this section, we let $k$ be a commutative $\bZ\left[1/2\right]$-algebra, and $G$ be a smooth affine group scheme over $k$ with an involution $\theta$. Put $K$ and $\fp$ as in Section \ref{sec:hopf}. Let $H$ be a $\theta$-invariant smooth closed subgroup scheme of $G$. It follows from the naturality of our Hopf structure and Corollary \ref{cor:surj} that $\bfH$ is a closed subgroup scheme of $\bfG$.
	
	The aim of this section is to examine relations of the quotient $\bfG/\bfH$ with other related $\bfG$-schemes under additional assumptions on $k$, $G$, and $H$. Let us note a general result on the quotient:
	
	\begin{lem}\label{lem:smooth}
		Let $L$ be a smooth quasi-compact and quasi-separated group scheme over a scheme $S$, and $M$ be a closed subgroup scheme of $L$ which is flat and locally of finite presentation over $S$. Then the fppf quotient $L/M$ is a smooth quasi-compact separated algebraic space over $S$.
	\end{lem}
	
	\begin{proof}
		According to \cite[Lemma 1.1]{MR4680514}, $L/M$ is an algebraic space, and  the quotient map $L\to L/M$ is faithfully flat and locally of finite presentation. Since $L$ is quasi-compact over $S$, so is $L/M$. In view of \cite[Proposition 1.2]{MR4680514}, $L/M$ is separated over $S$ since $M$ is a closed subgroup scheme. The rest is reduced to \cite[Lemma (17.7.5) and Proposition (17.7.7)]{MR0238860} by taking an atlas of $L/M$.
	\end{proof}

	Together with Theorem \ref{thm:smooth}, Lemma \ref{lem:smooth} immediately implies:
	
	\begin{prop}\label{prop:smooth}
		The fppf quotient $\bfG/\bfH$ is a smooth quasi-compact separated algebraic space over $k\left[t\right]$.
	\end{prop}
	
	\begin{ex}[{\cite[4.C. Th\'eor\`eme]{MR0335524}}]\label{ex:representable}
		The fppf quotient $\bfG/\bfH$ is represented by a scheme if $k=F$ is a field (Example \ref{ex:representable}).
	\end{ex}

	Before we get into specific settings below, let us renew the notations on Lie algebras from the former section: For a smooth affine group scheme over a commutative ring, we denote its Lie algebra by the corresponding small German letter. Its adjoint representation will be denoted by $\Ad$. The Lie bracket will be denoted by $\left[-,-\right]$. We remark that there is no conflict of the notations for the bold German letters and the subscript $(-)_t$ in virtue of Proposition \ref{prop:liealg}, Lemma \ref{lem:relation}, and \cite[Chapter II, \S 4, 1.4]{MR563524}. For an involution $\theta$ on the given group scheme, we will denote the induced involution on its Lie algebra by the same symbol $\theta$. To use results of \cite{MR0212023} in this section, we note that the definition of the Lie algebra of smooth affine group schemes in \cite{MR563524} coincides with that of \cite[Definitions 3.9.0.1 c), 4.7.2, and Scholie 4.9]{MR0212023} by \cite[Corollaire 3.9.0.2]{MR0212023}.
	
	Let us also introduce an additional notation for Section \ref{sec:flag}: For a smooth affine group scheme $L$ over a commutative ring $R$ and a smooth affine subgroup scheme $M$, we denote the normalizer of $\fm$ in $\fl$ with respect to the adjoint representation by $N_L(\underline{\fm})$ (see \cite[D\'efinition 2.3.3]{MR0212028} for a general formalism). We remark that $\underline{\fm}$ is a sub-copresheaf of $\underline{\fl}$ by the definition of the Lie algebras of group schemes as copresheaves on the category of commutative $R$-algebras. According to \cite[Corollaire 4.11.8]{MR0212023} and \cite[Rappel 5.3.0]{MR0218363}, $N_L(\underline{\fm})$ is represented by a closed subgroup scheme of $L$ which is of finite presentation over $L$. In particular, $N_L(\underline{\fm})$ is of finite presentation over $R$.
	
	Let us note that $\bfG/\bfH$ is a refined contraction analog of $G/H$ in the following sense:
	
	\begin{prop}\label{prop:quotient_at_0}
		The fiber of $\bfG/\bfH$ at $t=0$ is isomorphic to $K\times^{K\cap H} \underline{\fp/(\fp\cap\fh)}$. In particular, if $k=F$ is a field then this is identified with the normal bundle $T_{K/K\cap H} G/H$.
	\end{prop}
	
	\begin{proof}
		The first part is straightforward:
		\[(\bfG/\bfH)\otimes_{k\left[t\right]}k\left[t\right]/(t)
		\cong \bfG_0/\bfH_0\cong
		(K\ltimes\underline{\fp})/((K\cap H)\ltimes\underline{\fp\cap\fh})
		\cong K\times^{K\cap H} \underline{\fp/(\fp\cap\fh)}.
		\]
		The latter part follows since the $K$-orbit map $K/(K\cap H)\hookrightarrow G/H$ is an immersion if the base ring is a field.
	\end{proof}

	\subsection{Contraction families of symmetric varieties}
	
	In this section, we assume $k=F$ to be a field, and $G$ to be (possibly disconnected) reductive, i.e., $G^\circ$ to be reductive in the classical sense (cf.~\cite[Chapter I, Proposition (b) of Section 1.2]{MR1102012}, \cite[Proposition 3.3]{MR0234961}). Then $K$ is reductive by \cite[the beginning of Section 1]{MR1066573} (use \cite[Chapter I, 3.6 Corollary]{MR1102012} and \cite[Proposition 10.15]{MR3729270} to prove $K^\circ=((G^\circ)^\theta)^\circ$). Therefore $X\coloneqq G/K$ is an affine variety from the Matsushima criterion (see \cite{MR0476772, MR437549}). We denote the attached involution on $X$ to $\theta$ by the same symbol. In this section, we study $\bfX$ and $\bfX_0$.
	
	\begin{prop}\label{prop:finetale}
		The $F$-variety $X^{\theta}$ is finite \'etale. In particular, $\bfX_0$ is isomorphic to the disjoint union of copies of $\fp$ if $F$ is separably closed.
	\end{prop}
	
	\begin{proof}
		We may assume that $F$ is algebraically closed. Recall that $X^{\theta}$ is a smooth $F$-variety by \cite[Lemma 3.1.1]{MR4627704}. In the rest, we may identify the closed points of $X$ and $X^{\theta}$ with their $F$-points.
		
		For each point $gK\in X(F)$, the tangent space of $X$ at $gK$ can identified with $\fg/\Ad(g)\fk$. If $\theta(g)K=gK$ then $\theta$ induces an automorphism
		\[\fg/\Ad(g)\fk\to \fg/\Ad(\theta(g))\fk=\fg/\Ad(g)\Ad(g^{-1}\theta(g))\fk
		=\fg/\Ad(g)\fk.\]
		Moreover, the tangent space of $X^\theta$ at $gK$ coincides with its fixed point subspace. If an element $x\in\fg$ satisfies $\theta(x)\in \Ad(g)\fk$ then $x$ belongs to $\Ad(g)\fk$ by
		\[\begin{split}
			\Ad(g)^{-1}x&=\theta(\Ad(\theta(g)^{-1})\theta(x))\\
			&=\theta(\Ad(\theta(g)^{-1}g)\Ad(g)^{-1}\theta(x))\\
			&\in\theta(\Ad(K(F))\fk)\\
			&=\fk.
		\end{split}
		\]
		This implies that the tangent space of $X^\theta$ at $gK$ is trivial. Therefore $X^{\theta}$ is a smooth affine algebraic variety of dimension zero. Equivalently, $X^{\theta}$ is a finite \'etale $F$-variety. This completes the proof.
	\end{proof}
	
	For a digression, let us note how we can compute $X^\theta$ in a special case:
	
	\begin{cor}\label{cor:component}
		Assume that $G$ is connected and simply connected. Let $F^s$ be a separable closure of $F$, and $\Gamma$ be the absolute Galois group of $F$. Put the canonical actions of $\Gamma$ on $H^1(\theta,G(F^s))$ and $H^1(\theta,K(F^s))$. Then $X^\theta(F)$ is bijective to the $\Gamma$-invariant part of the kernel of the canonical map
		\begin{equation}
			H^1(\theta,K(F))\to H^1(\theta,G(F)).\label{eq:coh}
		\end{equation}
	\end{cor}
	
	\begin{proof}
		We obtain a bijection between the set of $K(F^s)$-orbits in $X^\theta(F^s)$ and the kernel of \eqref{eq:coh} from \cite[Chapter II, Caution of Section 6.8]{MR1102012} and \cite[Chapter I, section 5.4, Corollary 1]{MR1867431}. It is evident by construction that this map is $\Gamma$-equivariant. One deduces from \cite[Theorem 8.1]{MR0230728} and Proposition \ref{prop:finetale} that $K$ acts trivially on $X^\theta$. Therefore the set of $K(F^s)$-orbits in $X^\theta(F^s)$ is exactly identified with $X^\theta(F^s)$. The proof is completed by taking the $\Gamma$-invariant part.
	\end{proof}
	
	\begin{rem}
		Since $\theta$ acts trivially on $K$, $H^1(\theta,K(F^s))$ is identified with the set of $K(F^s)$-conjugacy classes of the elements of $K(F^s)$ of order at most two.
	\end{rem}
	
	\begin{rem}
		The cohomology $H^1(\theta,G(F))$ for $F=\bC$ and complexified Cartan involutions $\theta$ was studied in \cite{MR3786301}. We can apply the computations of \cite{MR3786301} to $H^1(\theta,K(\bC))$ since the identity map is a complexified Cartan involution of $K$.
	\end{rem}

	Let $x_0=K$ be the base point of $X$. This lifts to an $F\left[t\right]$-point $\bfx_0$ of $\bfX$ since $\theta(x_0)=x_0$. We denote the centralizer subgroup of $\bfG$ at $\bfx_0$ by $Z_{\bfG}(\bfx_0)$ (see \cite[D\'efinition 2.3.3]{MR0212028}).
	Write $\bfG\bfx_0$ for the $\bfG$-orbit attached to $\bfx_0$. It is evident by definitions that $\bfG\bfx_0\cong \bfG/Z_{\bfG}(\bfx_0)$.

	\begin{thm}\label{thm:G/K}
		\begin{enumerate}
			\renewcommand{\labelenumi}{(\arabic{enumi})}
			\item We have $Z_{\bfG}(\bfx_0)=\bfK$. In particular, we have an isomorphism $\bfG\bfx_0\cong\bfG/\bfK$.
			\item The $\bfG$-orbit $\bfG\bfx_0$ is an affine open subscheme of $\bfX$. Moreover, it is isomorphic onto $\bfX^\circ$ defined by $\bfx_0$ if $G$ is connected.
		\end{enumerate}
	\end{thm}
	
	\begin{proof}
		Since $\bfX$ is smooth and affine (Theorem \ref{thm:smooth}), $\bfX$ is separated and locally of finite presentation over $F\left[t\right]$. Hence the centralizer $Z_{\bfG}(\bfx_0)$ is represented by a closed subgroup of finite presentation of $\bfG$ (\cite[Exemples 6.2.4.~b)]{MR0234961}). In particular, $Z_{\bfG}(\bfx_0)$ is of finite presentation over $F\left[t\right]$ since $\bfG$ is affine and smooth over $F\left[t\right]$ (Theorem \ref{thm:smooth}). The action of $G$ at $x_0$ gives rise to a map $G\to X$, whose restriction to $K$ coincides with the constant map at $x_0$. In virtue of the functoriality of contraction, $\bfK$ fixes $\bfx_0$. Therefore $\bfK$ is contained in $Z_{\bfG}(\bfx_0)$.
		
		We wish to prove $\bfK=Z_{\bfG}(\bfx_0)$. In view of \cite[Corollaire (17.9.5)]{MR0238860} and Theorem \ref{thm:smooth}, it will suffice to see the equality at $t=0$ and the locus of $t\neq 0$. To see the action on these loci, we remark that the fppf quotient commutes with any base change. Therefore the action of $\bfG$ on $\bfX$ at $t=0$ can be identified with that of the Cartan motion group $K\ltimes\underline{\fp}$ on the normal bundle $T_{X^\theta}X$. It is straightforward to see that the stabilizer subgroup at the fiber of $\bfx_0$ at $t=0$ is $K$ (cf.~Proposition \ref{prop:finetale}). To see the equality on $t\neq 0$, we may take the base change to $F\left[\sqrt{t}^{\pm 1}\right]$ to identify the action with the base change of the original action of $G$ on $G/K$. Since the centralizer subgroup of $G$ at $x_0$ is $K$, we get \[\bfK\otimes_{F\left[t\right]} F\left[\sqrt{t}^{\pm 1}\right]
		=Z_{\bfG}(\bfx_0)\otimes_{F\left[t\right]} F\left[\sqrt{t}^{\pm 1}\right].\]
		This proves (1).
		
		We next prove (2). In virtue of \cite[Th\'eor\`eme 10.1.2]{MR0257095} or Example \ref{ex:representable} that $\bfG\bfx_0$ is represented by an $F[t]$-scheme. Moreover, it is smooth by Lemma \ref{lem:smooth}. The map $i$ is an isomorphism on the locus of $t\neq 0$ from Lemma \ref{lem:relation} since the fppf quotient commutes with arbitrary base changes. We also observe that if $G$ is connected, $\bfX\otimes_{F[t]} F[t^{\pm1}]$ has connected fibers (equivalently, $\bfX^\circ\otimes_{F[t]} F[t^{\pm1}]=\bfX\otimes_{F[t]} F[t^{\pm1}]$) since $X$ is geometrically connected in this case (recall Lemma \ref{lem:relation}). The fiber of $i$ at $t=0$ is an isomorphism onto the open subvariety $\bfX^\circ_0\coloneqq \bfX^\circ \cap \bfX_0\subset \bfX_0$ by Proposition \ref{prop:finetale}. Since we are working over a field, $\bfX^\circ_0$ is also closed in $\bfX_0$. The assertions now follow from \cite[Corollaire (17.9.5)]{MR0238860} and Lemma \ref{lem:affine}.
	\end{proof}
	
\begin{rem}\label{rem:quotient}
	Consider the general setting at the beginning of Section \ref{sec:quotient}. Then a similar argument implies that if $X=G/H$ is represented by an affine $k$-scheme then $\bfG/\bfH$ is represented by a smooth open $k[t]$-subscheme of $\bfX$. If $K/(K\cap H)$ is closed in $X^\theta$, then $\bfG/\bfH$ is affine.
\end{rem}

Let us see some instances of Remark \ref{rem:quotient}:

	\begin{ex}[{\cite[Corollaire 10.1.3]{MR0257095}}, {\cite[Th\'eor\`eme 5.1]{MR0212024}}]
		Let $H$ be of multiplicative type (\cite[D\'efinition 1.1]{MR0215854}). Then $G/H$ is represented by an affine $k$-scheme.
	\end{ex}
	
	\begin{ex}\label{ex:semisimpleorbit}
		Put $G=\SL_2$ and $\theta=((-)^T)^{-1}$. Set $H=\SO(2)$. Regard $\fg$ as the Lie algebra $\fsl_2$ of square matrices of size $2$ with trivial trace. Identify $\underline{\fsl}_2$ with $\bA^3_k$ through the basis
		\[\begin{array}{ccc}
			\xi_1=\left(\begin{array}{cc}
				0 & -1 \\
				1 & 0
			\end{array}\right),
			&\xi_2=\left(\begin{array}{cc}
				0 & 1 \\
				1 & 0
			\end{array}\right),
			&
			\xi_3=\left(\begin{array}{cc}
				1 & 0 \\
				0 & -1
			\end{array}\right).
		\end{array}\]
		Then the stabilizer subgroup of the adjoint action at $\xi_1$ is $\SO(2)$, and the attached $\SL_2$-orbit is $\Spec k\left[x,y,z\right]/(x^2-y^2-z^2-1)$ in Example \ref{ex:sl_2/so(2)}. The involution in Example \ref{ex:sl_2/so(2)} agrees with the differential of $\theta=((-)^T)^{-1}$. Therefore the description in Example \ref{ex:sl_2/so(2)} is available. In particular, $\bfG/\bfH$ is obtained by removing $\Spec k[x,y,z]/(x+1)$ from $\Spec k\left[t,x,y,z\right]/(x^2-ty^2-tz^2-1)$ at $t=0$.
	\end{ex}
	
	\begin{ex}
		Assume that $G$ is reductive in the sense of \cite[D\'efinition 2.7]{MR0228502}. Let $H$ be of multiplicative type, and $N_G(H)$ be its normalizer. According to $N_G(H)\subset G$ is a smooth closed subgroup scheme over $k$ (\cite[Proposition 2.1.2]{MR3362641}). Moreover, $G/N_{G}(H)$ is represented by an affine $k$-scheme if $H$ is a maximal torus of $G$ (\cite[Corollaire 5.8.3]{MR0218363}). If $H$ is $\theta$-stable, so is $N_G(H)$.
	\end{ex}
	
	\begin{ex}
		Assume $k=F$ to be a field. If $H$ is possibly disconnected reductive, $G/H$ is an affine variety (see \cite[Section 1]{MR437549}). In virtue of \cite[the beginning of Section 1]{MR1066573}, a typical example of such $H$ is obtained as $G^{\eta}$, where $\eta$ is an involution of $G$ commuting with $\theta$. 
	\end{ex}
	
	We now return to the setting at the beginning of the present section. The symmetric variety $G/K$ with $F=\bR$ appears in representation theory of Lie groups and harmonic analysis as an algebraic model of symmetric spaces. In fact, we obtain a manifold by taking its real points. Similarly, we have a one-parameter family of manifolds by taking real points of $\bfX$. Since $\bfX$ is smooth over $\bR[t]$, $\bfX(\bR)$ is naturally equipped with the structure of a manifold. However, the fiber at each of $t=t_0$ may be disconnected. In fact, write $\sigma_c$ for the involutions on $G(\bC)$ and $K(\bC)$ defined by the composition of the complex conjugation and $\theta$. If $G$ is connected and simply connected then for $t_0\in\bR$, we have
	\[\pi_0(\bfX_{t_0}(\bR))\cong
	\begin{cases}
		\Ker(H^1(\Gamma,K(\bC))\to H^1(\Gamma,G(\bC)))&(t_0>0)\\
		(\Ker(H^1(\theta,K(\bC))\to H^1(\theta,G(\bC))))^\Gamma&(t_0=0)\\
		\Ker(H^1(\sigma_c,K(\bC))\to H^1(\sigma_c,G(\bC)))&(t_0<0).
	\end{cases}\]
	We note that if $\theta$ is a Cartan involution, there is a canonical isomorphism
	\[(\Ker(H^1(\theta,K(\bC))\to H^1(\theta,G(\bC))))^\Gamma
	\cong \Ker(H^1(\Gamma,K(\bC))\to H^1(\Gamma,G(\bC)))
	\]
	(\cite[Corollary 4.7]{MR3786301}).
	For more explicit computations, see Example \ref{ex:galois}.
	
	If one is interested in symmetric spaces of real Lie groups,  one can simply resolve this component issue by an analytic analog of \cite[Corollaire (15.6.5)]{MR0217086}:

	\begin{prop}
		Let $p:M\to N$ be a submersion of smooth manifolds, and $s$ be a section of $p$. For each $y\in N$, let $p^{-1}(y)^\circ$ be the connected component of $p^{-1}(y)$ containing $s(y)$. Set
		$M^\circ=\coprod_{y\in N} p^{-1}(y)^\circ$.
		Then $M^\circ$ is an open submanifold of $M$.
	\end{prop}
	
	\begin{proof}
		Take $x\in M^\circ$. We wish to find an open neighborhood of $x\in M$ contained in $M^\circ$. Write $y=p(x)$. Since $p^{-1}(y)$ is a manifold (\cite[Chapter 1, Theorem 5.6]{MR3289090}), one can choose a continuous map $c:\left[0,1\right]\to p^{-1}(y)$ satisfying $c(0)=x$ and $c(1)=s(y)$, where $\left[0,1\right]\subset\bR$ is the closed interval from $0$ to $1$. For each $t\in \left[0,1\right]$, one can take open neighborhoods $U_t\ni c(t)$ and $V_t\ni p(c(t))=p(x)$ with the following properties by \cite[Chapter 1, the proof of Theorem 5.6]{MR3289090}:
		\begin{enumerate}
			\item[(i)] $V_t=p(U_t)$;
			\item[(ii)] For each element $y'\in V_t$, $p^{-1}(y')\cap U_t$ is connected.
		\end{enumerate}
		In fact, shrink $W$ in \cite[Chapter 1, the proof of Theorem 5.6]{MR3289090} to assume that $W$ is the product of an open subset of $\bR^n$ containing $0$ and an open ball in $\bR^{m-n}$ around $0$. We may also replace $V_1$ and $U_1$ by a smaller open subset $V'_1\ni y$ and $p^{-1}(V_1')\cap U_1$ respectively with the property that $s(V'_1)\subset U_1$ to assume that $s(V_1)\subset U_1$.
		
		Since $\left[0,1\right]=\cup_{t\in \left[0,1\right]} c^{-1}(U_t)$, one can find a finite subset $\{0,1\}\subset I\subset \left[0,1\right]$ such that $\left[0,1\right]=\cup_{t\in I} c^{-1}(U_t)$. We construct a sequence $t_0,t_1,\ldots,t_n$ of distinct elements of $I$ with the following properties:
		\begin{enumerate}
			\renewcommand{\labelenumi}{(\roman{enumi})}
			\item $t_0=0$ and $t_n=1$;
			\item For $1\leq i\leq n$,
			$\cup_{j=0}^{i-1} (c^{-1}(U_{t_j}) \cap c^{-1}(U_{t_{i}}))\neq \emptyset$.
			\item For $0\leq i\leq n-2$, $c^{-1}(U_{t_i})\cap c^{-1}(U_1)=\emptyset$.
		\end{enumerate}
		Set $t_0=0$. For $i\geq 1$, define $t_i\in I$ as follows: If $c^{-1}(U_{t_{i-1}})\cap c^{-1}(U_1)\neq \emptyset$ then set $n=i$ and $t_i=1$; Suppose otherwise. Write $I_{i-1}=I\setminus\{t_0,t_1,\ldots,t_{i-1}\}$. Then we have
		\[\left[0,1\right]=\left(\cup_{j=0}^{i-1} c^{-1}(U_{t_j})\right)\cup 
		\cup_{t\in I_{i-1}} c^{-1}(U_{t}).\]
		We have elements $0\in c^{-1}(U_0)\subset\cup_{j=0}^{i-1} c^{-1}(U_{t_j})$ and $1\in \cup_{t\in I_{i-1}} c^{-1}(U_{t})$ since $1\in I_{i-1}$. Since $\left[0,1\right]$ is connected, $\left(\cup_{j=0}^{i-1} c^{-1}(U_{t_j})\right)\cap 
		\left(\cup_{t\in I_{i-1}} c^{-1}(U_{t})\right)$ is nonempty. One can and do choose $t_i\in I_{i-1}$ such that $\left(\cup_{j=0}^{i-1} c^{-1}(U_{t_j})\right)\cap 
		c^{-1}(U_{t_i})\neq \emptyset$. This procedure will stop since $I$ is a finite set.
		
		For each $1\leq i\leq n$, choose $0\leq m(i)\leq i-1$ such that
		\[c^{-1}(U_{t_{m(i)}})\cap c^{-1}(U_{t_{i}})\neq \emptyset.\]
		Set
		$V=\cap_{i=1}^{n} p(U_{t_{m(i)}}\cap U_{t_{i}})$. Then we have $V\subset \cap_{i=0}^n V_{t_i}$ since $m(1)=0$. Observe that $V$ contains $p(x)$. In fact, for any $1\leq i\leq n$, take an element
		\[u\in c^{-1}(U_{t_{m(i)}})\cap c^{-1}(U_{t_{i}}).\]
		Then $c(u)$ is contained in $U_{t_{m(i)}}\cap U_{t_{i}}$. We thus get
		\[p(x)=p(c(u))\in p(U_{t_{m(i)}}\cap U_{t_i}).\]
		We also note that $V$ is open since $p$ is an open map (\cite[Chapter 1, Theorem 5.6]{MR3289090}). 
		We define a descending sequence $n(j)$ by
		\[\begin{array}{cc}
			n(-1)=n,&n(j+1)=m(n(j)).
		\end{array}\]
		This sequence eventually stops, i.e., one can find a nonnegative integer $j_0$ such that $n(j_0)=0$.
		
		We prove that the open neighborhood $U_0\cap p^{-1}(V)$ of $x$ in $M$ is contained in $M^\circ$. Take any element $x'\in U_0\cap p^{-1}(V)$.
		Write $y'=p(x')$. Then for each $1\leq i\leq n$, $p^{-1}(y')\cap U_{t_{m(i)}}\cap U_{t_i}\neq \emptyset$ by definition of $V$. For each $-1\leq j\leq j_0-1$, fix an element $x'_{n(j)}\in p^{-1}(y')\cap U_{t_{n(j+1)}}\cap U_{t_{n(j)}}$. For convention, we set $x'_{n(-2)}\coloneqq s(y')$. Then $x'_{n(-2)}=s(y')$ belongs to $U_1=U_{t_n}=U_{t_{n(-1)}}$. According to our choice of $U_{t}$, one can connect $x'_{n(j-1)}$ and $x'_{n(j)}$ in $p^{-1}(y')\cap U_{t_{n(j)}}$ for each $-1\leq j\leq j_0-1$ (recall $y'\in V\subset V_{t_{n(j)}}$). Therefore we reach
		\[x'_{n(j_0-1)}\in p^{-1}(y')\cap U_{t_{n(j_0)}}=p^{-1}(y')\cap U_{t_0}=p^{-1}(y')\cap U_0\]
		from $s(y')$ by a path. We can connect $x'_{n(j_0-1)}$ with $x'$ in $p^{-1}(y')\cap U_0$ since \[p^{-1}(y')\cap U_0\]
		is connected. Therefore $x'$ belongs to $p^{-1}(y')\circ\subset M^\circ$. This completes the proof.
	\end{proof}
	
	\begin{cor}\label{cor:symmetric_space}
		Put $F=\bR$. Regard $\bfX$ and $\bfG/\bfK$ as $\bR$-schemes by composing the structure morphisms to $\Spec\bR\left[t\right]$ with the canonical map to $\Spec\bR$.
		Then the open immersion $\bfG/\bfK\to\bfX$ induces a diffeomorphism $(\bfG/\bfK)(\bR)^\circ\cong\bfX(\bR)^\circ$
		of the manifolds of the fiberwise unit components. Moreover, the fiber at $t=t_0\in\bR$ is
		$G_{t_0}(\bR)^\circ/K(\bR)\cap G_{t_0}(\bR)^\circ$ (resp.~$\fp$) if $t_0\neq 0$ (resp.~$t_0=0$), where $ G_{t_0}(\bR)^\circ$ is the unit component of the Lie group $G_{t_0}(\bR)$.
	\end{cor}

	\subsection{Contraction families of partial flag schemes}\label{sec:flag}
	
	In this section, we study the case where $H=Q$ is a $\theta$-stable parabolic subgroup.
	
	We start with a quite more general setting. I.e., let $k$ be an arbitrary commutative ring with $1/2$, and $G$ be a reductive group scheme over $k$ in the sense of \cite[D\'efinition 2.7]{MR0228502}, i.e., a smooth affine group scheme over $k$ whose geometric fibers are connected reductive algebraic groups. We say that $G$ is of type (RA) if for every root $\alpha$ relative to a maximal torus $T$ of each geometric fiber of $G$, the positive generator of the ideal
	\[\{\chi(\alpha)\in\bZ:~\chi\in \Hom_{\bZ}(X^\ast(T),\bZ)\}\subset\bZ\]
	is a unit of $k$, where $X^\ast(T)$ is the character group of $T$ (\cite[D\'efinition 5.1.6]{MR0218363}).
	
	Let us give a quick review on parabolic subgroups and partial flag schemes from SGA 3. A smooth subgroup scheme $Q$ of $G$ is called a parabolic subgroup if it is so at every geometric fiber in the classical sense (\cite[D\'efinition 1.1]{MR0218364}). Let $\cP_G$ denote the total flag scheme, i.e., the moduli scheme of parabolic subgroups of $G$ (\cite[Section 3.2, Th\'eor\`eme 3.3]{MR0218364}). Let $\type G$ be the finite \'etale $k$-scheme of parabolic types of $G$ (see \cite[Section 3]{MR0228503} and \cite[Section 3.2, D\'efinition 3.4]{MR0218364}). We have a canonical morphism $\cP_G\to \type G$ (\cite[Section 3.2]{MR0218364}). One can recognize from the arguments of \cite{MR0218364} that this map exhibits the \'etale quotient map of $\cP_G$ by the conjugate action of $G$. For a parabolic subgroup $Q\subset G$, we call the image of $Q$ in $(\type G)(k)$ the (parabolic) type of $Q$ (\cite[Section 3.2, D\'efinition 3.4]{MR0218364}). For each $k$-point $x$ of $\type G$, the fiber of the morphism $\cP_G\to \type G$ at $x$ is called the partial flag scheme of type $x$, and it will be denoted by $\cP_{G,x}$ (cf.~\cite[Corollaire 3.6]{MR0218364}). 
	
	\begin{ex}[{\cite[Corollaire 3.6]{MR0218364}}]
		There is a canonical parabolic type denoted by $\emptyset$. This corresponds to the flag scheme, i.e., the moduli scheme of Borel subgroups (\cite[Corollaire 5.8.3 (i)]{MR0218363}).
	\end{ex}
	
	One can define a morphism
	\begin{equation}
		\cP_{G}\to\Gr(\fg)\label{eq:iota}
	\end{equation}
	by assigning the Lie algebras of parabolic subgroups (see \cite[Corollaire 4.11.8]{MR0212023}). According to \cite[Corollaire 5.3.3]{MR0218363}, \eqref{eq:iota} is a monomorphism if $G$ is of type (RA). Since $\cP_G$ and $\Gr(\fg)$ are projective (\cite[Corollaire 3.5]{MR0218364}), \eqref{eq:iota} is a closed immersion in this case (\cite[Th\'eor\`eme 5.5.3 (i) and Corollaire 5.4.3 (i)]{MR217084}, \cite[Corollaire 18.12.6]{MR0238860}). For a parabolic type $x$, we will denote the restriction of \eqref{eq:iota} to the closed subscheme $\cP_{G,x}$ by $\iota_x$. For a latter argument, let us record a general observation on the normalizer of the Lie algebra of a parabolic subgroup here:
	
	\begin{lem}\label{lem:normalizer_q}
		Suppose that $G$ is of type $\mathrm{(RA)}$. Then for a parabolic subgroup $Q$ of $G$, we have $N_G(\underline{\fq})=Q$.
	\end{lem}
	
	\begin{proof}
		This follows by the monomorphicity of \eqref{eq:iota} and \cite[Proposition 1.2]{MR0218364}.
	\end{proof}
	
	Take an involution of $\theta$. Then $\bfG_t$ is a reductive group scheme over $k\left[t^{\pm 1}\right]$ by Lemma \ref{lem:relation} and \cite[the sentence below D\'efinition 2.7]{MR0228502}. Let us note a general remark for finding elements of $ (\type \bfG_t)(k\left[t^{\pm 1}\right])$:
	
	\begin{lem}
		Let $\bar{}$ be the Galois involution on $\type G\otimes_k k\left[\sqrt{t}^{\pm 1}\right]$ with respect to the quadratic Galois extension $k\left[\sqrt{t}^{\pm 1}\right]\supset k\left[t^{\pm 1}\right]$. We denote the involution on $\type G$ induced from $\theta$ by the same symbol. Then we have a canonical bijection
		\[(\type \bfG_t)(k\left[t^{\pm 1}\right])
		\cong\left\{x\in (\type G)\left(k\left[\sqrt{t}^{\pm 1}\right]\right):~\theta(\bar{x})=x\right\}.\]
	\end{lem}
	
	\begin{proof}
		The Galois involution on $(\type G)\left(k\left[\sqrt{t}^{\pm 1}\right]\right)$ corresponding to the $k\left[t^{\pm 1}\right]$-form $\bfG_t$ of $G\otimes_k k\left[\sqrt{t}^{\pm 1}\right]$ is given by $x\mapsto \theta(\bar{x})$. The assertion then follows from generalities on Galois descent.
	\end{proof}

	\begin{ex}\label{ex:theta_descent}
		Let $x\in (\type G)(k)$ such that $\theta(x)=x$. Then $x\otimes_k k\left[\sqrt{t}^{\pm 1}\right]$ descends to an element of $(\type \bfG_t)(k\left[t^{\pm 1}\right])$.
	\end{ex}
	
	\begin{ex}\label{ex:theta_descent_Q}
		Let $Q$ be a $\theta$-stable parabolic subgroup of $G$. Then the parabolic type of $Q$ attaches an element of $(\type \bfG_t)(k\left[t^{\pm 1}\right])$ by Example \ref{ex:theta_descent}. This coincides with the type of $\bfQ_t$.
	\end{ex}
	
	\begin{ex}
		Assume $k=F$ is a field. Then the minimal parabolic subgroups of $G$ determine a unique element of $(\type \bfG_t)(k\left[t^{\pm 1}\right])$ by \cite[20.9 Theorem]{MR1102012}.
	\end{ex}
	
	Henceforth we assume that $G$ is of type (RA). Then $\bfG_t$ is also of type (RA) by Lemma \ref{lem:relation} and \cite[Remarques 5.1.7 c)]{MR0218363}.
	
	Take $x\in (\type \bfG_t)(k\left[t^{\pm 1}\right])$. Let $j_0$ denote the canonical open immersion
	\[\Spec k\left[t^{\pm 1}\right]\hookrightarrow \Spec k\left[t\right]\]
	of affine $k$-schemes. Compose the closed immersion $\iota_x$ with the base change \[j_{\Gr(\bfg)}:\Gr(\bfg_t)\hookrightarrow \Gr(\bfg)\]
	of $j_0$ to get an affine immersion $\cP_{\bfG_t,x}\hookrightarrow \Gr(\bfg)$.
	
	\begin{defn}
		We call the scheme-theoretic closure of $\cP_{\bfG_t,x}$ in $\Gr(\bfg)$ the partial flag scheme of $\bfG$ of type $x$, and denote it by $\cP_{\bfG,x}$. If $x=\emptyset$ then we call it the flag scheme of $\bfG$, and refer to it as $\cB_{\bfG}$.
	\end{defn}

	\begin{rem}
		A key idea to prove that the total flag scheme of a reductive group scheme without the (RA) hypothesis was to pass to the adjoint group in the sense of \cite[D\'efinition 4.3.6]{MR0218363} (see \cite[Proof of Th\'eor\`eme 5.8.1]{MR0218363}). Therefore passing to the adjoint group of $G$ is another possible definition of partial flag schemes without the (RA) hypothesis. To be more precise, let $G'$ be the quotient of $G$ by its center $Z$ (see \cite[Corollaire 4.1.7 and Proposition 4.3.5 (ii)]{MR0218363}). Then $Z$ is $\theta$-stable in $G$, and therefore $\theta$ descends to an involution of $G'$. We note that $\bfZ_t$ is of finite presentation over $k\left[t^{\pm 1}\right]$ by Lemma \ref{lem:relation} (3), \cite[Corollaire 4.1.7]{MR0218363}, and \cite[Proposition 2.1 b)]{MR0215854}. Similarly, the center of $\bfG_t$ is flat of finite presentation over $k\left[t^{\pm 1}\right]$ by\cite[Corollaire 4.1.7]{MR0218363} and \cite[Proposition 2.1 a)]{MR0215854}. With these in mind, one can prove in a similar way to Theorem \ref{thm:G/K} that $\bfG'_t$ is canonically isomorphic to the adjoint group of $\bfG_t$. We identify the total flag scheme and the scheme of parabolic types of $\bfG_t$ with those of $\bfG'_t$ respectively. Then imbed $\cP_{\bfG_t}$ into $\Gr(\bfg'_t)$. Finally, take the schematic closure of $\cP_{\bfG_t,x}$ in $\Gr(\bfg')$. One can see that similar results to the below hold under this version of the definition (the statements need minor modifications). We also note that there is no canonical map $\Gr(\bfg')\to\Gr(\bfg)$ even if $G$ is of type (RA) since $\fg'$ is not a quotient of $\fg$ in general. For example, think of $G=\SL_{2n+1}$ with $n\geq 1$, the trivial involution, and $k=\bZ\left[1/2\right]$. Therefore it is difficult to compare these two definitions in general. On the other hand, the definitions coincide if $G$ is of type (RA) and $Z$ is smooth by \cite[Proposition 8.17 (2)]{MR4225278} and Proposition \ref{prop:detect_image} since $\fg'\cong\fg/\fz$ in this case.
	\end{rem}

	\begin{rem}
		One can see that $\cP_{\bfG_t,x}$ is scheme-theoretically dense in $\cP_{\bfG,x}$ by \cite[Lemma 01RG]{stacks}.
	\end{rem}
	
	\begin{ex}\label{ex:Q}
		Let $Q$ be a $\theta$-stable parabolic subgroup of $G$, and $x$ be the attached parabolic type of $\bfG_t$. In this case, we write $\cP_{\bfG,x}=\overline{\bfG_t/\bfQ_t}$. We note that $\cP_{\bfG_t,x}$ is naturally identified with $\bfG_t/\bfQ_t$ by Example \ref{ex:theta_descent_Q} and \cite[Corollaire 3.6]{MR0218364}.
	\end{ex}

	\begin{ex}\label{ex:SO(2,1)}
		Put $G=\SL_2$ and $\theta=((-)^T)^{-1}$. Regard
		\[\bfg=\fk\otimes_k k\left[t\right]\oplus \fp\otimes_k k\left[t\right]\sqrt{t}\subset\fg\otimes_k k\left[\sqrt{t}\right]\]
		as in the proof of Proposition \ref{prop:liealg}. Let $\xi_i$ ($i\in\{1,2,3\}$) be as in Example \ref{ex:semisimpleorbit}. We observe that $\xi_1$ and $\xi_2$ are (resp.~$\xi_3$ is) a basis of $\fp$ (resp.~$\fk$) under the standard identification $\fg\cong\fsl_2$. Therefore $\sqrt{t}\xi_1$, $\sqrt{t}\xi_2$, $\xi_3$ form a free basis of $\bfg$ as a $k\left[t\right]$-module. Hence the maps
		\[(\Gr_2 \bfg)(R)\ni M\mapsto (\bfg\otimes_{k\left[t\right]} R/M,\sqrt{t}\xi_1,\sqrt{t}\xi_2,\xi_3)\in\bP^2_{k\left[t\right]}(R)\]
		for commutative $k\left[t\right]$-algebras $R$ determine an isomorphism $\Gr_2 \bfg\cong\bP^2_{k\left[t\right]}$. Here we would like to prove that the closed immersion
		\[\cB_{\bfG}\hookrightarrow\Gr_2\bfg\cong\bP^2_{k\left[t\right]}\]
		is an isomorphism onto the closed subscheme
		\[\bfB\coloneqq \Proj k\left[t,x,y,z\right]/(x^2+y^2-tz^2)
		\subset \Proj k\left[t,x,y,z\right]=\bP^2_{k\left[t\right]}.
		\]
		
		Following the definition of $\cB_{\bfG}$, we start with proving that the map
		\[\iota_{\emptyset}:\cB_{\bfG_t}\to \Gr_2\bfg_t\]
		is an isomorphism onto $\Proj k\left[t^{\pm 1},x,y,z\right]/(x^2+y^2-tz^2)$.
		In view of the faithfully flat descent, we may work over $k\left[\sqrt{t}^{\pm 1}\right]$ to prove that
		\[\cB_{G}\otimes_k k\left[\sqrt{t}^{\pm 1}\right]
		\cong\cB_{\bfG_t}\otimes_{k\left[t^{\pm 1}\right]}
		k\left[\sqrt{t}^{\pm 1}\right]
		\hookrightarrow
		\bP^2_{k\left[\sqrt{t}^{\pm 1}\right]}\]
		is an isomorphism onto $\Proj k\left[\sqrt{t}^{\pm 1},x,y,z\right]/(x^2+y^2-tz^2)$ (see Lemma \ref{lem:relation} (3) for the first isomorphism). To compute this map, let $B_{\std}\subset\SL_2$ be the Borel subgroup consisting of upper triangular matrices. We know $\cB_G\cong G/B_{\std}$ (\cite[Corollaire 5.8.3 (iii)]{MR0218363}). It is easy to show that the canonical action of $G$ on $\bP^1_k$ at $B_{\std}$ gives rise to an isomorphism $G/B_{\std}\cong\bP^1_k$ as the orbit map. Write
		\[\bP^1_{k\left[\sqrt{t}^{\pm 1}\right]}=\Proj k\left[\sqrt{t}^{\pm 1},v_1,v_2\right].\]
		Then unwinding the definitions, we can see that the total map
		\begin{equation}
			\bP^1_{k\left[\sqrt{t}^{\pm 1}\right]}\cong
			\cB_{G}\otimes_k k\left[\sqrt{t}^{\pm 1}\right]
			\to
			\bP^2_{k\left[\sqrt{t}^{\pm 1}\right]}\label{eq:etale_localization}
		\end{equation}
		is given by
		\[(\cO_{\Proj k\left[\sqrt{t}^{\pm 1},v_1,v_2\right]}(2),\sqrt{t}(v^2_1-v^2_2),
		-2\sqrt{t} v_1 v_2,v^2_1+v^2_2)\]
		in terms of the moduli description in \cite[Corollary 13.33]{MR4225278}
		(compute the map on the principal affine open subspaces).
		We claim that it is an isomorphism onto the closed subscheme
		$\Proj k\left[\sqrt{t}^{\pm 1},x,y,z\right]/(x^2+y^2-tz^2)$. In fact, this map factors through this closed subscheme by \cite[Lemma 01NA]{stacks}. To construct the inverse of the resulting map
		\begin{equation}
			\bP^1_{k\left[\sqrt{t}^{\pm 1}\right]}\to \Proj k\left[\sqrt{t}^{\pm 1},x,y,z\right]/(x^2+y^2-tz^2),\label{eq:pluecker}
		\end{equation}
		observe that the morphisms
		\[\Spec k\left[\sqrt{t}^{\pm 1},u_i\right]\to \Proj k\left[\sqrt{t}^{\pm 1},x,y,z\right]/(x^2+y^2-tz^2)\]
		($i\in\{1,2\}$) defined by
		\begin{flalign*}
			&\left(k\left[\sqrt{t}^{\pm 1},u_1\right],
			\sqrt{t}(u^2_1-1),-2\sqrt{t}u_1,u^2_1+1\right),\\
			&\left(k\left[\sqrt{t}^{\pm 1},u_2\right],
			\sqrt{t}(1-u^2_2),-2\sqrt{t}u_2,1+u^2_2\right)
		\end{flalign*}
		form a Zariski covering of $\Proj k\left[\sqrt{t}^{\pm 1},x,y,z\right]/(x^2+y^2-tz^2)$. Their fiber product is canonically isomorphic to 
		\[\Spec k\left[\sqrt{t}^{\pm 1},u^{\pm 1}_1\right]
		\cong \left[\sqrt{t}^{\pm 1},u^{\pm 1}_2\right]\]
		with $u_1\mapsto u^{-1}_2$. We define maps to $\bP^1_{k\left[\sqrt{t}^{\pm 1}\right]}$ on these loci by
		\[\begin{array}{cc}
			\left(k\left[\sqrt{t}^{\pm 1},u_1\right],
			u_1,1\right),
			&\left(k\left[\sqrt{t}^{\pm 1},u_2\right],
			1,u_2\right).
		\end{array}\]
		One can easily check that they glue up to the inverse map of \eqref{eq:pluecker}.
		
		We complete the isomorphism $\cB_{\bfG}\cong\bfB$ by computing the scheme theoretic closure of $\cB_{\bfG_t}$ locally on the principal open subschemes attached to $x,y,z$ (recall that the scheme theoretic image is local in the target by definition or Proposition \ref{prop:bc_image}). In fact, one can prove that the intersections of the ideals of
		\[\begin{array}{ccc}
			k\left[t^{\pm 1},y,z\right],&k\left[t^{\pm 1},x,z\right],&k\left[t^{\pm 1},x,y\right]
		\end{array}\]
		generated by the dehomogenized polynomials
		\begin{equation}
			1+y^2-tz^2,~x^2+1-tz^2,~x^2+y^2-t\label{eq:generators}
		\end{equation}
		with
		\begin{equation}
			k\left[t,y,z\right],~k\left[t^{\pm 1},x,z\right],~k\left[t^{\pm 1},x,y\right]\label{eq:poly_rings}
		\end{equation}
		are the ideals of \eqref{eq:poly_rings} generated by \eqref{eq:generators} respectively by using the regularity of the elements
		\[\begin{array}{ccc}
			1+y^2\in k\left[y,z\right],&x^2+1\in k\left[x,z\right],&x^2+y^2\in k\left[x,y\right].
		\end{array}\]
		
		As a consequence, one finds that the flag scheme of the current $\bfG$ is not smooth at $t=0$. In fact, we may add $\sqrt{-1}$ to the base ring $k$. Then the fiber at $t=0$ is not smooth at the intersection of the two projective lines (recall the description of the fiber of $\bfB$ at $t=0$ in Example \ref{ex:sl_2}).
		
		Finally, we remark that similar arguments prove $\cB_{\bfG}\cong\bfB$ for \[G=\PGL_2\cong\SO(2,1).\]
	\end{ex}
	
	\begin{rem}\label{rem:relation_with_gluing}
		Put $G=\SL_2$. Assume that $k$ is a commutative $\bZ\left[1/2,\sqrt{-1}\right]$-algebra. Then the base change of the maps $\Spec k\left[t,v_i\right]\to\bfB$ for $i\in\{1,2\}$ in Example \ref{ex:sl_2} to $k\left[\sqrt{t}^{\pm 1}\right]$ are recovered from the map \eqref{eq:etale_localization} by restricting it to the affine open subspaces defined by
		\begin{flalign*}
			\left(k\left[\sqrt{t}^{\pm 1},v_1\right],1+\sqrt{-1}\sqrt{t}v_1,
			\sqrt{-1}+\sqrt{t}v_1\right),\\
			\left(k\left[\sqrt{t}^{\pm 1},v_2\right],\sqrt{t}v_2+\sqrt{-1},
			\sqrt{t}\sqrt{-1}v_2+1\right).
		\end{flalign*}
		For example, if $i=1$, it is deduced as
		\begin{flalign*}
			&\left(k\left[\sqrt{t}^{\pm 1},v_1\right],1+\sqrt{-1}\sqrt{t}v_1,
			\sqrt{-1}+\sqrt{t}v_1\right)\\
			&\mapsto (k\left[\sqrt{t}^{\pm 1},v_1\right],\sqrt{t}((1+\sqrt{-1}\sqrt{t}v_1)^2-(\sqrt{-1}+\sqrt{t}v_1)^2),\\
			&-2\sqrt{t}(1+\sqrt{-1}\sqrt{t}v_1)(\sqrt{-1}+\sqrt{t}v_1),
			(1+\sqrt{-1}\sqrt{t}v_1)^2+(\sqrt{-1}+\sqrt{t}v_1)^2)\\
			&=\left(k\left[\sqrt{t}^{\pm 1},v_1\right],	\sqrt{t}(2-2tv^2_1),
			-2\sqrt{t}\sqrt{-1}(1+tv^2_1),
			4\sqrt{t}\sqrt{-1}v_1\right)\\
			&=\left(k\left[\sqrt{t}^{\pm 1},v_1\right],	1-tv^2_1,-\sqrt{-1}(1+tv^2_1),
			2\sqrt{-1}v_1\right).
		\end{flalign*}
		We note that the affine open subspaces above arise from the $\theta$-stable affine subspaces of $\bP^1_k$ and the correspondences $\frac{w_i}{\sqrt{t}}\mapsto v_i$. 
		
		We get the morphisms over $k\left[t^{\pm 1}\right]$ by the Galois descent. Then the maps from the affine lines to $\bfB$ over $k\left[t\right]$ are obtained as their unique extension by \cite[Lemma 01RH]{stacks}.
	\end{rem}
	
	\begin{rem}
		The fiber of the isomorphism $\cB_{\bfG}\cong\bfB$ above at $t=1$ gives rise to an isomorphism
		\[\bP^1_k\cong\cB_G\cong\Proj k\left[x,y,z\right]/(x^2+y^2-z^2),\]
		where the first isomorphism was explained in Example \ref{ex:SO(2,1)}. It is easy to show that the involution on $\bP^1_k$ induced from $\theta$ and these isomorphisms is given by
		\[(L,a,b)\mapsto (L,b,-a),\]
		which appeared in Example \ref{ex:sl_2}. We also note that the involution induced from $\theta$ and the isomorphism
		$\Gr_2\fg\cong \bP^2_k$ determined by the basis $(\xi_i)_{1\leq i\leq 3}$
		is given by
		\[(L,a,b,c)\mapsto (L,a,b,-c)\]
		since $\theta(\xi_i)=-\xi_i$ for $i\in\{1,2\}$ and $\theta(\xi_3)=\xi_3$. We remark if necessary that
		\[(L,-a,-b,c)=(L,a,b,-c)\]
		as points of $\bP^2_k$ by definitions.
	\end{rem}
	
	Let us record two basic observations on $\cP_{\bfG,x}$.
	
	\begin{prop}\label{prop:restriction}
		We have a canonical isomorphism
		$\cP_{\bfG,x}\otimes_{k\left[t\right]} k\left[t^{\pm 1}\right]\cong\cP_{\bfG_t,x}$.
	\end{prop}
	
	\begin{proof}
		Since $\iota_x$ is a closed immersion, the assertion follows from Proposition \ref{prop:bc_image} and Example \ref{ex:closed_immersion}.
	\end{proof}
	
	\begin{prop}\label{prop:action}
		The partial flag scheme $\cP_{\bfG,x}$ is a $\bfG$-invariant closed subscheme of $\Gr(\bfg)$.
	\end{prop}

	\begin{proof}
		Let $j_{\bfG}$ denote the canonical immersion $\bfG_t\hookrightarrow \bfG$. Then we have an affine immersion
		\begin{equation}
			j_{\bfG}\times_{\Spec k\left[t\right]} (j_{\Gr(\bfg)}\circ \iota_x).
			\label{eq:composite_immersion}
		\end{equation}
		In particular, \eqref{eq:composite_immersion} is quasi-compact.
		Indeed, the morphisms $j_{\bfG}$ and $j_{\Gr(\bfg)}$ are affine since they are obtained by base changes of $j_0$. Since $\iota_x$ is a closed immersion, $\iota_x$ is affine.
		
		Consider the commutative diagram
		\[\begin{tikzcd}
			\bfG_t\times_{\Spec k\left[t^{\pm 1}\right]} \cP_{\bfG_t,x}\ar[d]\ar[r, equal]
			&\bfG_t\times_{\Spec k\left[t\right]} \cP_{\bfG_t,x}\ar[r, "\eqref{eq:composite_immersion}"]
			&\bfG\times_{\Spec k\left[t\right]} \Gr(\bfg)\ar[d]\\
			\cP_{\bfG_t,x}\ar[rr, hook, "{j_{\Gr(\bfg)}\circ \iota_x}"]
			&&\Gr(\bfg),
		\end{tikzcd}\]
		where the vertical arrows are the action maps. In view of the functoriality of scheme-theoretic images in \cite[Lemma 01R9]{stacks} and the definition of $\cP_{\bfG,x}$, the proof will be completed by showing that the scheme-theoretic image of \eqref{eq:composite_immersion} is $\bfG\times_{\Spec k\left[t\right]}\cP_{\bfG,x}$.
		
		The affine immersion \eqref{eq:composite_immersion} has a factorization
		\[\begin{split}
			\bfG_t\times_{\Spec k\left[t\right]} \cP_{\bfG_t,x}
			&\xrightarrow{j_{\bfG}\times_{\Spec k\left[t\right]} \cP_{\bfG_t,x}}
			\bfG\times_{\Spec k\left[t\right]} \cP_{\bfG_t,x}\\
			&\xrightarrow{\bfG\times_{\Spec k\left[t\right]} (j_{\Gr(\bfg)}\circ \iota_x)}
			\bfG\times_{\Spec k\left[t\right]} \Gr(\bfg).
		\end{split}\]
		The first map is obtained by the base change of $j_0$. In virtue of Example \ref{ex:open_dense} and Propositions \ref{prop:bc_image}, \ref{prop:detect_image}, we may compute the scheme-theoretic image of \eqref{eq:composite_immersion} as that of $\bfG\times_{\Spec k\left[t\right]} (j_{\Gr(\bfg)}\circ \iota_x)$. The assertion now follows from Proposition \ref{prop:bc_image} and the definition of $\cP_{\bfG,x}$.
	\end{proof}
	
	We next turn into another candidate for a contraction analog of partial flag schemes. Let $Q$ be a $\theta$-stable parabolic subgroup of $G$.

	\begin{prop}\label{prop:point}
		The partial flag scheme $\overline{\bfG_t/\bfQ_t}$ contains $\bfq$ as a $k\left[t\right]$-point.
	\end{prop}
	
	\begin{proof}
		Recall that we have a canonical factorization
		\[\bfG_t/\bfQ_t\overset{j}{\hookrightarrow} \overline{\bfG_t/\bfQ_t}\overset{i}{\hookrightarrow}\Gr(\bfg)\]
		by \cite[Chapter 2, Exercise 3.17 (d)]{MR1917232}. Consider the commutative diagram
		\[\begin{tikzcd}
			\Spec k\left[t^{\pm 1}\right]\ar[rrd, hook, "j_0"']\ar[r]\ar[rrr, bend left, "\bfq_t"]
			&\bfG_t/\bfQ_t\ar[r, hook, "j"]&
			\overline{\bfG_t/\bfQ_t}\ar[r, hook, "i"]
			&\Gr(\bfg)\\
			&&\Spec k\left[t\right],\ar[u, dashed]\ar[ru, "\bfq"']
		\end{tikzcd}\]
		where the left horizontal arrow is given by the base point. Then the section of $\Gr(\bfg)$ attached to $\bfq$ is a closed immersion since $\Gr(\bfg)$ is separated over $k\left[t\right]$. Therefore its restriction to $\Spec k\left[t^{\pm 1}\right]$ is quasi-compact. Moreover, the $k\left[t\right]$-point $\fq$ of $\Gr(\bfg)$ exhibits $\Spec k\left[t\right]$ as the scheme-theoretic closure of $\Spec k\left[t^{\pm 1}\right]$ by Example \ref{ex:open_dense} and Proposition \ref{prop:detect_image}. The dotted arrow now exists by \cite[Chapter 2, Exercise 3.17 (d)]{MR1917232}. This shows the assertion.
	\end{proof}
	
	As we briefly explained in the introduction, we wish to compare the $\bfG^\circ$-orbit attached to $\bfq$ with $\overline{\bfG_t/\bfQ_t}$. Before we state the main result of this section, let us note general results on structures of $\bfG^\circ$ and $\bfQ^\circ$:
	
	\begin{lem}\label{lem:unit_component}
		\begin{enumerate}
			\renewcommand{\labelenumi}{(\arabic{enumi})}
			\item The group scheme $K^\circ$ is reductive.
			\item If $G$ is simply connected, i.e., the geometric fibers of $G$ are so in the classical sense, then $K^\circ =K$.
			\item The closed subgroup scheme $Q\cap K^\circ\subset K^\circ$ is parabolic. In particular, $Q\cap K^\circ$ has connected fibers.
			\item We have $(Q\cap K)^\circ =Q\cap K^\circ$. 
			\item The group schemes $\bfG^\circ$ and $\bfQ^\circ$ are affine.
			\item If $G$ is simply connected then $\bfG^\circ=\bfG$ and $\bfQ^\circ=\bfQ$.
		\end{enumerate} 
	\end{lem}
	
	We remark that the (RA) hypothesis of $G$ is not needed here.
	
	\begin{proof}
		For (1), see \cite[the beginning of Section 1]{MR1066573} and \cite[Proposition 3.1.3]{MR3362641}. Part (2) is a consequence of \cite[Theorem 8.1]{MR0230728}. Part (3) follows from (1), \cite[Proposition 3.1.3]{MR3362641}, and \cite[Propositions 5.2.5, 5.3.1]{hayashijanuszewski}. For (4), we may work fiberwisely to assume $k=F$ is a field by definition of the unit component. Then the assertion follows since $Q\cap K^\circ$ is a connected open subgroup of $Q\cap K$. Part (5) follows from \cite[Proposition 3.1.3]{MR3362641}, (4), and Corollary \ref{cor:affine}. For (6), we may see on the locus of $t\neq 0$ and at $t=0$. The equality on $t\neq 0$ follows from Lemma \ref{lem:relation}. The equality at $t=0$ for $\bfG$ follows from (2). The equality at $t=0$ for $\bfQ$ is verified by (2) and (4).
	\end{proof}
	
	\begin{thm}\label{thm:stabilizer}
		Assume the following conditions:
		\begin{enumerate}
			\renewcommand{\labelenumi}{(\roman{enumi})}
			\item $G$ and $K^\circ$ are of type $\mathrm{(RA)}$.
			\item Every geometric fiber of $Q\cap K^\circ$ admits a maximal torus $T$ such that the differential of each nonzero weight of the geometric fiber of $\fp$ is nonzero.
		\end{enumerate} 
		Then:
		\begin{enumerate}
			\renewcommand{\labelenumi}{(\arabic{enumi})}
			\item We have $N_{\bfG^\circ}(\underline{\bfq})=\bfQ^\circ$. In particular, the $\bfG^\circ$-orbit in $\overline{\bfG_t/\bfQ_t}$ attached to $\bfq\in \overline{\bfG_t/\bfQ_t}(k\left[t\right])$ (recall Proposition \ref{prop:action}) is isomorphic to $\bfG^\circ/\bfQ^\circ$.
			\item The $\bfG^\circ$-orbit of (1) is represented by a smooth quasi-compact separated scheme over $k\left[t\right]$.
		\end{enumerate} 
	\end{thm}
	
	\begin{rem}
		The orbit map $\bfG^\circ/\bfQ^\circ\hookrightarrow\Gr(\bfg)$ only depends on the $K^\circ(k)$-conjugacy class of $Q$. On the other hand, different $K^\circ(k)$-conjugacy classes give different orbits in general. For example, if $k=\bC$, then the two affine lines at $t=0$ in Example \ref{ex:SO(2,1)} arise from the two closed $K^\circ(\bC)$-orbits in $\bP^1_{\bC}(\bC)$.
	\end{rem}
	
	\begin{cor}\label{cor:orbit}
		Assume the following conditions:
		\begin{enumerate}
			\renewcommand{\labelenumi}{(\roman{enumi})}
			\item $G$ is simply connected.
			\item $G$ and $K$ are of type $\mathrm{(RA)}$ (recall that $K$ is reductive under the assumption (i) by Lemma \ref{lem:unit_component} (1) and (2)).
			\item Every geometric fiber of $Q\cap K$ admits a maximal torus $T$ such that the differential of each nonzero weight of the geometric fiber of $\fp$ is nonzero.
		\end{enumerate}
		Then:
		\begin{enumerate}
			\renewcommand{\labelenumi}{(\arabic{enumi})}
			\item The $\bfG$-orbit in $\overline{\bfG_t/\bfQ_t}$ attached to $\bfq\in \overline{\bfG_t/\bfQ_t}(k\left[t\right])$ is isomorphic to $\bfG/\bfQ$.
			\item The $\bfG$-orbit of (1) is represented by a smooth quasi-compact separated scheme over $k\left[t\right]$.
		\end{enumerate}
	\end{cor}
	
	\begin{proof}
		This is immediate from Theorem \ref{thm:stabilizer} and Lemma \ref{lem:unit_component} (6).
	\end{proof}
	
	The key to the proof of Theorem \ref{thm:stabilizer} is to compute the normalizer of $\bfq$ at $t=0$. For this, let us introduce some notations: Suppose that we are given $k$-modules $\fl$ and $\fm_i$ ($i\in \{1,2,3\}$) with $\fm_3\subset\fl$ and together with a $k$-linear map $\left[-,-\right]:\fm_1\otimes_k \fm_2\to\fl$. We note that in latter applications, $\fl$ will be a Lie algebra over $k$, $\fm_i$ will be $k$-submodules for $i\in \{1,2,3\}$, and $\left[-,-\right]$ will be the restriction of the Lie bracket of $\fl$. Let us set
	\begin{equation}
		N_{\fm_1}(\fm_2;\fm_3)\coloneqq
		\{x\in \fm_1:~\left[x,\fm_2\right]\subset \fm_3\}.\label{eq:n}
	\end{equation}
	We remark that $N_{\fm_1}(\fm_2;\fm_3)$ is identified with the fiber product
	\begin{equation}
		\fm_1\times_{\Hom_k(\fm_2,\fl)} \Hom_k(\fm_2,\fm_3)\label{eq:pullback}
	\end{equation}
	of the map $\fm_1\to\Hom_k(\fm_2,\fl)$
	corresponding to $\left[-,-\right]$ and the inclusion map
	\[\Hom_k(\fm_2,\fm_3)\hookrightarrow\Hom_k(\fm_2,\fl).\]
	
	We next introduce a copresheaf analog of \eqref{eq:n}: Let $\fl$ and $\fm_i$ ($i\in \{1,2,3\}$) be as before. Assume:
	\begin{enumerate}
		\renewcommand{\labelenumi}{(\roman{enumi})}
		\item $\fm_2$ is finitely presented as a $k$-module, and
		\item $\fm_3$ is a direct summand of $\fl$.
	\end{enumerate}
	Then we define a copresheaf
	$N_{\underline{\fm}_1}(\underline{\fm}_2,\underline{\fm}_3)$ on the category of commutative $k$-algebras by
	\begin{flalign*}
		&N_{\underline{\fm}_1}(\underline{\fm}_2,\underline{\fm}_3)(R)\\
		&=\{x\in \fm_1\otimes_k R:~
		x\in N_{\fm_1\otimes_k S}
		(\fm_2\otimes_k S,\fm_3\otimes_k S)~\mathrm{for~every~}R\mathrm{-algebra~}S\}\\
		&=N_{\fm_1\otimes_k R}(\fm_2\otimes_k R,\fm_3\otimes_k R),
	\end{flalign*}
	where $x$ is regarded as an element of $\fm_1\otimes_k S$ in the second line through the identification
	$\fm_1\otimes_k S\cong (\fm_1\otimes_k R)\otimes_R S$
	and the unit $\fm_1\otimes_k R\to (\fm_1\otimes_k R)\otimes_R S$. The last equality follows by $S$-linear formations.
	
	\begin{lem}\label{lem:base_change_I}
		\begin{enumerate}
			\renewcommand{\labelenumi}{(\arabic{enumi})}
			\item For any commutative flat $k$-algebra $R$, we have a canonical isomorphism
			\[N_{\fm_1}(\fm_2;\fm_3)\otimes_k R\cong N_{\fm_1\otimes_k R}(\fm_2\otimes_k R,\fm_3\otimes_k R).\] 
			\item Suppose that $k=F$ is a field (of characteristic not two). Then we have \[N_{\underline{\fm}_1}(\underline{\fm}_2;\underline{\fm}_3)
			=\underline{N_{\fm_1}(\fm_2;\fm_3)}.\]
		\end{enumerate}
	\end{lem}
	
	We do not need the hypothesis $1/2\in k$ for the formalism of $N_{\underline{\fm}_1}(\underline{\fm}_2;\underline{\fm}_3)$. The field $F$ in (2) can be of characteristic two if one wants. 
	
	\begin{proof}
		Part (1) follows from \eqref{eq:pullback} through a formal argument: We have
		\begin{flalign*}
			&N_{\fm_1}(\fm_2;\fm_3)\otimes_k R\\
			&\cong (\fm_1\times_{\Hom_k(\fm_2,\fm_3)} \Hom_k(\fm_2,\fl))\otimes_k R\\
			&\cong (\fm_1\otimes_k R)\times_{\Hom_k(\fm_2,\fm_3)\otimes_k R} (\Hom_k(\fm_2,\fl)\otimes_k R)\\
			&\cong (\fm_1\otimes_k R)\times_{\Hom_R(\fm_2\otimes_k R,\fm_3\otimes_k R)} \Hom_R(\fm_2\otimes_k R,\fl\otimes_k R)\\
			&\cong N_{\fm_1\otimes_k R}(\fm_2\otimes_k R,\fm_3\otimes_k R)
		\end{flalign*}
		for any commutative flat $k$-algebra $R$. In fact, the second isomorphism follows since flat base changes respect fiber products. The third isomorphism is verified by the fact that $\fm_2$ is finitely presented as a $k$-module. We used the assumption (ii) on $\fm_3$ in the last isomorphism to apply \eqref{eq:pullback} to $N_{\fm_1\otimes_k R}(\fm_2\otimes_k R,\fm_3\otimes_k R)$. Part (2) is immediate from (1).
	\end{proof}

	We next see that we can apply these notations to our setting:
	
	\begin{lem}\label{lem:direct_summand}
		The $k$-submodules $\fk$, $\fp$, $\fq$, $\fp\cap\fq$, and $\fk\cap\fq$ are direct summands of $\fg$.
	\end{lem}
	
	Although we should have already known that $\fq$ is a direct summand of $\fg$ on the course of defining the map \eqref{eq:iota}, we show this here for convenience to the reader. In particular, we do not need the (RA) hypothesis here.
	
	\begin{proof}
		Notice that $\fg/\fq$ is finitely generated and projective as a $k$-module. In fact, this statement is local in the Zariski topology of $\Spec k$. This therefore follows from \cite[Corollaire 4.11.8]{MR0212023} (recall that $\fg$ is finitely generated and projective). In particular, $\fq$ is a direct summand of $\fg$.
		
		The assertions for $\fk$ and $\fp$ follow from the decomposition $\fg=\fk\oplus\fp$.
		Since $\fq$ is $\theta$-stable, it restricts to the decomposition
		\begin{equation}
			\fq=(\fq\cap\fk)\oplus(\fq\cap\fp).\label{eq:q_dec}
		\end{equation}
		Therefore the assertions for $\fp\cap\fq$ and $\fk\cap\fq$ are verified by combining that for $\fq$ with \eqref{eq:q_dec}.
	\end{proof}
	
	\begin{lem}\label{lem:base_change_II}
		\begin{enumerate}
			\renewcommand{\labelenumi}{(\arabic{enumi})}
			\item The definitions of
			\[\begin{array}{cccc}
				\fk,&\fp,&\fq\cap \fk=\{x\in\fq:~\theta(x)=x\},
				&\fq\cap \fp=\{x\in\fq:~\theta(x)=-x\}
			\end{array}\]
			commute with arbitrary base changes, i.e., for any commutative $k$-algebra $R$, we have a canonical isomorphism
			\[\fk\otimes_k R\cong\{x\in\fg\otimes_k R:~(\theta\otimes_k R)(x)=x\}\]
			\[\fp\otimes_k R\cong\{x\in\fg\otimes_k R:~(\theta\otimes_k R)(x)=-x\}\]
			\[(\fq\cap \fk)\otimes_k R\cong\{x\in\fq\otimes_k R:~(\theta\otimes_k R)(x)=x\}.\]
			\[(\fq\cap \fp)\otimes_k R\cong\{x\in\fq\otimes_k R:~(\theta\otimes_k R)(x)=-x\}.\]
			\item The intersections $\fk\cap \fq$ and $\fp\cap\fq$ commute with arbitrary base changes. That is, for any $k$-algebra $R$, we have canonical isomorphisms
			\[\begin{array}{cc}
				(\fk\cap \fq)\otimes_k R\cong (\fk\otimes_k R)\cap (\fq\otimes_k R),&
				(\fp\cap \fq)\otimes_k R\cong (\fp\otimes_k R)\cap (\fq\otimes_k R).
			\end{array}\]
		\end{enumerate}
	\end{lem}

	\begin{proof}
		Since $Q$ is $\theta$-stable, so is $\fq$ in $\fg$. Moreover,
		\[\begin{array}{cc}
			\fg=\fk\oplus \fp,&\fq=(\fk\cap \fq)\oplus (\fp\cap\fq)
		\end{array}\]
		exhibit the eigenspace decompositions of $\fg$ and $\fq$ respectively for the involution $\theta$. Since we are working over $\bZ\left[1/2\right]$-algebras, these decompositions commute with arbitrary base changes. This shows (1). Part (2) is a formal consequence of (1). 	
	\end{proof}
	
	To compute the normalizer, let us note generalities on algebraic groups:
	
	\begin{lem}\label{lem:alg_grp}
		Assume that $k=F$ be an algebraically closed field. 
		\begin{enumerate}
			\renewcommand{\labelenumi}{(\arabic{enumi})}
			\item Any maximal torus $H$ in $Q$ is a maximal torus in $G$.
			\item Let $T$ be a maximal torus of $Q\cap K^\circ$. Then the centralizer $Z_G(T)$ of $T$ in $G$ is contained in $Q$.
		\end{enumerate}
	\end{lem}
	
	We do not need any condition on the characteristic of $F$ for (1). We only need to assume the characteristic of $F$ not to be two for (2).
	
	\begin{proof}
		Choose a Borel subgroup $B$ of $G$ contained in $Q$ (\cite[Chapter IV, Corollary of Section 11.2]{MR1102012}). Take any maximal torus $H'$ of $B$. Then $H'$ is a maximal torus of $G$ and therefore of $Q$ by \cite[Proposition 4.2]{MR0219539}. Hence $H$ in (1) and $H'$ are $Q(F)$-conjugate to each other by \cite[Theorem 1.1.19 2]{MR3362641}. Since $H'$ is a maximal torus in $G$, so is $H$.
		
		For (2), let $H$ be any maximal torus in $Q$ containing $T$. Then $H$ is a maximal torus of $G$ by (1). Notice also that $T$ is a maximal torus in $K^\circ$ by Lemma \ref{lem:unit_component} and (1). Hence $Z_G(T)$ is the unique maximal torus of $G$ containing $T$ (\cite[Theorem 3.1.3]{MR4627704}). The uniqueness implies $Z_G(T)=H\subset Q$ as desired.
	\end{proof}

	\begin{proof}[Proof of Theorem \ref{thm:stabilizer}]
		For (1), we may only see the assertions at $t=0$ by a similar argument to Theorem \ref{thm:G/K} (use Proposition \ref{prop:restriction} and Example \ref{ex:Q}). Recall that the fiber of $\bfQ^\circ$ at $t=0$ is $(Q\cap K^\circ)\ltimes \underline{\fp\cap\fq}$ by Lemma \ref{lem:unit_component}. We wish to prove that $N_{K^\circ \ltimes\underline{\fp}}(\bfq_0)=(Q\cap K^\circ)\ltimes \underline{\fp\cap\fq}$. Since the Lie algebra of the fiber of $\bfQ^\circ$ at $t=0$ is $\bfq_0$, $N_{K^\circ \ltimes\underline{\fp}}(\bfq_0)$ contains $(Q\cap K^\circ)\ltimes \underline{\fp\cap\fq}$. We wish to prove the converse containment.
		
		We note that the adjoint representation of $K^\circ\ltimes\fp$ on the Lie algebra $\fk\ltimes\fp$ of $K^\circ\ltimes\fp$ is expressed as
		\begin{equation}
			\Ad((g,x))(y,z)=(\Ad(g)y,\Ad(g)(\left[x,y\right]+z)),\label{eq:expression}
		\end{equation}
		where $(g,x)\in (K^\circ\ltimes\fp)(R)$ and $(y,z)\in (\fk\ltimes\fp)\otimes_k R$ with $R$ running through all commutative $k$-algebras. Let $R$ be a commutative $k$-algebra, and $(g,x)$ be an $R$-point of $K^\circ\ltimes \fp$. If $(g,x)$ normalizes $((\fq\cap \fk)\ltimes (\fp\cap\fq))\otimes_k R$ then $g \in (K^\circ\cap Q)(R)$ by \eqref{eq:expression} and Lemmas \ref{lem:unit_component} (3), \ref{lem:normalizer_q}. Henceforth we may assume $g \in (K^\circ\cap Q)(R)$. Then $(g,x)$ normalizes $((\fq\cap \fk)\ltimes (\fp\cap\fq))\otimes_k R$ if and only if $x$ belongs to $N_{\underline{\fp}}(\underline{\fk\cap\fq};\underline{\fp\cap\fq})(R)$.
		
		The proof of (1) will be completed by showing $N_{\underline{\fp}}(\underline{\fk\cap\fq};\underline{\fp\cap\fq})=\underline{\fp\cap\fq}$. It is evident that the left hand side contains the right hand side. To see the converse, observe that the left hand side is equal to $N_{\underline{\fp}}(\underline{\fk\cap\fq};\underline{\fq})$
		by Lemma \ref{lem:base_change_II} (2). A similar argument to \cite[Rappel 5.3.0]{MR0218363} shows that $N_{\underline{\fp}}(\underline{\fk\cap\fq};\underline{\fq})$ is represented by an affine $k$-scheme of finite presentation. The right hand side $\underline{\fp\cap\fq}$ is represented by a smooth affine $k$-scheme by Lemma \ref{lem:direct_summand}. We may therefore pass to the geometric fibers by \cite[Corollaire (17.9.5)]{MR0238860} to assume $k=F$ to be an algebraically closed field. In view of Lemmas \ref{lem:base_change_II} (2) and \ref{lem:base_change_I} (2), the assertion is reduced to the containment
		$N_{\fp}(\fk\cap\fq;\fq)\subset\fp\cap\fq$.
		
		Let $x\in N_{\fp}(\fk\cap\fq;\fq)\subset\fp$. We wish to prove $x\in\fq$. Take a maximal torus $T\subset K^\circ$ in (ii). Since $N_{\fp}(\fk\cap\fq;\fq)$ is a $T$-submodule of $\fp$, we may assume that $x$ has weight $\alpha$. Suppose that $\alpha$ is nonzero. Then one can find an element $h\in\ft$ such that $\left[h,x\right]=x$ by the hypothesis of $T$. Since $h\in\fk\cap\fq$, we get $x\in\fq$ (recall the definition of $N_{\fp}(\fk\cap\fq;\fq)$). Suppose $\alpha$ to be trivial. Let $H$ be the centralizer of $T$ in $G$. Observe that $\fh$ is the $T$-fixed point subalgebra of $\fg$, i.e., the 0-weight space of $\fg$ with respect to the adjoint action of $T$ by \cite[Th\'eor\`eme 5.2.3 (ii), Section 3.5.1, Remarque 3.6.1]{MR0212023}. We also have $\fh\subset\fq$ by Lemma \ref{lem:alg_grp} (2). We now obtain $x\in\fh\subset \fq$ as desired.
		
		Part (2) follows from Theorem \ref{thm:smooth}, \cite[Th\'eor\`eme 10.1.2]{MR0257095}, and Lemma \ref{lem:smooth}.
	\end{proof}

	\appendix
	\section{Scheme-theoretic image}\label{sec:sch_img}
	
	Following \cite[Chapter 2, Exercise 3.17]{MR1917232} and \cite[Sections 01R5, 01RA, 01U2]{stacks}, we collect a few basic facts on scheme-theoretic image which we use in this paper.
	
	\begin{defn}[{\cite[Chapter 2, Exercise 3.17]{MR1917232}}]\label{defn:img}
		Let $f:X\to Y$ be a quasi-compact morphism of schemes. Then we write $\Image f\coloneqq\Spec \cO_Y/\Ker f^\sharp$, and call it the scheme-theoretic image of $f$.
	\end{defn}
	
	\begin{defn}[{\cite[Definition 01RB, Lemmas 01RD, 01RG]{stacks}}]\label{defn:closure}
		Let $i:Y\hookrightarrow X$ be a quasi-compact (not necessarily open) immersion. Then we call $\Image i$ the scheme-theoretic closure of $Y$ in $X$. We say $Y$ is scheme-theoretically dense in $X$ if $\Image i=X$, in which case $i$ is automatically an open immersion.
	\end{defn}
	
	The following evident fact will be used repeatedly in Section \ref{sec:flag}:
	
	\begin{ex}[cf.~{\cite[Example 056A]{stacks}}]\label{ex:open_dense}
		Let $k$ be a commutative ring. Then $\Spec k\left[t^{\pm 1}\right]$ is scheme-theoretically dense in $\Spec k\left[t\right]$.
	\end{ex}
	
	\begin{ex}[{\cite[Chapter 2, Proposition 2.24]{MR1917232}}]\label{ex:closed_immersion}
		Let $i:Y\hookrightarrow X$ be a closed immersion. Then we have a canonical isomorphism $Y\cong\Image i$. We will identify $\Image i$ with $Y$ to write $Y=\Image i$.
	\end{ex}
	
	The scheme-theoretic image commutes with formation of flat base changes:
	
	\begin{prop}[{\cite[Lemma 081I]{stacks}}]\label{prop:bc_image}
		Let $f:X\to Y$ be a quasi-compact morphism of schemes, and $g:Y'\to Y$ be a flat morphism of schemes. Then the scheme-theoretic image of the projection $X\times_Y Y'\to Y'$ is canonically isomorphic to $(\Image f)\times_Y Y'$.
	\end{prop}
	
	Once we find a candidate for the scheme-theoretic image, the following assertion is useful for proving that it is exactly so:
	
	\begin{prop}\label{prop:detect_image}
		Let $X\overset{f}{\to} Y\overset{g}{\to} Z$ be quasi-compact morphisms of schemes. If $\Image f=Y$ then we have $\Image g=\Image (g\circ f)$. In particular, if $\Image f=Y$ and $g$ is a closed immersion then we have $\Image (g\circ f)=Y$.
	\end{prop}
	
	\begin{proof}
		We remark that the equality $\Image f=Y$ holds if and only if $f^\sharp$ is monic. Assume these equivalent conditions. Then its direct image $g_\ast f^\sharp$ is also monic since $g_\ast$ is left exact. Since the structure homomorphism $(g\circ f)^\sharp$ is defined as the composite map
		\[\cO_Z\overset{g^\sharp}{\to} g_\ast\cO_Y\overset{g_\ast f^\sharp}{\hookrightarrow} g_\ast f_\ast\cO_X=(g\circ f)_\ast \cO_X,\]
		the kernels of $g^\sharp$ and $(g\circ f)^\sharp$ coincide. The assertion now follows by definition of the scheme-theoretic image.
	\end{proof}

\end{document}